\documentclass{article}
\usepackage[utf8]{inputenc}

\usepackage{geometry}

\usepackage{float}
\usepackage{caption} 
\usepackage{graphicx} 

\usepackage{amsthm}
\usepackage{amsmath}
\usepackage{amssymb}
\usepackage{mathrsfs}
\usepackage{amsfonts}
\usepackage{hyperref}
\usepackage{lineno}
\usepackage{fancyhdr}
\usepackage{dsfont}
\usepackage{indentfirst}
\usepackage{url}
\usepackage{cleveref}
\usepackage{enumerate}

\usepackage{nomencl}
\makenomenclature

\newtheorem{theorem}{Theorem}[section]
\newtheorem{corollary}[theorem]{Corollary}
\newtheorem{lemma}[theorem]{Lemma}

\newtheorem{claim}{Claim}

\newtheorem{proposition}[theorem]{Proposition}

\newtheorem{observation}[theorem]{Observation}

\theoremstyle{definition}
\newtheorem{remark}[theorem]{Remark}
\newtheorem{definition}[theorem]{Definition}

\newtheorem{problem}[theorem]{Problem}

\usepackage{subfigure}
\usepackage{tikz}
\usetikzlibrary{decorations.markings}
\usetikzlibrary{backgrounds}

\usetikzlibrary{shapes, arrows, calc, arrows.meta, fit, positioning} 
\tikzset{  
	-stealth,auto,node distance =0.8 cm and 1 cm, thin, 
	state/.style ={circle, draw, inner sep=0.2pt}, 
	point/.style = {circle, draw, inner sep=0.18cm, fill, node contents={}},  
	el/.style = {inner sep=2pt, align=right, sloped}  
}
\title{A complete characterization of split digraphs with a strong arc decomposition }

\date{\today}

\author{Jiangdong Ai\thanks{Corresponding author. School of Mathematical Sciences and LPMC, Nankai University, Tianjin 300071, P.R.
China. Email: jd@nankai.edu.cn. Partially supported
by the Fundamental Research Funds for the Central Universities, Nankai University. },~ Fankang He\thanks{School of Mathematical Sciences and LPMC, Nankai University, Tianjin 300071, P.R.
China. Email: hefankang@mail.nankai.edu.cn.},~ Zhaoxiang Li\thanks{School of Mathematical Sciences and LPMC, Nankai University, Tianjin 300071, P.R.
China. Email: ZhaoxiangLi@mail.nankai.edu.cn.},~ Zhongmei Qin\thanks{School of Science, Chang'an University, Xi’an 710064, Shaanxi Province, P.R.China. Email:qinzhm@chd.edu.cn. Partially supported by Natural Science Basic Research Program of Shaanxi (Program No.2024JC-YBMS-041, 2022JM-019) and the Fundamental Research Funds for the Central Universities, Chang'an University.},~ Changxin Wang\thanks{School of Mathematics and Statistics, Northwestern Polytechnical University, Xi’an, Shaanxi 710129, PR China. Email: Simonang@163.com.}}

\begin{document}
	
	\maketitle
 \begin{abstract}
     A \textbf{strong arc decomposition} of a (multi-)digraph $D(V, A)$ is a partition of its arc set $A$ into two disjoint arc sets $A_1$ and $A_2$ such that both of the spanning subdigraphs $D(V, A_1)$ and $D(V, A_2)$ are strong. In this paper, we fully characterize all split digraphs that do not have a strong decomposition. This resolves two problems proposed by Bang-Jensen and Wang and contributes to a series of efforts aimed at addressing this problem for specific graph classes. This work continues the research on semicomplete composition [Bang-Jensen, Gutin and Yeo, J. Graph Theory, 2020]; on locally semicomplete digraphs [Bang-Jensen and Huang, J. Combin. Theory Ser. B, 2010]; on a type of tournaments [Bang-Jensen and Yeo, Combinatorica, 2004]. 
 \end{abstract}
%
%
	
\section{Introduction}
    For a straightforward look at the background of the topic and our results, we move some necessary notation and definitions to Section~\ref{Sec_2} and refer readers to \cite{BG09} for the standard terminology and notation not introduced in this paper.

    A \textbf{strong arc decomposition} of a (multi-)digraph $D(V,A)$ is a partition of its arc set $A$ into two disjoint arc sets $A_1$ and $A_2$ such that both of the spanning subdigraphs $D(V, A_1)$ and $D(V, A_2)$ are strong. Determining whether a (multi-)digraph $D(V, A)$ has a strong arc decomposition attracted a lot of attention. We can easily see that every digraph $D$ with a strong arc decomposition is 2-arc-strong. Then, asking if every 2-arc-strong digraph has a strong arc decomposition is natural. Unfortunately, the following digraphs give a negative answer.

    Let $S_4$ be the digraph depicted in Figure \ref{fig-MD-exceptional}. It is not hard to check that $S_4$ is 2-arc-strong but does not contain a strong arc decomposition. 

 Bang-Jensen and Yeo~\cite{bangCOM24} proved that for a 2-arc-strong semicomplete digraph $D$, $S_4$ is the only exception that does not have a strong arc decomposition. For completeness of the venation, we give the theorem here.
\begin{theorem}\cite{bangCOM24}\label{thm:semi SAD}
	A 2-arc-strong semicomplete digraph $D$ has a strong arc decomposition if and only if $D$ is not isomorphic to the digraph $S_4$ depicted in Figure \ref{fig-MD-exceptional}. Furthermore, a strong arc decomposition of $D$ can be obtained in polynomial time when it exists.		
\end{theorem}
 
 Bang-Jensen, Gutin and Yeo~\cite{bangJGT95} generalized the above theorem to semicomplete multidigraphs with six more exceptions, see Figure~\ref{fig-MD-exceptional}. 
 \begin{theorem}\cite{bangJGT95}\label{thm:multi semi SAD}
	A 2-arc-strong semicomplete multi-digraph $D=(V,A)$ on $n$ vertices has a strong arc decomposition if and only if $D$ is not isomorphic to one of the exceptional digraphs depicted in Figure \ref{fig-MD-exceptional}. Furthermore, a strong arc decomposition of $D$ can be obtained in polynomial time when it exists.	
\end{theorem}
\begin{figure}[H]
	\centering
	\subfigure{\begin{minipage}[t]{0.23\linewidth}
			\centering\begin{tikzpicture}[scale=0.8]
				\filldraw[black](0,0) circle (3pt)node[label=left:$v_1$](v1){};
				\filldraw[black](2,0) circle (3pt)node[label=right:$v_2$](v2){};
				\filldraw[black](0,-2) circle (3pt)node[label=left:$v_3$](v3){};
				\filldraw[black](2,-2) circle (3pt)node[label=right:$v_4$](v4){};
				\foreach \i/\j/\t in {
					v1/v2/0,
					v2/v3/0,
					v3/v4/0,
					v4/v1/0,
					v1/v3/15,
					v2/v4/15,
					v3/v1/15,
					v4/v2/15
				}{\path[draw, line width=0.8] (\i) edge[bend left=\t] (\j);}			
			\end{tikzpicture}\caption*{$S_4$}\end{minipage}}
	\subfigure{\begin{minipage}[t]{0.23\linewidth}
			\centering\begin{tikzpicture}[scale=0.8]
				\filldraw[black](0,0) circle (3pt)node[label=left:$v_1$](v1){};
				\filldraw[black](2,0) circle (3pt)node[label=right:$v_2$](v2){};
				\filldraw[black](0,-2) circle (3pt)node[label=left:$v_3$](v3){};
				\filldraw[black](2,-2) circle (3pt)node[label=right:$v_4$](v4){};
				\foreach \i/\j/\t in {
					v1/v2/0,
					v2/v3/0,
					v3/v4/0,
					v4/v1/0,
					v1/v3/15,
					v2/v4/15,
					v3/v1/10,
					v3/v1/30,
					v4/v2/15
				}{\path[draw, line width=0.8] (\i) edge[bend left=\t] (\j);}	
			\end{tikzpicture}\caption*{$S_{4,1}$}\end{minipage}}
	\subfigure{\begin{minipage}[t]{0.23\linewidth}
			\centering\begin{tikzpicture}[scale=0.8]
				\filldraw[black](0,0) circle (3pt)node[label=left:$v_1$](v1){};
				\filldraw[black](2,0) circle (3pt)node[label=right:$v_2$](v2){};
				\filldraw[black](0,-2) circle (3pt)node[label=left:$v_3$](v3){};
				\filldraw[black](2,-2) circle (3pt)node[label=right:$v_4$](v4){};
				\foreach \i/\j/\t in {
					v1/v2/0,
					v1/v2/20,
					v2/v3/0,
					v3/v4/0,
					v4/v1/0,
					v1/v3/15,
					v2/v4/15,
					v3/v1/15,
					v4/v2/15
				}{\path[draw, line width=0.8] (\i) edge[bend left=\t] (\j);}	
			\end{tikzpicture}\caption*{$S_{4,2}$}\end{minipage}}
	\subfigure{\begin{minipage}[t]{0.23\linewidth}
			\centering\begin{tikzpicture}[scale=0.8]
				\filldraw[black](0,0) circle (3pt)node[label=left:$v_1$](v1){};
				\filldraw[black](2,0) circle (3pt)node[label=right:$v_2$](v2){};
				\filldraw[black](0,-2) circle (3pt)node[label=left:$v_3$](v3){};
				\filldraw[black](2,-2) circle (3pt)node[label=right:$v_4$](v4){};
				\foreach \i/\j/\t in {
					v1/v2/0,
					v2/v3/0,
					v3/v4/0,
					v4/v1/0,
					v1/v3/15,
					v2/v4/10,
					v4/v2/-30,
					v3/v1/10,
					v1/v3/-30,
					v4/v2/15
				}{\path[draw, line width=0.8] (\i) edge[bend left=\t] (\j);}	
			\end{tikzpicture}\caption*{$S_{4,3}$}\end{minipage}}

        \subfigure{\begin{minipage}[t]{0.23\linewidth}
			\centering\begin{tikzpicture}[scale=0.8]
				\filldraw[black](0,0) circle (3pt)node[label=left:$v_1$](v1){};
				\filldraw[black](2,0) circle (3pt)node[label=right:$v_2$](v2){};
				\filldraw[black](0,-2) circle (3pt)node[label=left:$v_3$](v3){};
				\filldraw[black](2,-2) circle (3pt)node[label=right:$v_4$](v4){};
				\foreach \i/\j/\t in {
					v1/v2/0,
                        v1/v2/20,
					v2/v3/0,
					v3/v4/0,
					v4/v1/0,
					v1/v3/15,
					v2/v4/15,
					v3/v1/10,
					v1/v3/-30,
					v4/v2/15
				}{\path[draw, line width=0.8] (\i) edge[bend left=\t] (\j);}	
			\end{tikzpicture}\caption*{$S_{4,4}$}\end{minipage}}
        \subfigure{\begin{minipage}[t]{0.23\linewidth}
			\centering\begin{tikzpicture}[scale=0.8]
				\filldraw[black](0,0) circle (3pt)node[label=left:$v_1$](v1){};
				\filldraw[black](2,0) circle (3pt)node[label=right:$v_2$](v2){};
				\filldraw[black](0,-2) circle (3pt)node[label=left:$v_3$](v3){};
				\filldraw[black](2,-2) circle (3pt)node[label=right:$v_4$](v4){};
				\foreach \i/\j/\t in {
					v1/v2/0,
                        v1/v2/20,
					v2/v3/0,
					v3/v4/0,
					v4/v1/0,
					v1/v3/15,
					v2/v4/10,
					v3/v1/15,
					v4/v2/-30,
					v4/v2/15
				}{\path[draw, line width=0.8] (\i) edge[bend left=\t] (\j);}	
			\end{tikzpicture}\caption*{$S_{4,5}$}\end{minipage}}
        \subfigure{\begin{minipage}[t]{0.23\linewidth}
			\centering\begin{tikzpicture}[scale=0.8]
				\filldraw[black](0,0) circle (3pt)node[label=left:$v_1$](v1){};
				\filldraw[black](2,0) circle (3pt)node[label=right:$v_2$](v2){};
				\filldraw[black](0,-2) circle (3pt)node[label=left:$v_3$](v3){};
				\filldraw[black](2,-2) circle (3pt)node[label=right:$v_4$](v4){};
				\foreach \i/\j/\t in {
					v1/v2/0,
					v2/v3/0,
					v3/v4/0,
					v4/v1/0,
					v1/v3/15,
					v2/v4/10,
					v4/v2/-30,
					v3/v1/10,
					v1/v3/-30,
					v4/v2/15,
                        v1/v2/20
				}{\path[draw, line width=0.8] (\i) edge[bend left=\t] (\j);}	
			\end{tikzpicture}\caption*{$S_{4,6}$}\end{minipage}}

	\caption{2-arc-strong directed multigraphs without strong arc decompositions.}
	\label{fig-MD-exceptional}
\end{figure}
Later, Bang-Jensen and Huang~\cite{bangJCTB102} extended semicomplete digraphs to locally semicomplete digraphs, see the following theorem.
 \begin{theorem}\cite{bangJCTB102}\label{}
     A 2-arc-strong locally semicomplete digraph $D$ has a strong arc decomposition if and only if $D$ is not the square of an even cycle. Every 3-arc-strong locally semicomplete digraph has a strong arc decomposition and such a decomposition can be obtained in polynomial time.
 \end{theorem}
Bang-Jensen, Gutin and Yeo \cite{bangJGT95} considered the strong 
arc decomposition of semicomplete composition and solved it completely.
\begin{theorem}\cite{bangJGT95}
    Let $T$ be a strong semicomplete digraph on $t\geq 2$ vertices and let $G_1,\dots, H_t$ be arbitrary digraphs. Then $Q=T[G_1,\dots, H_t]$ has a strong arc decomposition if and only if $Q$ is 2-arc-strong and is not isomorphic to one of the following four digraphs: $S_4, C_3[\bar{K}_2,\bar{K}_2,\bar{K}_2], C_3[\bar{K}_2,\bar{K}_2, P_2]$ and $C_3[\bar{K}_2,\bar{K}_2,\bar{K}_3]$. In particular, every 3-arc-strong semicomplete composition has a strong arc decomposition.
\end{theorem}
    Recently, Bang-Jensen and Wang \cite{bang2024} have considered the strong arc decomposition of split digraphs, which is another generalization of semicomplete digraphs. Their main result is the following:
    \begin{theorem}\cite{bang2024}
    Let $D=\left(V_1, V_2; A\right)$ be a 2-arc-strong split digraph such that $V_1$ is an independent set and the subdigraph induced by $V_2$ is semicomplete. If every vertex of $V_1$ has both out- and in-degree at least 3 in $D$, then $D$ has a strong arc decomposition.
    \end{theorem}
   They presented an infinite family of split digraphs to demonstrate that being 2-arc-strong is not sufficient to ensure a strong arc decomposition in a split digraph. Additionally, they proposed the following open problems:
\begin{problem}\label{P_1}
    Does all but a finite number of 2-arc-strong semicomplete split digraphs have a strong arc decomposition?
\end{problem}
\begin{problem}\label{P_2}
    Does every 2-arc-strong split digraph with minimum degree at least 5 have a strong arc decomposition?
\end{problem}
   In this paper, we enhance the aforementioned theorem to make it in a consistent framework with semicomplete digraphs, semicomplete multidigraphs, and semicomplete compositions. Our main result is as follows:
\begin{theorem}\label{thm:2as}
       A 2-arc-strong split digraph $D=(V_1,V_2;A)$ has a strong arc decomposition if and only if $D$  is not isomorphic to any of the digraphs illustrated in Lemma~\ref{lem:counter example}, Lemma~\ref{lem:counter example2}, the Appendix, or their arc-reversed versions (reverse all arcs). 
   \end{theorem}
   Our result also addresses and resolves Problem~\ref{P_1} and Problem~\ref{P_2}. Regarding Problem~\ref{P_2}, the answer is correct. In fact, there are infinitely many counterexamples, but the number becomes finite when restricted to semicomplete split digraphs. As for Problem~\ref{P_1}, the answer is negative. Moreover, we provide the unique counterexample $(iv)^*\times (iv)$ which is illustrated in Appendix.

\section{Preliminaries}\label{Sec_2}
Throughout this paper, we primarily follow the standard terminology and notation as in \cite{BG09}. In this section, we first provide some necessary definitions to ensure the paper is self-contained.

A \emph{directed graph} (or just a \emph{digraph}) $D$ consists of a non-empty finite set $V(D)$ of elements called \emph{vertices} and a finite set $A(D)$ of ordered pairs of distinct vertices called \emph{arcs}. If we allow $A(D)$ to be a multiset, i.e., contains multiple copies of the same arc (often, called \emph{multiple} or \emph{parallel arcs}), then $D$ is a \emph{directed multigraph} or {\em multi-digraph}. A (multi-)digraph is {\em semicomplete} if it has no pair of nonadjacent vertices.

For a digraph $D=(V_D,A_D)$ and an arc set $A_1$, we use the symbol $D+A$ to denote the multi-digraph $(V\cup V(A), A_D\cup A_1)$.

Let $D$ be a (multi-)digraph and let $X$ be a subset of $V(D)$. We use $D[ X]$\label{def:D[X]} to denote the  (multi-)digraph induced by $X$. If $X \subseteq V(D)$, the (multi-)digraph $D-X$ is the subdigraph induced by $V(D)-X$, i.e., $D-X = D[ V(D)- X ]$. For a subdigraph $H$ of $D$, we define $D-H=D-V(H)$. We may use $x$ to represent single vertex set $\{x\}$ and write $v\in H$ instead of $v\in V(H)$. A cycle and a path always mean a directed cycle and path. For subsets $X,Y$ of $V(D)$, a path $P$ is an $(X,Y)$-path if it starts at a vertex $x$ in $X$ and ends at a vertex $y$ in $Y$ such that $V(P)\cap(X\cup Y)=\{x,y\}$. For a path $P$, we use $P[x,y]$ to denote the subpath of $P$ from $x$ to $y$.
 	
A digraph $D$ is a {\em split digraph} if $V(D)$ can be partitioned into two sets $V_1$ and $V_2$ such that $V_1$ is an independent set and $V_2$ induces a semicomplete digraph. We will denote a split digraph $D$ by $D=(V_1, V_2; A)$, where the order in the union matters. And we use $A(V_1,V_2)$ to denote the arcs between $V_1$ and $V_2$.
Given a split digraph $D = (V_1, V_2; A)$, we say the vertex partition $V(D) = V_1 \cup V_2$ is a \emph{maximal partition} if there is no vertex partition $V = V'_1 \cup V'_2$ such that $V'_1$ is an independent set, $V'_2$ induces a semicomplete digraph and $V'_2 \supsetneq V_2$. That is to say, $V_2$ is a maximal vertex subset such that the induced digraph of it is semicomplete. 

A digraph $D$ is {\em strong} if for every pair $x,y$ of distinct vertices in $D$, there exists an $(x,y)$-path and a $(y,x)$-path. A digraph $D$ is {\em $k$-arc-strong} if $D-W$ is also strong for any subset $W \subseteq A(D)$ with a size less than $k$.  An arc $e$ in a strong digraph $D$ is called a {\em cut-arc} if $D-e$ is not strong. 
	
A {\em strong component} of a digraph $D$ is a maximal induced subdigraph of $D$ which is strong. For every digraph $D$, we can label its strong components $D_1,\ldots{},D_t$ ($t\geq 1$) such that there is no arc from $D_j$ to $D_i$ unless $j\leq i$. We call such an ordering an {\em acyclic ordering} of the strong components of $D$. For a semicomplete digraph $D$, it is easy to see that the acyclic ordering  $D_1,\ldots{},D_t$ ($t\geq 1$) is unique and we call $D_1$ (resp., $D_t$) the {\em initial} (resp., {\em terminal}) strong component of $D$.

\subsection{Useful tools from known results}
A {\em vertex decomposition} of a digraph $D$ is a partition of its vertex set $V(D)$ into disjoint sets $(U_1,\ldots,U_{\ell})$ where $l\geq 1$. The {\em index} of a vertex $v$ in the decomposition, denoted by $ind(v)$, is the index $i$ such that $v\in U_i$. An arc $xy$ is called a {\em backward arc} if $ind(x)>ind(y)$. A {\em nice decomposition} of a digraph $D$, introduced in \cite{bangJGT102}, is a vertex decomposition such that $D[ U_i]$ is strong for all $i\in[l]$, and the set of cut‐arcs of $D$ is exactly the set of backward arcs. A {\em natural ordering} $(x_1y_1,\dots,x_ry_r)$ of its backward arcs is the ordering of these arcs in decreasing order according to the index of their tails.

Bang-Jensen, Havet and Yeo \cite{bangJGT102} provided the following theorem.

\begin{theorem}\cite{bangJGT102}\label{thm:nice-decom}
    Every strong semicomplete digraph $D$ of order at least 4 admits a unique nice decomposition. Furthermore, the nice decomposition can be constructed in polynomial time.
\end{theorem}

Moreover, the authors of \cite{bangJGT102} gave a rough construction of the nice composition. 

\begin{proposition}\cite{bangJGT102} \label{prop-nicedecom}
	Let $D$ be a strong semicomplete digraph of order at least 4. Suppose that $(U_1,\ldots,U_{\ell})$ is the nice decomposition of $D$ and  $(x_1y_1,\ldots,x_ry_r)$ is the natural ordering of the backward arcs. Then the following statements hold.
		
	(i)  $x_1\in U_{\ell}$ and $y_r\in U_1$;
		
	(ii) $ind(y_{j+1})<ind(y_j)\leq ind(x_{j+1})< ind(x_j)$ for all $j\in[r-1]$ and $ind(y_{j+1})\leq ind(x_{j+2})<ind(y_j)$ for all $j\in[r-2]$.
\end{proposition}


Next, we introduce the concept of {\em splitting-off}~(for convenience, we have employed this term incorrectly in terms of its grammatical function sometimes), a useful operation for obtaining a semicomplete multi-digraph on $V_2$ from a split graph $D=(V_1, V_2; A)$. 

\begin{definition}\label{def-splitting}
	Let $D=(V_1, V_2; A)$ be a split digraph and let $P$ be a path with both end-points in $V_2$. For each vertex $t\in V(P)\cap V_1$, we call splitting-off the pair $(ut, tv)$ at $t$ if we replace $ut, tv \in P$ with a new arc $uv$ (or a multi-arc if $uv$ already exists). The arc $uv$ is called a {\em splitting arc}. For all $t\in V(P)\cap V_1$, if we split off all such pairs at $t$, then we call this operation by splitting-off the path $P$. The reverse operation where we replace a splitting arc with the two original arcs is called  {\em lifting }the arc.
\end{definition}

We can derive various semicomplete multi-digraphs from the splitting-off operation on $V_2$. Theorem \ref{thm:multi semi SAD} ensures the existence of strong arc decompositions of multi-digraphs. Bang-Jensen and Wang \cite{bang2024} gave the following lemma.

\begin{lemma}\label{lem:(D-X)->D}\cite{bang2024}
	Let $D$ be a multi-digraph and $X$ a subset of $V(D)$ such that every vertex of $D-X$ has two in-neighbors and two out-neighbors in $X$. If $X$ has a strong arc decomposition then $D$ has a strong arc decomposition.
\end{lemma}

\subsection{Finding a strong arc decomposition}
We will now prove two important lemmas that will assist us in finding strong arc decompositions of a 2-arc-strong split digraph.

\begin{definition}
    Let $D = (V_1,V_2;A)$ be a split digraph. We say two arc-disjoint strong subdigraphs $D_1$ and $D_2$ constitute a \emph{pending decomposition} of $D$ if, for each $i \in [2]$, we have $V_2 \subseteq V(D_i)$ and for any vertex $t \in V(D_i) \setminus V(D_{3-i})$, $t$ has at least one in-arc and one out-arc in $A(D) \setminus A(D_i)$.
\end{definition}

\begin{lemma}\label{lem:pending decom}
    If a 2-arc-strong split digraph $D = (V_1,V_2;A)$ has a pending decomposition, then $D$ has a strong arc decomposition.
\end{lemma}

\begin{proof}
    Since $D$ is 2-arc-strong, we have $d^+(v)\geq 2$ and $d^-(v)\geq 2$ for any $v\in V(D)$. Let $D_1$ and $D_2$ be a pending decomposition of $D$. We first show that $G=D[V(D_1) \cup V(D_2)]$ has a strong arc decomposition.
    
    For any $t \in V(D_1) \setminus V(D_2)$, $t$ has at least one in-arc and one out-arc, denoted as $t^-t$ and $tt^+$, respectively, in $A(D) \setminus A(D_1)$. Let $A_2$ be the arc set obtained by adding such $t^-t$ and $tt^+$ for each $t \in V(D_1) \setminus V(D_2)$ to $A(D_2)$. Similarly, do the same for all $s \in V(D_2) \setminus V(D_1)$ and Let $A_1$ be the corresponding arc set. Then $A_1 \cap A_2 = \emptyset$ and $G[A_i]$ is strong for $i \in [2]$. We can freely assign the remaining arcs in $A(G) \setminus (A_1 \cup A_2)$ to $A_1$ or $A_2$ and then find the desired strong arc decomposition of $G=D[V(D_1) \cup V(D_2)]$. Moreover, for any $v \in V(D \setminus G) \subseteq V_1$, $v$ has at least two in-neighbors and two out-neighbors in $V(G)$. By Lemma~\ref{lem:(D-X)->D}, $D$ has a strong arc decomposition.
\end{proof}

Since $D=(V_1, V_2; A)$ is 2-arc-strong, there are two arc-disjoint $(X,Y)$-paths $Q_1, Q_2$ for any two disjoint proper subsets $X, Y\subset V_2$. 

We say an arc set $C \subseteq A(V_1,V_2)$ is \emph{feasible} if there is a partition $C = S_1 \cup S_2 \cup ... \cup S_k$, where each $S_i$ has the form $\{v_ix_i, x_iu_i\}$ with $v_i,u_i \in V_2, v_i\neq u_i$ and $x_i \in V_1$ and $x_i\neq x_j$ for $i\neq j$. Let $V_1(C)$ denote the vertices in $V_1$ that have an incident arc in $C$, i.e. $V_1(C)=\{x_1,\dots, x_k\}$; and let $\bar{C}$ be the arc (multi-)set by splitting-off $S_i$ for all $i\in[k]$. Denote by $A'(Q_i)$ the set of arcs between $V_1$ and $V_2$ in $Q_i$, let $A_{Q_1, Q_2}(C) = (A'(Q_1)\cup A'(Q_2)) \setminus C$ and $V_{Q_1,Q_2}(C) = V_1(A_{Q_1, Q_2}(C))$. When $C = \emptyset$, we simply use $A_{Q_1,Q_2}$, $V_{Q_1,Q_2}$ to denote $A_{Q_1,Q_2}(\emptyset)$, $V_{Q_1,Q_2}(\emptyset)$ respectively. We use $V_1(Q_i)$ to denote $V_1\cap V(Q_i)$.
\begin{definition}\label{def:critical}
    Let $D=(V_1, V_2; A)$ be a split digraph and let $X,Y$ be two disjoint proper subsets of $V_2$. We say an arc-disjoint $(X,Y)$-path pair $\{Q_1,Q_2\}$ is \emph{$(X,Y)_{C}$-critical}, in which $C$ is feasible, if $V_{Q_1,Q_2}(C)\cap V_1(C)=\emptyset$ and there is no arc-disjoint $(X,Y)$-path pair $\{Q_3,Q_4\}$ such that $A_{Q_3, Q_4}(C) \subsetneq A_{Q_1, Q_2}(C)$.
\end{definition}

\begin{remark}
    For an $(X,Y)_{C}$-critical path pair $\{Q_1,Q_2\}$, if there exists $v\in V_1(Q_1)\cap V_1(Q_2)$, we have $v\notin V_1(C)$. Consequently, $A'(Q_i)\setminus C$ is feasible for $i=1,2$. Additionally, for any arc-disjoint $(X,Y)$-path pair $\{Q_1,Q_2\}$ such that $V_{Q_1,Q_2}(C)\cap V_1(C)=\emptyset$, we can always find an $(X,Y)_{C}$-critical path pair $\{Q_3,Q_4\}$ such that $A_{Q_3, Q_4}(C) \subseteq A_{Q_1, Q_2}(C)$.
\end{remark}

Let $D=(V_1,V_2;A)$ be a split digraph and $C$ a feasible set. For an $(X,Y)_{C}$-critical path pair $\{Q_1,Q_2\}$, we use $\bar{A}_{Q_1,Q_2}(C)$ to denote a (multi-)set obtained by adding $\overline{A'(Q_1)\setminus C}$ to $\overline{A'(Q_2)\setminus C}$.

\begin{lemma}\label{lem:critical AC to SAD}
    Let $D=(V_1,V_2;A)$ be a split digraph and $C$ a feasible set. If $\{Q_1,Q_2\}$ is an $(X,Y)_{C}$-critical path pair such that $D[V_2]+\bar{A}_{Q_1,Q_2}(C)+\bar{C}$ has a strong arc decomposition, then $D$ has a strong arc decomposition.
\end{lemma}

\begin{proof}
    Suppose that $D[V_2]+\bar{A}_{Q_1,Q_2}(C)+\bar{C}$ has a strong arc decomposition $D_1 = (V_2, A_1)$ and $D_2 = (V_2, A_2)$. Then we lift all splitting arcs in $A_1$ and $A_2$ to obtain the corresponding strong subdigraphs, $D'_1$ and $D'_2$, of $D$. 
    
    If there is some vertex $t\in V_1$ only contained in $D'_i$, we have the following claim.
    
	\begin{claim}
		$d^+_{D'_i}(t)=d^-_{D'_i}(t) = 1$.
	\end{claim}
 
	\begin{proof}
		Since $t$ appears after lifting, we have $d^+_{D'_i}(t)=d^-_{D'_i}(t) \in\{1,2\}$. Suppose that $d^+_{D'_i}(t)=d^-_{D'_i}(t)= 2$, then there are 4 arcs $at,tb,ct, td\in A_{Q_1,Q_2}(C)$ in $D'_i$. Since $D'_i$ and $D'_{3-i}$ are strong, there is an $(X,Y)$-path $P_i$ in $D'_i$ and an $(X,Y)$-path $P_{3-i}$ in $D'_{3-i}$.

        As $P_i$ is a path, there are at least two arcs of $\{at,tb,ct,td\}$ not in $P_i$, we may assume these two arcs are $at,tb$. Then, $P_1$ and $P_2$ are two arc-disjoint $(X,Y)$-paths such that $A_{P_1,P_2}(C)\subsetneq A_{Q_1,Q_2}(C) $, which is a contradiction to the fact that $\{Q_1,Q_2\}$ is $(X,Y)_{C}$-critical. 
	\end{proof}
 
     Notice that for all $t\in V_1(C) \cup V_{Q_1,Q_2}(C)$ and $i\in [2]$, we have $d^+_{D'_i}(t)=d^-_{D'_i}(t)\in\{0,1\}$. This implies that $D'_1$ and $D'_2$ form a pending decomposition of $D$. Therefore, $D$ has a strong arc decomposition by Lemma~\ref{lem:pending decom}.
  
\end{proof}

\begin{remark}\label{rmk:critical is pending}
    The above proof provides a slightly stronger statement: if the decomposition $D_1$ and $D_2$ is a strong arc decomposition of $D[V_2]+\bar{A}_{Q_1,Q_2}(C)+\bar{C}$, then the decomposition $D'_1$ and $D'_2$ is a pending decomposition of $D$, where $D'_1$ $(D'_2)$ is obtained by lifting all splitting arcs in $D_1$ $(D_2)$.
\end{remark}

\begin{lemma}\label{lem:counter example}
Let $D=(V_1,V_2;A)$ be a 2-arc-strong split digraph, if $D$ has a copy of at least one of the following structures, then $D$ has no strong arc decomposition.
	\begin{itemize}
		\item There are $x_1,x_2,u\in V(D)$ such that $N_D^+(u)=\{x_1,x_3\},$ $N_D^-(u)=\{x_1,x_2\},$ $N_D^+(x_1)=\{x_2,u\}, N_D^+(x_2)=\{v,u\},$  where $x_3\in V(D)\backslash\{x_1,x_2,u\}, v\in V(D)\backslash\{x_1,x_2,u\}.$

            \item There are $x_1,x_2,u\in V(D)$ such that $N_D^-(u)=\{x_1,x_3\},$ $N_D^+(u)=\{x_1,x_2\},$ $N_D^-(x_1)=\{x_2,u\}, N_D^-(x_2)=\{v,u\},$  where $x_3\in V(D)\backslash\{x_1,x_2,u\}, v\in V(D)\backslash\{x_1,x_2,u\}.$
	\end{itemize}
\end{lemma}

\begin{remark}
    As we have characterized all the neighbors of $u$, we have $u\in V_1$ and $x_1,x_2,x_3\in V_2$ when $|V_2|\geq 5$ by the structure of the split digraph. And note that $x_3 = v$ is possible. 
\end{remark}

\begin{proof}[Proof of Lemma~\ref{lem:counter example}]
        If the first case occurs, suppose, for the sake of contradiction, that $D$ has a strong arc decomposition $D_1, D_2$. Without loss of generality, assume $x_1x_2\in D_1$ and $x_1u\in D_2$ since $d_D^+(x_1)=2$. As $d_D^-(u)=2$, we have $x_2u\in D_1$. Since $d_D^+(x_2)=2$, we have $x_2v\in D_2$. As $d_D^+(u)=2$, it follows that only one of $ux_1$ or $ux_3$ can be in $D_1$.

        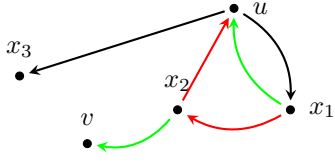
\begin{figure}[H]
    \centering
    \begin{tikzpicture}[scale=0.3]
			
			\filldraw[black](15,0) circle (5pt)node[label=right:$x_1$](x1){};
			\filldraw[black](10,0) circle (5pt)node[label=above:{$x_2$}](x2){};
			\filldraw[black](6,-1.5) circle (5pt)node[label=above:{$v$}](v){};
			\filldraw[black](12.5,4.5) circle (5pt)node[label=right:{$u$}](u){};
			\filldraw[black](3,1.5) circle (5pt)node[label=above:{$x_3$}](x3){};
			
			\path[draw, bend left=30, line width=0.8pt, red] (x1) to (x2);
			\path[draw, bend left=30, line width=0.8pt, green] (x2) to (v);
			\path[draw, bend left=30, line width=0.8pt, green] (x1) to (u);
			\path[draw, line width=0.8pt, red] (x2) to (u);
                \path[draw, bend left=30, line width=0.8pt] (u) to (x1);	
			\path[draw, line width=0.8pt] (u) to (x3);	
                
		\end{tikzpicture}
    \caption{An illustration of Proof of Lemma~\ref{lem:counter example}}
    \label{fig:a8}
\end{figure}
        
        If $ux_1\in D_1$, then $ux_3\in D_2$. There is no arc from $\{x_1,x_2,u\}$ to other vertices in $D_1$, which contradicts the fact that $D_1$ is strong. If $ux_1\in D_2$, then $ux_3\in D_1$. There is no arc from $\{x_1,u\}$ to other vertices in $D_2$, which contradicts the fact that $D_2$ is strong.

        By similar arguments, $D$ also has no strong arc decomposition if the second case occurs.
\end{proof}

\section{Proof of Theorem~\ref{thm:2as} when $|V_2| \geq 5$}
  
Let us consider the induced subdigraph $D[ V_2]$. If $D[ V_2]$ is 2-arc-strong and $|V_2|\geq 5$, then $D[ V_2]$ has a strong arc decomposition by Theorem~\ref{thm:semi SAD}. Consequently, $D$ also has a strong arc decomposition by Lemma~\ref{lem:(D-X)->D}. Therefore, we will focus on the case where $D[V_2]$ is not 2-arc-strong and $|V_2|\geq 5$ in the following.

	
	

\subsection{When $D[V_2]$ is strong but not 2-arc-strong}
Since $D[V_2]$ is strong but not 2-arc-strong and $|V_2|\geq 5$, by Theorem~\ref{thm:nice-decom}, it has a nice decomposition $(U_1,\ldots,U_{\ell})$ and a natural ordering of the backward arcs $(x_1y_1,\ldots,x_ry_r)$ of this decomposition. 

Let $\{Q_1,Q_2\}$ be a $(U_{\ell},U_1)_{B}$-critical path pair in $D$, where $B\subseteq A(V_1,V_2)$ is a feasible arc set. Since we can choose $B=\emptyset$ and $D$ is 2-arc-strong, such $B$ and $\{Q_1,Q_2\}$ exist. If the multi-digraph $D[V_2]+\bar{A}_{Q_1,Q_2}(B)+\bar{B}$ is 2-arc-strong, then it has a strong arc decomposition by Theorem~\ref{thm:multi semi SAD}. Therefore, $D$ has a strong arc decomposition by Lemma~\ref{lem:critical AC to SAD}.

Suppose that $D[V_2]+\bar{A}_{Q_1,Q_2}(B)+\bar{B}$ is not 2-arc-strong, which implies there is at least one cut-arc $e$ in it. Note that $e$ is not a splitting arc because $D[ V_2]$ is strong. Therefore, we have $e\in A(D[V_2])$. Since $D[V_2]$ has a nice decomposition $(U_1,...,U_{\ell})$ and a natural ordering of the backward arcs $(x_1y_1,\ldots,x_ry_r)$, the arc $e$ must be a backward arc. We can assume $e=x_iy_i$ for some $i\in [r]$.

\begin{lemma}\label{lem:cut-arc}
	If $x_iy_i\in D[ V_2]$ is a cut-arc in $\bar{G}:=D[V_2]+\bar{A}_{Q_1,Q_2}(B)+\bar{B}$, then at least one of the following cases holds:
	\begin{enumerate}
		\renewcommand{\theenumi}{($\alpha$\arabic{enumi})}
		\item\label{(alpha1)} $i=2,U_{\ell}=\{x_1\},$ $U_{\ell-1}=\{y_1\}=\{x_2\}$, $d^+_{\bar{G}}(x_2)=1$.
		
		\item\label{(alpha2)} $i=r-1,U_1=\{y_r\},$ $U_{2}=\{x_r\}=\{y_{r-1}\}$, $d^-_{\bar{G}}(y_{r-1})=1$.
	\end{enumerate}
\end{lemma}

\begin{proof}
	Since $Q_1$ and $Q_2$ are arc-disjoint, there is an integer $j\in[2]$ such that $x_iy_i\notin Q_j$. Given that $x_iy_i$ is a cut-arc, there are no other $(x_i,y_i)$-paths other than $x_iy_i$ in $\bar{G}$. Consider the following structure:
	\[x_i\dashrightarrow U_{\ell}\stackrel{Q_j}{\longrightarrow}U_1 \dashrightarrow y_i.\]
Then at least one of the following cases occurs:
	\begin{itemize}
		\item There is no $(x_i,U_{\ell})$-path in $\bar{G}\backslash x_iy_i,$
		
		\item There is no $(U_1,y_i)$-path in $\bar{G}\backslash x_iy_i.$
	\end{itemize}
	
	If $x_iy_i$ is a cut-arc such that there is no $(x_i,U_{\ell})$-path in $\bar{G}\backslash x_iy_i$, then $x_i\notin U_{\ell}$. According to Proposition \ref{prop-nicedecom}, we have $x_1\in U_{\ell}, y_i\notin U_{\ell}$, which means $x_1\neq y_i,x_i$. Additionally, all vertices in $U_\ell$ must dominate $x_i$ as there is no $(x_i,U_{\ell})$-path in $G\backslash x_iy_i$ This implies $U_{\ell}= \{x_1\}$ and $x_i=y_1$ as $x_1y_1$ is the only backward arc from $U_\ell$.
 
    If $\{x_i\}\neq U_{\ell-1}$, then there is a vertex $z\in U_{\ell-1}$ such that $x_i\rightarrow z\rightarrow x_1,$ which contradicts the fact that there is no $(x_i,U_{\ell})$-path in $\bar{G}\backslash x_iy_i$. Hence, $U_{\ell-1}= \{x_i\}$. Furthermore, $x_i=x_2$ and $y_i=y_2$, considering the structure of backward arcs. 
    
    If $d^+_{\bar{G}}(x_i)\geq 2$, then there is a vertex $x_i^+\neq x_i$(possibly, $x_i^+=x_1$) such that  $x_i\rightarrow x_i^+\rightarrow x_1$ in $\bar{G}\backslash x_iy_i,$ which leads to a contradiction. Finally, we have $U_{\ell}=\{x_1\},$ $U_{\ell-1}=\{x_i\}=\{y_1\}=\{x_2\}$ and $d^+_{\bar{G}}(x_i)=1$. A rough construction of $\bar{G}$ can be seen in Figure \ref{fig:cut-arc x_2y_2}.
	
	\begin{figure}[H]
		\centering\begin{tikzpicture}[scale=0.3]
			\foreach \i in {(-15,0),(-10,0),(-5,0),(5,0)}{\draw[ line width=0.8pt] \i ellipse [x radius=50pt, y radius=70pt];}
			\coordinate [label=center:$\ldots$] () at (0,0);
			\coordinate [label=center:$U_1$] () at (-15,4);
			\coordinate [label=center:$U_2$] () at (-10,4);
			\coordinate [label=center:$U_3$] () at (-5,4);
			\coordinate [label=center:$\ldots$] () at (0,4);
			\coordinate [label=center:$U_{\ell-2}$] () at (5,4);
			\coordinate [label=center:$U_{\ell-1}$] () at (10,4);
			\coordinate [label=center:$U_{\ell}$] () at (15,4);
			\draw[line width=1.5pt] (-2,3) -- (2,3); 
			
			\filldraw[black](15,0) circle (5pt)node[label=above:$x_1$](x1){};
			\filldraw[black](10,0) circle (5pt)node[label=above:{$x_2(y_1,x_i)$}](x2){};
			\filldraw[black](6,-1.5) circle (5pt)node[label=below:{~$y_2(y_i)$}](y2){};
			\filldraw[black](4,-1.5) circle (5pt)node[label=above:$x_3$](x3){};
			\filldraw[black](-10,-1.5) circle (5pt)node[label=above:$y_{r-1}$](yr-1){};
			\filldraw[black](-15,-1.5) circle (5pt)node[label=above:$y_r$](yr){};
			\filldraw[black](-5,-1.5) circle (5pt)node[label=above:$x_r$](xr){};
			
			\path[draw, bend left=30, line width=0.8pt] (x1) to (x2);
			\path[draw, bend left=30, line width=0.8pt, red] (x2) to (y2);
			\path[draw, bend left=30, line width=0.8pt] (x3) to (yr-1);
			\path[draw, bend left=30, line width=0.8pt] (xr) to (yr);

			\path[draw, bend right=3, line width=0.8pt, green] (x1) to (-5,1.5);
			\path[draw, bend right=10, line width=0.8pt, green] (-10,1.5) to (-15,1.5);

		\end{tikzpicture}\caption{An illustration of a nice decomposition $(U_1,\ldots, U_{\ell})$ with backward arcs $\{x_1y_1,x_2y_2,\ldots,x_ry_r\}$ where $x_2=y_1$. All arcs between different $U_i$ which are not shown are from left to right. The green arcs are splitting arcs, and $x_2y_2$ is a cut-arc.}
		\label{fig:cut-arc x_2y_2}
	\end{figure}
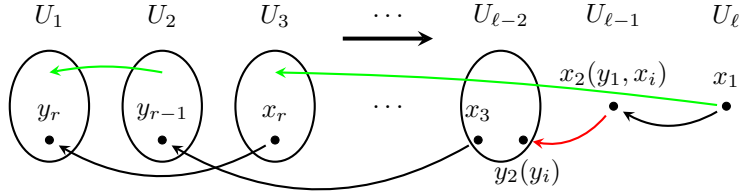
	
	Similarly, if $x_iy_i$ is a cut-arc such that there is no $(U_1,y_i)$-path in $\bar{G}\backslash x_iy_i,$ then we have $U_1=\{y_r\},$ $U_{2}=\{y_{i}\}=\{x_r\}=\{y_{r-1}\}$ and $d^-_{\bar{G}}(y_{r-1})=1$.
\end{proof}

We can see that there can be at most two cut-arcs $x_2y_2,x_{r-1}y_{r-1}$ in $D[V_2]+\bar{A}_{Q_1,Q_2}(B)+\bar{B}$ as established by Lemma~\ref{lem:cut-arc}. The approach for the remaining part of this subsection is as follows: Based on structures shown in $D[V_2]+\bar{A}_{Q_1,Q_2}(B)+\bar{B}$, we first attempt to find another feasible arc set $B_{\text{new}}$ (which maybe empty), and a $(U_{\ell}, U_1)_{B_{\text{new}}}$-critical path pair $\{Q_3, Q_4\}$ such that neither \ref{(alpha1)} nor \ref{(alpha2)} occurs in $D[V_2]+\bar{A}_{Q_3,Q_4}(B_{\text{new}})+\bar{B}_{\text{new}}$, thereby making $D[V_2]+\bar{A}_{Q_3,Q_4}(B_{\text{new}})+\bar{B}_{\text{new}}$ 2-arc-strong. To find such a feasible $B_{\text{new}}$, we will follow an algorithmic procedure. This procedure may involve adding some additional arcs in $D[V_2]$. However, in the final pending decomposition of $D$, these additional arcs will not be used.

\begin{lemma}\label{lem:x1x2 in Q}
	If \ref{(alpha1)} occurs for $D[V_2]+\bar{A}_{Q_1,Q_2}(B)+\bar{B}$, then we can find a $(U_{\ell},U_1)_{B}$-critical path pair $\{Q'_1,Q'_2\}$, such that $A_{Q_1,Q_2}(B)=A_{Q'_1,Q'_2}(B)$ and for some $i\in [2]$ we have $x_1x_2\in A(Q'_i)$.
\end{lemma}

\begin{proof}
	If $x_1x_2\in Q_i$ for some $i\in [2]$, then we are done. Suppose $x_1x_2\notin Q_i$ for $i\in [2]$, we may assume $Q_1=x_1t_1t_1^+Q_1[t_1^+,U_1],Q_2=x_1t_2t_2^+Q_2[t_2^+,U_1]$ where $t_1,t_2\in V_1$. Consider the $(U_{\ell},U_1)$-path $Q_3=x_1x_2Q_3[x_2,U_1]$ in $D[ V_2]$, the existence of such a path is from that $D[V_2]$ is strong. 
 
    If $Q_3$ is arc-disjoint with $Q_1,Q_2$, then $\{Q_1,Q_3\}$ is an arc-disjoint $(U_{\ell},U_1)$-path pair. As $\{Q_1,Q_2\}$ is a $(U_{\ell},U_1)_{B}$-critical path pair, we have $A_{Q_1,Q_3}(B)=A_{Q_1,Q_2}(B)$, and $\{Q_1,Q_3\}$ is $(U_{\ell},U_1)_{B}$-critical path pair as desired.
    
    If $Q_3$ is not arc-disjoint with $Q_1,Q_2$, consider the first common arc $ab$ in $Q_3$, we may assume $ab\in Q_1$. As $x_1x_2\notin Q_1$, we have $ab\neq x_1x_2$, then replace $Q_1$ with $Q_1'=Q_3[x_1,b]Q_1[b,U_1]$, by a similar argument, we have $\{Q_1',Q_2\}$ is the path pair as desired.
\end{proof}

First, we set $B=\emptyset$ and let $\{Q_1,Q_2\}$ be a $(U_{\ell},U_1)_{B}$-critical path pair. By symmetry, we may assume that \ref{(alpha1)} occurs (of course, may both cases occur). By Lemma~\ref{lem:x1x2 in Q}, we may assume $Q_1=x_1tt^+Q_1[t^+,U_1]$ and $Q_2=x_1x_2y_2Q_2[y_2,U_1]$, where $t\in V_1, x_2\notin V(Q_1)$. 



For convenience, we denote by $G$ the digraph $D[V_2]+B+A_{Q_1,Q_2}(B)$ and $\bar{G}$, the digraph $D[V_2]+\bar{B}+\bar{A}_{Q_1,Q_2}(B)$. Similarly, we can define $\bar{G}_{\text{new}}$ after we define $G_\text{new}$, which we will determine later in different cases. Since $D$ is 2-arc-strong, there is a vertex $u \in V_1 \cap N^{+}_D(x_2)$ as $d^{+}_{\bar{G}} (x_2) = 1$ ($y_2$ contributes in $D[V_2]$). Depending on the situations of $u, t$, and other related vertices, we distinguish among eight cases (which we will see later), from \ref{ope:a1} to \ref{ope:a8}. We give the corresponding operations on $D$, $B$ and $\{Q_1,Q_2\}$ to avoid \ref{(alpha1)} among the first seven cases and the eighth case will result in some counterexamples. Regarding the first seven cases, sometimes we call them \emph{procedures} for simplicity and convenience.

Here, we will introduce an operation for rebuilding $B$.
For a vertex $f\in V_1(D)$, let $e$ be an in-neighbor of $f$. If $f\notin V_1(B)$, there is an out-neighbor $f^+$ of $f$ which is not $e$, since $D$ is 2-arc-strong. Add arcs $ef$ and $ff^+$ to $B$. If $f\in V_1(B)$, which means there are arcs $f^-f,ff^+ \in B$, then replace $f^-f$ with $ef$ and replace $ff^+$ with $ff'$, where $f'$ is another out-neighbor of $f$, if $f^+=e$. We use $B\leftarrow\{ef,ff^+\}$ to denote the above operation.


\begin{enumerate}
	\renewcommand{\theenumi}{\textbf{(A\arabic{enumi})}}
	\item\label{ope:a1} When $u \notin V_{Q_1, Q_2}$:



    Let $B_\text{new}:=B\leftarrow\{x_2u,uu^+\}$. Since $V_{Q_1,Q_2}(B_\text{new})\cap V_1(B_\text{new})=\emptyset$, there is a $(U_\ell,U_1)_{B_\text{new}}$-critical path pair $\{Q_3,Q_4\}$ such that $A_{Q_3,Q_4}(B_\text{new})\subseteq A_{Q_1,Q_2}(B_\text{new})$. Let $G_\text{new}$ be $D[V_2]+B_\text{new}+A_{Q_3,Q_4}(B_\text{new})$. In this case, $d^{+}_{\bar{G}_\text{new}}(x_2) \geq 2$ which avoids \ref{(alpha1)}.

    \begin{remark}\label{rmk:end procedure}
         We consider the following two cases:
	
	\begin{enumerate}
		\item If $|A_{Q_3,Q_4}(B_\text{new})|= |A_{Q_1,Q_2}(B)|$, then for each vertex $v \in V_2\setminus\{x_2\}$, we have $d^-_{\Bar{G}}(v) \leq d^-_{\Bar{G}_{\text{new}}}(v)$. This yields that if \ref{(alpha2)} does not occur in $\bar{G}$, then it will not occur in $\bar{G}_\text{new}$ as well since the in-degree of $y_{r-1}$ in $\bar{G}_\text{new}$ will not decrease to 1 if $y_{r-1}\neq x_2$. If $y_{r-1}= x_2$ and $d^-_{\Bar{G}_{\text{new}}}(v)=1$, it is a contradiction to that $|V_2|\geq 4$.
  
		\item If $A_{Q_3,Q_4}(B_\text{new})\subsetneq A_{Q_1,Q_2}(B)$, then $|A_{Q_3,Q_4}(B_\text{new})|<|A_{Q_1,Q_2}(B)|$.
	\end{enumerate}
	
    Thus, we have either $|A_{Q_3,Q_4}(B_\text{new})|<|A_{Q_1,Q_2}(B)|$ or \ref{(alpha2)} does not occur after this procedure if it did not occur before. And $V_1(B)\subseteq V_1(B_\text{new})$. We can verify that these properties hold in the following procedures as well.
    \end{remark}

        \item\label{ope:a2} When $u\in Q_2$:

    There exists $u^-\in V_2 \setminus \{x_2\}$ such that $u^-u\in A(Q_2)$. If $u^-u\in B$, let $B_\text{new}:=B\leftarrow\{x_2u,uu^+\}$. And if $u^-u\notin B$, we define $B_{\text{new}}:=B$. Let $Q_2'= x_1x_2uQ_2[u,U_1]$. Observe that $Q_2'$ is arc-disjoint with $Q_1$, and $V_{Q_1,Q_2'}(B_\text{new})\cap V_1(B_\text{new})=\emptyset$, we can find a $(U_\ell,U_1)_{B_{\text{new}}}$-critical path pair $\{Q_3,Q_4\}$ such that $A_{Q_3,Q_4}(B_\text{new})\subseteq A_{Q_1,Q_2'}(B_\text{new})$. 
      
    Let $G_\text{new} := D[V_2]+B_\text{new}+A_{Q_3,Q_4}(B_\text{new})$. If $x_2u\notin G_\text{new}$, then $B_\text{new}=B$, $A_{Q_3,Q_4}(B)\subsetneq A_{Q_1,Q_2}(B)$, which contradicts to the fact that $\{Q_1,Q_2\}$ is $(U_{\ell}, U_1)_{B}$-critical. So, we have $x_2u\in G_\text{new}$, then we have $d^{+}_{\bar{G}_\text{new}}(x_2) \geq 2$, which avoids \ref{(alpha1)}.
	
	\item\label{ope:a3}When $u\notin Q_2, u\in Q_1$ and $u\neq t$: 
 	
	There exists $u^-\in V_2 \setminus \{x_2\}$ such that $u^-u\in A(Q_1)$. If $u^-u\in B$, let $B_\text{new}:=B\leftarrow\{x_2u,uu^+\}$. and if $u^-u\notin B$, we define $B_{\text{new}}:=B$.
 
    
    Let $Q_2'=x_1 x_2 uQ_1[u,U_1]$, and let
	\begin{equation*}
		Q_1'=\begin{cases}
			x_1tt^+Q_2[t^+,U_1], & \text{if } t^+\in V(Q_2),\\
			x_1tt^+x_2Q_2[x_2,U_1], &\text{else}.\\
		\end{cases}
	\end{equation*}
    Observe that $Q_1'$ is arc-disjoint with $Q_2'$, and $V_{Q_1',Q_2'}(B_\text{new})\cap V_1(B_\text{new})=\emptyset$, so we can find a $(U_\ell,U_1)_{B_{\text{new}}}$-critical path pair $\{Q_3,Q_4\}$ such that $A_{Q_3,Q_4}(B_\text{new})\subseteq A_{Q_1',Q_2'}(B_\text{new})$. Let $G_\text{new} := D[V_2]+B_\text{new}+A_{Q_3,Q_4}(B_\text{new})$, and the same arguments follow as \ref{ope:a2}.

	\item\label{ope:a4}When $u\notin Q_2,u \in Q_1, u=t$ and there exists $s\in (N^+_{D}(x_1) \cap V_1),s\neq t$, and $s\notin V_{Q_1,Q_2}(B)$:

    Firstly, we get $B_\text{new}$ from $B$ by doing the following two operations.
    \begin{itemize}
        \item If $x_1t\in B$, replace it with $x_2t$. And if $x_1t \not\in B$, add $x_2t$ and $tt^+$ to $B$.

        \item $B\leftarrow\{x_1s,ss^+\}$.
        
    \end{itemize}
    
    
    Let $Q_1'=x_1x_2tt^+Q_1[t^+,U_1]$ and let 
    \begin{equation*}
	Q_2'=\begin{cases}
		x_1ss^+Q_2[s^+,U_1], & \text{if } s^+\in V(Q_2),\\
            x_1ss^+, & \text{if } s^+\notin V(Q_2),s\in U_1,\\
		x_1ss^+x_2Q_2[x_2,U_1], &\text{else}.\\
		\end{cases}
    \end{equation*}
    
    Observe that $Q_1'$ is arc-disjoint with $Q_2'$, and $V_{Q_1',Q_2'}(B_\text{new})\cap V_1(B_\text{new})=\emptyset$, so we can find a $(U_\ell,U_1)_{B_{\text{new}}}$-critical path pair $\{Q_3,Q_4\}$ such that $A_{Q_3,Q_4}(B_\text{new})\subseteq A_{Q_1',Q_2'}(B_\text{new})$. Let $G_\text{new} := D[V_2]+B_\text{new}+A_{Q_3,Q_4}(B_\text{new})$. We have $d^{+}_{\bar{G}_\text{new}}(x_2) \geq 2$, which avoids \ref{(alpha1)}, as $x_2t\in B_\text{new}$.


	\item\label{ope:a5}When $u\notin Q_2, u \in Q_1, u=t$  and there exists $s\in (N^+_{D}(x_1) \cap V_1),s\neq t$, and $s\in V_{Q_1,Q_2}(B)$:

    Firstly, we get $B_\text{new}$ from $B$ by doing the following two operations.
    \begin{itemize}
        \item If $x_1t\in B$, replace it with $x_2t$. And if $x_1t \not\in B$, add $x_2t$ and $tt^+$ to $B$.

        \item If $s\in V_1(B)$, which means there exists an arc $s^-s\in B$, then replace it with $x_1s$. And if $s\notin V_1(B)$, then we do nothing.
    \end{itemize}
    
    If $s\in Q_1$, then $Q_1'=x_1sQ_1[s,U_1]$ is arc-disjoint with $Q_2'=Q_2$.
    
    If $s\notin Q_1$, then $s\in Q_2$, $Q_1'=x_1x_2tQ_1[t,U_1]$ is arc-disjoint with $Q_2'=x_1sQ_2[s,U_1]$.
    
    Observe that $V_{Q_1',Q_2'}(B_\text{new})\cap V_1(B_\text{new})=\emptyset$, so we can find a $(U_\ell,U_1)_{B_{\text{new}}}$-critical path pair $\{Q_3,Q_4\}$ such that $A_{Q_3,Q_4}(B_\text{new})\subseteq A_{Q_1',Q_2'}(B_\text{new})$. Let $G_\text{new} := D[V_2]+B_\text{new}+A_{Q_3,Q_4}(B_\text{new})$. We have $d^{+}_{\bar{G}_\text{new}}(x_2) \geq 2$, which avoids \ref{(alpha1)}, as $x_2t\in B_\text{new}$.\\
    
    \textbf{If \ref{ope:a1}-\ref{ope:a5} do not occur for all $u\in V_1 \cap N^{+}_D(x_2)$, then we have $N^+_D(x_1)\cap V_1=N^+_D(x_2)\cap V_1=\{t\}$ with $t\notin V(Q_2)$. Then we do the following operations.}\\
    
	\item\label{ope:a6}When there exists $w\in N^+_D(t)\backslash\{x_1,t^+\}$:

    We additionally add two parallel arcs $x_2x_1$ and another arc $x_1x_2$ to $D[V_2]$. Let $B_\text{new}=B\leftarrow\{x_2t,tt^+\}$, and we can check $\{Q_3=x_1x_2tQ_1[t,U_1], Q_4=Q_2\}$ is a $(U_\ell,U_1)_{B_{\text{new}}}$-critical path pair, and let $G_\text{new} := D[V_2]+B_\text{new}+A_{Q_3,Q_4}(B_\text{new})$. We have $d^{+}_{\bar{G}_\text{new}}(x_2) \geq 2$, which avoids \ref{(alpha1)}.

	\item\label{ope:a7}When $N^+_D(t)=\{x_1,t^+\}$ and there exists $w\in N^-_D(t)\backslash\{x_1,x_2\}$:

    We do the same thing as that of \ref{ope:a6}.


    \item\label{ope:a8} When \ref{ope:a1}-\ref{ope:a7} do not occur for all $u\in V_1 \cap N^{+}_D(x_2)$:
    
    We have $N^+_D(x_1)\cap V_1=N^+_D(x_2)\cap V_1=\{t\}$ with $t\notin Q_2$ and $N^+_D(t)=\{x_1,t^+\}$, $N^-_D(t)=\{x_1,x_2\}$, note that $D$ satisfies the structure in Lemma \ref{lem:counter example}. Therefore, $D$ has no strong arc decomposition.
\end{enumerate}

We may use \ref{ope:a1}*-\ref{ope:a8}* 
to denote the symmetric procedures of \ref{ope:a1}-\ref{ope:a8} when \ref{(alpha2)} occurs, and denote by $t^*, u^*, s^*,w^*$ the symmetric vertices of $t,u,s,w$, respectively if \ref{(alpha2)} occurs. If we enter either \ref{ope:a8} or \ref{ope:a8}*, then $D$ has no strong arc decomposition by Lemma~\ref{lem:counter example}. And if neither \ref{ope:a8} nor \ref{ope:a8}* is entered, we proceed as follows.

If \ref{(alpha1)} occurs in $\bar{G}$, then we enter one of the procedures from \ref{ope:a1} to \ref{ope:a7} to obtain $G_\text{new}$, which avoids \ref{(alpha1)}. If \ref{(alpha2)} then occurs in $\bar{G}_\text{new}$, then we can find new $\{Q_1,Q_2\}$ such that $y_{r-1}y_r\in A(Q_2)$ and $A_{Q_1,Q_2}(B_\text{new})= A_{Q_3,Q_4}(B_\text{new})$, similar to the process described in Lemma~\ref{lem:x1x2 in Q} for its symmetric case. We then update $G_\text{new}$ to $G$, $B_\text{new}$ to $B$. Next, enter one of the procedures from \ref{ope:a1}*-\ref{ope:a7}* to obtain a new $G_\text{new}$. This process is repeated if \ref{(alpha1)} occurs again for this new $G_\text{new}$.

We remark that at the end of each procedure, the critical path pair is called $\{Q_3, Q_4\}$. If we are not done after one procedure, it indicates that we have to enter another one, and in this case, $\{Q_3, Q_4\}$ will be renamed as $\{Q_1, Q_2\}$ to reflect the default settings for the new procedure. The feasible set $B_\text{new}$ will also be subject to similar adjustments.

This process usually takes more steps than it looks like, as a sequence of operations aimed at avoiding \ref{(alpha1)} may inadvertently lead to \ref{(alpha2)}, and vice versa. Although this process will eventually come to an end, it can be quite complex.

\begin{claim}\label{clm:end}
    The process will eventually terminate. And if we enter procedure \ref{ope:a6} or \ref{ope:a7} (\ref{ope:a6}* or \ref{ope:a7}*), then \ref{(alpha1)} (\ref{(alpha2)}) will not occur, regardless of the steps taken in \ref{ope:a1}*-\ref{ope:a7}* (\ref{ope:a1}-\ref{ope:a7}).
\end{claim}

\begin{proof}
    Without loss of generality, we may assume that \ref{(alpha1)} occurs initially and that $\{Q_1,Q_2\}$ is a $(U_{\ell},U_1)_{B}$-critical path pair, with $B=\emptyset$. According to Remark \ref{rmk:end procedure}, either we are done if \ref{(alpha2)} does not occur or $|A_{Q_3,Q_4}(B_\text{new})|<|A_{Q_1,Q_2}(B)|$. Thus, we can  obtain a 2-arc-strong $\bar{G}_\text{new}$ in finite steps since the original $|A_{Q_1,Q_2}|$ is bounded. And if we enter \ref{ope:a6} or \ref{ope:a7} (\ref{ope:a6}* or \ref{ope:a7}*), then \ref{(alpha1)} (\ref{(alpha2)}) will not occur, regardless of the steps taken in \ref{ope:a1}*-\ref{ope:a7}* (\ref{ope:a1}-\ref{ope:a7}) as these procedure do not remove the additional arcs. This implies that $d^{+}_{\bar{G}_\text{new}}(x_2) \geq 2$.
\end{proof}

If we have never entered \ref{ope:a6} , \ref{ope:a7}, \ref{ope:a6}* or \ref{ope:a7}*, then $G_\text{new}$ is a subdigraph of $D$. For $|V_2|\geq 5$, $\bar{G}_\text{new}$ has a strong arc decomposition by Theorem~\ref{thm:multi semi SAD}. Consequently, $D$ has a strong arc decomposition by Lemma~\ref{lem:critical AC to SAD}.

If we have entered \ref{ope:a6} , \ref{ope:a7}, \ref{ope:a6}* or \ref{ope:a7}*, then by the following lemma and Lemma~\ref{lem:pending decom}, $D$ has a strong arc decomposition.

\begin{lemma}\label{lem: gnew has pending}
   If we have entered \ref{ope:a6} , \ref{ope:a7}, \ref{ope:a6}* or \ref{ope:a7}*, then we can find a pending decomposition of $D$ for $|V_2|\geq 4$.
\end{lemma}

\begin{proof}
    We can first assume we have entered \ref{ope:a6} or \ref{ope:a7} by symmetry(of course we can also have entered \ref{ope:a6}* or \ref{ope:a7}*). Since we have added additional arcs in $D[V_2]$, resulting in two parallels arcs $x_1x_2$ and two parallels arcs $x_2x_1$ in $\bar{G}_\text{new}$, so $\bar{G}_\text{new}$ is not isomorphic to the counterexamples in Theorem~\ref{thm:multi semi SAD}. So, we can find a strong arc decomposition $\bar{G}_1$ and $\bar{G}_2$ of $\bar{G}_\text{new}$. By lifting all splitting arcs in $\bar{G}_1$ and $\bar{G}_2$, we obtain $G_1$ and $G_2$, which form a pending decomposition of $D$ with the additional arcs by Remark~\ref{rmk:critical is pending}. The next step is to reallocate the arcs in $G_1$ and $G_2$ to ensure that they remain strong. And we have the following: 

    \begin{enumerate}
        \renewcommand{\theenumi}{\textbf{(\arabic{enumi})}}

        \renewcommand{\thefootnote}{\fnsymbol{footnote}}
        \setcounter{footnote}{3}
        
        \item\label{choose (1)} $x_2y_2\in G_1, x_2t,tt^+ \footnote{In order to match our symbols, if we entered procedure ~\ref{ope:a6}, and we replace $tt^+$ with $tw$ in some procedure later from \ref{ope:a1}*-\ref{ope:a7}*, which implies $w\in \{y_{r-1},y_r\}$, then we exchange the name of $t^+$ and $w$. The same arguments also hold for \ref{ope:a6}*.}\in G_2$ as $x_2y_2,x_2t^+$ are all out-arcs of $\{x_1,x_2\}$ in $\bar{G}_\text{new}$.

        \item\label{choose (2)} Let $G_i$ contain 2-cycle $x_1x_2x_1$ for $i=1,2$. Because the parallel arcs do not contribute to making a graph strong.

        \item\label{choose (3)} For any $v\in V_2\backslash\{x_1,x_2,y_2\}$, $vx_i\in A(G_1)$ and $vx_{3-i}\in A(G_2)$ for some $i\in [2]$ since there are 2-cycle $x_1x_2x_1$ in both $G_1$ and $G_2$.

        \item\label{choose (4)} $y_2x_1\in G_2$, as if not, we can move it from $G_1$ to $G_2$, we can check $G_1$ is still strong by finding a $(y_2,x_2)$-path in $G_1$. Since $G_1$ is strong before moving $y_2x_1$, there is a $(y_2,v)$-path $P$ for any $v\in V_2\backslash\{x_1,x_2,y_2\}$ as $|V_2|\geq 4$, and observe that $y_2x_1\notin P$ as $x_1$  can only reach to $x_2$, and $x_2$ can only reach to $x_1$ and $y_2$. Now, we are done because we have $vx_2\in G_1$ or $vx_1\in G_1$ by~\ref{choose (3)}.
    \end{enumerate}

    \begin{claim}\label{clm:t^+x1 not in G2}
        $G_2$ is still strong after removing $t^+x_1$ and $t^+x_2$(if it exists, as $y_2x_2$ may not in $D$). 
    \end{claim}

    \begin{proof}
        We prove this by finding another $(t^+,\{x_1,x_2\})$-path in $G_2$ without passing through $t^+x_1$ and $t^+x_2$, given that there is a 2-cycle $x_1x_2x_1$ in $G_2$. Since $G_2$ is strong before removing any arcs, there is a $(t^+,v)$-path $P$ for any $v\in V_2\backslash\{x_1,x_2,t^+\}$ as $|V_2|\geq 4$. Observe that $t^+x_1,t^+x_2\notin P$ as $x_1$  can only reach $x_2$, $x_2$ can only reach $x_1$ and $t$, and $t$ can only reach $t^+$. Now, we are done because we have $vx_1\in G_2$ or $vx_2\in G_2$ by~\ref{choose (3)} and~\ref{choose (4)}.
    \end{proof}

We will continue by considering the following two cases.
    
    \noindent\textbf{Case 1:} We do not enter \ref{ope:a6}* and \ref{ope:a7}*.

    \textbf{Case 1.1:} We enter~\ref{ope:a6}. Recall that there exists $w\in N^+_D(t)\backslash\{x_1,t^+\}$. 
    
    1. By Claim~\ref{clm:t^+x1 not in G2}, we move $t^+x_1$ and $t^+x_2$(if it exists) from $G_2$ to $G_1$. Both $G_1$ and $G_2$ remain strong. 
    
    2. We remove the additional arcs $x_1x_2,x_2x_1$ from $G_1$ and the additional arc $x_2x_1$ from $G_2$. 
    
    3. We add arcs $x_1t, tw$, Which are not in $G_\text{new}$, to $G_1$. 
    
    We want to show that through reallocation of arcs in $G_1$ and $G_2$, the resulting two graphs (we continue to refer to as $G_1$ and $G_2$ for convenience) remain strong and can form a pending decomposition.
    In the following, we only move the arcs incident to $x_1$ and $x_2$, in addition, for any vertex $v\in V_2\backslash\{x_1,x_2,y_2\}$, once we move $vx_i$ from $G_j$ to $G_{3-j}$, then we move $vx_{3-i}$ from $G_{3-j}$ to $G_j$ for $i,j=[2]$. After performing this reallocation, it suffices to verify the presence of an $(x_1,x_2)$-path and an $(x_2,x_1)$-path in each $G_i$ for $i=[2]$ to ensure that both $G_1$ and $G_2$ remain strong.
    \begin{itemize}
	\item $t^+=y_2$: As $t^+=y_2$, we have moved $y_2x_1$ to $G_1$. If $w\neq x_2$, we choose $wx_2$ in $G_1$ and $wx_1$ in $G_2$ as $w\neq t^+$. $G_1$ is strong as $x_1twx_2y_2x_1$ is a cycle; Observe that there is a $(y_2,w)$-path $P$ in $G_2$ which do not go through $x_1$ and $x_2$, then $x_1x_2ty_2Pwx_1$ is a cycle, $G_2$ is strong. If $w=x_2$, then there exists $v\in V_2\backslash\{x_1,x_2,y_2\}$ as $|V_2|\geq 4$, and we choose $vx_2$ in $G_1$ and $vx_1$ in $G_2$. By a similar argument, we have that $G_1$ and $G_2$ are strong. 
	
	\item $t^+\neq y_2$: We have moved $t^+x_1$ and $t^+x_2$ to $G_1$. Observe that there is a $(y_2,t^+)$-path $Q$ in $G_1$ which do not go through $x_1$ and $x_2$, then $x_2y_2Qt^+x_1$ is a $(x_2,x_1)$-path in $G_1$. If $w=x_2$, then $x_1tx_2$ is a $(x_2,x_1)$-path in $G_1$, if $w=y_2$, then $x_1ty_2Qt^+x_2$ is a $(x_2,x_1)$-path in $G_1$, if $w\neq x_2,y_2$, then we choose $wx_2$ in $G_1$ and $wx_2 \in G_2$, $x_1twx_2$ is a $(x_2,x_1)$-path in $G_1$. So $G_1$ is strong. As $y_2x_1\in G_2$, Observe that there is a $(t^+,y_2)$-path $P$ in $G_2$ which do not go through $x_1$ and $x_2$, then $x_1x_2tt^+Py_2x_1$ is a cycle, $G_2$ is strong.
    \end{itemize}

    \textbf{Case 1.2:} We enter~\ref{ope:a7}. Recall that $N^+_D(t)=\{x_1,t^+\}$ and there exists a vertex $w\in N^-_D(t)\backslash\{x_1,x_2\}$.
    
    1. We remove the additional arc $x_2x_1$ from $G_1$ and the additional arcs $x_1x_2,x_2x_1$ from $G_2$. 
    
    2. Add arcs $wt,tx_1$ to $G_1$ and arc $x_1t$ to $G_2$, which are not in $G_\text{new}$.
    
    What we will do next is analogous to Case $1.1$.
    
    If $t^+=y_2$, then there exists a vertex $v\in V_2\backslash\{x_1,x_2,y_2\}$ as $|V_2|\geq 4$, and we choose $vx_2$ in $G_2$ and $vx_1$ in $G_1$. Observe that there is a $(y_2,v)$-path $P$ in $G_2$ which does not pass through $x_1$ and $x_2$, then $x_1ty_2Pvx_2$ is an $(x_1,x_2)$-path in $G_2$, $x_2ty_2x_1$ is an $(x_1,x_2)$-path in $G_2$ as $y_2x_1\in G_2$. Thus, $G_2$ is strong. 
    
    If $t^+\neq y_2$, then we choose $t^+x_2$ in $G_2$ and $t^+x_1$ in $G_1$. Observe that there is a $(t^+,y_2)$-path $Q$ in $G_2$ which does not pass through $x_1$ and $x_2$, then $x_1tt^+x_2$ is an $(x_1,x_2)$-path in $G_2$, $x_2tt^+Qy_2x_1$ is a $(x_2,x_1)$-path in $G_2$ as $y_2x_1\in G_2$. So $G_2$ is strong.
    
    As $x_1x_2\in G_1$, we only need to check the presence of an $(x_2,x_1)$-path in $G_1$. If $w=y_2$, then $x_2y_2tx_1$ is the path we need, and if $w\neq x_2$, observe that there is a $(y_2,w)$-path $P$ in $G_1$ which does not pass through $x_1$ and $x_2$, then $x_2y_2Pwtx_1$ is the path we need. Thus, $G_1$ is strong in either case.
    
    Refer to Figure~\ref{fig:case of 67} for an illustration of Case 1.1 and Case 1.2. 

    \begin{figure}[H]
	\centering
	\begin{minipage}[t]{0.45\linewidth}
		\centering\begin{tikzpicture}[scale=0.3]
			
			\filldraw[black](15,0) circle (5pt)node[label=above:$x_1$](x1){};
			\filldraw[black](10,0) circle (5pt)node[label=above:{$x_2$}](x2){};
			\filldraw[black](6,-1.5) circle (5pt)node[label=above:{$y_2$}](y2){};
			\filldraw[black](12.5,4.5) circle (5pt)node[label=above:{$t$}](u){};
			\filldraw[black](3,1.5) circle (5pt)node[label=above:{$t^+$}](u+){};
			\filldraw[black](3,-1.5) circle (5pt)node[label=above:{$w$}](w){};
			
			\path[draw, bend left=30, line width=0.8pt, green] (x1) to (x2);
			\path[draw, bend left=30, line width=0.8pt, red] (x2) to (y2);
			\path[draw, line width=0.8pt, red] (x1) to (u);
			\path[draw, line width=0.8pt, green] (x2) to (u);
			\path[draw, line width=0.8pt, red] (u) to (w);
			\path[draw, line width=0.8pt, green] (u) to (u+);

		\end{tikzpicture}\caption*{Case 1.1}
	\end{minipage}
	\begin{minipage}[t]{0.45\linewidth}
			\centering\begin{tikzpicture}[scale=0.3]
				
				\filldraw[black](15,0) circle (5pt)node[label=above:$x_1$](x1){};
				\filldraw[black](10,0) circle (5pt)node[label=above:{$x_2$}](x2){};
				\filldraw[black](6,-1.5) circle (5pt)node[label=above:{$y_2$}](y2){};
				\filldraw[black](12.5,4.5) circle (5pt)node[label=above:{$t$}](u){};
				\filldraw[black](3,1.5) circle (5pt)node[label=above:{$t^+$}](u+){};
				\filldraw[black](3,-1.5) circle (5pt)node[label=above:{$w$}](w){};
				
				\path[draw, bend left=30, line width=0.8pt, red] (x1) to (x2);
				\path[draw, bend left=30, line width=0.8pt, red] (x2) to (y2);
				\path[draw, line width=0.8pt, green, bend left=30] (x1) to (u);
				\path[draw, line width=0.8pt, red, bend left=30] (u) to (x1);
				\path[draw, line width=0.8pt, green] (x2) to (u);
				\path[draw, line width=0.8pt, red] (w) to (u);
				\path[draw, line width=0.8pt, green] (u) to (u+);	
    
			\end{tikzpicture}\caption*{Case 1.2}
	\end{minipage}
        \hfill
	\caption{An illustration of Case 1.1 and Case 1.2, where red arcs are in $G_1$, green arcs are in $G_2$.}
	\label{fig:case of 67}
    \end{figure}

    \noindent\textbf{Case 2:} We also enter \ref{ope:a6}* or \ref{ope:a7}*.

    \textbf{Case 2.1:} When $t=t^*$, then the procedures we have entered are \ref{ope:a6} and \ref{ope:a6}*, as $x_1t, x_2t, t^*y_{r-1}, t^*y_r\in A(D)$. 

    1. Note that $x_2t^+=t^{*-}y_{r-1}$ in $\bar{G}_\text{new}$, we have a stronger result of \ref{choose (1)}, that is $x_2y_2,x_{r-1},y_{r-1}\in G_1, x_2t,tt^+\in G_2$.

    2. A stronger result of \ref{choose (2)}, we let $G_i$ contain 2-cycle $y_ry_{r-1}y_r$ for $i= 1,2$.

    3. We remove the additional arcs $x_1x_2,x_2x_1,y_ry_{r-1},y_{r-1}y_r$ from $G_1$, and the additional arcs $x_2x_1$, $y_ry_{r-1}$ from $G_1$. And we add arcs $x_1t,ty_r$, which are not in $G_\text{new}$, to $G_1$.

    In the following, we only move the arcs $y_rx_2,y_{r-1}x_1,y_rx_1$ and if we can find cycles containing the vertices $x_1x_2y_{r-1},y_r$ in both $G_1$ and $G_2$, then $G_1$ and $G_2$ are strong.
    
    We move $y_rx_1$ to $G_2$, and move $y_rx_2,y_{r-1}x_1$ to $G_1$. In $G_2$, $x_1x_2ty_{r-1}y_rx_1$ is the cycle we need, so $G_2$ is strong. If $r-1=2$, then $x_2y_{r-1}x_1ty_rx_2$ is the cycle we need, and if $r-1\neq 2$, then observe that there is a $(y_2,x_{r-1})$-path $P$ (of course $y_2=x_{r-1}$ is possible) in $G_1$ which does not pass through $x_1,x_2,y_{r-1},y_r$ (the same path as in $G_1$ before removing additional arcs), then $x_2y_2Px_{r-1}y_{r-1}x_1ty_rx_2$ is the cycle we need, so $G_2$ is strong.

    \textbf{Case 2.2:} When $t\neq t^*$:
   
    We recommend that readers refer to Remark~\ref{rmk:legal} before proceeding with the following operations.
    
    1. We do the same arguments as those in Case 1.1 if we have entered \ref{ope:a6} or Case 1.2 if we have entered \ref{ope:a7} to obtain a pending decomposition of $D$ with the additional arcs $y_ry_{r-1},y_ry_{r-1},y_{r-1}y_r$.

    2. We apply \ref{choose (1)}*-\ref{choose (4)}* on $G_1$ and $G_2$, where \ref{choose (1)}*-\ref{choose (4)}* denote the symmetric operations of \ref{choose (1)}-\ref{choose (4)}.

    3.  We do the same arguments as those in Case 1.1* if we entered \ref{ope:a6}* or Case 1.2* if we entered \ref{ope:a7}* to get a pending decomposition of $D$, where Case 1.1* and Case 1.2* are the symmetric cases of Case 1.1 and Case 1.2.

    In every case from Case 1 to Case 2, we can always get a pending decomposition $G_1$ and $G_2$ of $D$, which meets our needs.
\end{proof}
\begin{remark}\label{rmk:legal}
    Recall that during the process of this proof, when we apply \ref{choose (1)}-\ref{choose (4)} and Case 1.1, Case 1.2, all the conditions we need are:

    \begin{enumerate}[i).]
        \item\label{i)} $G_1$ and $G_2$ are strong;
        
        \item\label{ii)} All out-arcs of $\{x_1,x_2\}$ in $A(G_1)\cup A(G_2)$ before we do \ref{choose (1)} are $x_2y_2,x_2t$;

        \item\label{iii)} All arcs incident to $t$ in $\in A(G_1)\cup A(G_2)$ are $x_2t$ and $tt^+$, so we can find a distinct in-arc and a distinct out-arc of $t$ in $D$ to add to $G_1$.
    \end{enumerate}

    To do \ref{choose (1)}*-\ref{choose (4)}* and Case 1.1*, Case 1.2*, we need to check the symmetric conditions, which are:

    \begin{enumerate}[i)*.]
        \item\label{i)} $G_1$ and $G_2$ are strong;
        
        \item\label{ii)} All in-arcs of $\{y_{r-1},y_r\}$ in $A(G_1)\cup A(G_2)$ before we do \ref{choose (1)}* are $x_{r-1}y_{r-1},t^*y_{r-1}$;

        \item\label{iii)} All arcs incident to $t^*$ in $\in A(G_1)\cup A(G_2)$ are $t^*y_{r-1},t^{*-}t^*$, so we can find a distinct in-arc and a distinct out-arc of $t$ in $D$ to add to $G_1$.
    \end{enumerate}

 \ref{i)})* is straightforward to verify. For \ref{ii)})*, as $t^*\neq t$ and $N^-(y_r)\cap V_1=N^-(y_{r-1})\cap V_1=\{t^*\}$, so all in-arcs of $\{y_{r-1},y_r\}$ in $A(G_1)\cup A(G_2)$ are $x_{r-1}y_{r-1},t^*y_{r-1}$. For \ref{iii)})*, because $t^*\neq t$, we did not add any arcs incident to $t^*$. This necessitates distinguishing between Case 2.1 and Case 2.2 in this proof.
\end{remark}
	
\subsection{$D[ V_2]$ is not strong}

If $D[V_2]$ is not strong, then it has the acyclic ordering of its strong component $C_1,\ldots{},C_p$ ($p\geq 2$). Similar to Lemma~\ref{lem:cut-arc}, we have the following lemma.

\begin{lemma}\label{lem:cut-arcs in not strong}
	If the multi-digraph $D[V_2]+C$, where $C$ is a set of some arcs with both endpoints in $V_2$, has two arc-disjoint $(C_p,C_1)$-paths $Q_1$ and $Q_2$, then $D[V_2]+C$ is strong. Besides, if $D[V_2]+C$ is not 2-arc-strong, which means there exists a cut-arc $xy$, then $xy \notin C$ and at least one of the following occurs:
	\begin{enumerate}
		\renewcommand{\theenumi}{($\beta$\arabic{enumi})}
		\item\label{(beta1)} $C_p=\{y\},$ $C_{p-1}=\{x\}$, $d^+_{D[V_2]+C}(x)=1$,
		
		\item\label{(beta2)} $C_1=\{x\},$ $C_{2}=\{y\}$, $d^-_{D[V_2]+C}(y)=1$,
		
		\item\label{(beta3)} $xy\in A(D[C_p])$,
		
		\item\label{(beta4)} $xy\in A(D[C_1])$.
	\end{enumerate}
\end{lemma}

\begin{proof}
    Since $D[V_2]$ is semicomplete and $C_1,\ldots{},C_p$ is its acyclic ordering, $D[V_2]+Q_i,i\in [2]$ is strong, and so $D[V_2]+C$ is strong. 
 
    If $D[V_2]+C$ is not 2-arc-strong, let $xy$ be a cut-arc in $D[V_2]+C$. If $xy\in C$, then there exists $i\in [2]$ such that $xy\notin Q_i$ since $Q_1$ and $Q_2$ are arc-disjoint. This yields that $(D[V_2]+C) \setminus xy\supseteq D[V_2]+Q_{i}$ is still strong, which contradicts to the fact $xy$ is cut-arc. Hence $xy\notin C$.
	
	
    By checking the existence of the following form of path
    \[x\dashrightarrow C_p\stackrel{Q_i}{\longrightarrow}C_1 \dashrightarrow y,\]
    in $(D[V_2]+C) \setminus xy$, where $xy\notin Q_i$, we have $xy\notin D[ V_2-C_1-C_p]$.
	
	If $y\in C_p, x\notin C_p$. If $C_p\neq \{y\},$ then there exists $z\in C_p$ and $z$ dominates $y$ because of $C_p$ is strong, and $x\rightarrow z\rightarrow y$ is an $(x,y)$-path other than $xy$, which contradicts to the fact that $xy$ is a cut-arc in $D[V_2]+C$. Hence $C_p= \{y\}$. If $x\notin C_{p-1}$, then $x\rightarrow z\rightarrow y$ is an $(x,y)$-path other than $xy$, where $z\in C_{p-1}$, which contradicts to the fact that $xy$ is a cut-arc in $D[V_2]+C$. Hence $x\in C_{p-1}$. If $C_{p-1}\neq \{x\}$, then as $C_{p-1}$ is strong, there exists out-arc $xz\in D[C_{p-1}]$, and $x\rightarrow z\rightarrow y$ is an $(x,y)$-path other than $xy$, which contradicts to the fact that $xy$ is a cut-arc in $D[V_2]+C$. Hence $C_{p-1}=\{x\}$. If $d_{D[V_2]+C}^+(x)\geq 2$, then there exists an $x^+$ (maybe $y$), such that $x\rightarrow x^+\rightarrow y$ (or $x\rightarrow y$) in $(D[V_2]+C) \setminus xy$, which contradicts to the fact that $xy$ is a cut-arc in $D[V_2]+C$. Hence $d^+_{D[ V_2]+C}(x)=1$. In conclusion, $C_p=\{y\},$ $C_{p-1}=\{x\}$, $d^+_{D[ V_2]+C}(x)=1$.
	
    Similarly, if $x\in C_1, y\notin C_1$, then $C_1=\{x\},$ $C_{2}=\{y\}$, $d^-_{D[ V_2]+C}(y)=1$.
	
    As $xy\in D[V_2]$, then $x\in C_p, y\notin C_p$ or $y\in C_1, x\notin C_1$ will not occur, so the left cases are \ref{(beta3)} and \ref{(beta4)}.
\end{proof}

Let $(U_{1},\ldots, U_{l})$ be the nice decomposition of $C_p$ when $|C_p|\geq 4$.  Similarly, suppose that $(W_{1},\ldots, W_{p})$ is the nice decomposition of $C_1$ when $|C_1|\geq 4$.
\[X=\begin{cases}
	U_{l}, &\mbox{if } |C_p|\geq 4;\\
	C_p, &\mbox{otherwise}.
\end{cases} \mbox{ \;\;and \;\;} Y=\begin{cases}
	W_{1}, &\mbox{if } |C_1|\geq 4;\\
	C_1, &\mbox{otherwise}.
\end{cases}\]

The following proof is similar to the proof for the case that $D[V_2]$ is strong. First, we let $B=\emptyset$ and $\{Q_1,Q_2\}$ be an $(X,Y)_{B}$-critical path pair. If $D[V_2]+\bar{A}_{Q_1,Q_2}(B)+\bar{B}$ is 2-arc-strong, then by Lemma~\ref{lem:critical AC to SAD}, Theorem~\ref{thm:multi semi SAD} and $|V_2|\geq 5$, we are done. If it is not 2-arc-strong, then by Lemma~\ref{lem:cut-arcs in not strong}, at least one of \ref{(beta1)} to \ref{(beta4)} occurs. Note that \ref{(beta1)} and \ref{(beta3)} (or \ref{(beta2)} and \ref{(beta4)}) can not occur at the same time. By symmetry of \ref{(beta1)} and \ref{(beta2)} and also symmetry of \ref{(beta3)} and \ref{(beta4)}, we may assume that \ref{(beta1)} or \ref{(beta3)} occurs (of course, \ref{(beta2)} or \ref{(beta4)} may occur at the same time). Similar to the proof for $D[V_2]$ is strong, we do the following operations for different cases. 

For convenience, we denote by $G$ the digraph $D[V_2]+B+A_{Q_1,Q_2}(B)$ and $\bar{G}$ the digraph $D[V_2]+\bar{B}+\bar{A}_{Q_1,Q_2}(B)$. Similarly, we can define $\bar{G}_{\text{new}}$ after we define $G_\text{new}$, which we will determine later in different cases. We give the corresponding procedures on $D$, $B$ and $\{Q_1,Q_2\}$ to avoid \ref{(beta1)} or \ref{(beta3)}.

Here, we recall an operation for rebuilding $B$.
For a vertex $f\in V_1(D)$, let $e$ be an in-neighbor of $f$. If $f\notin V_1(B)$, there is an out-neighbor $f^+$ of $f$ which is not $e$, since $D$ is 2-arc-strong. Add arcs $ef$ and $ff^+$ to $B$. If $f\in V_1(B)$, which means there are arcs $f^-f,ff^+ \in B$, then replace $f^-f$ with $ef$ and replace $ff^+$ with $ff'$, where $f'$ is another out-neighbor of $f$, if $f^+=e$. We use $B\leftarrow\{ef,ff^+\}$ to denote this operation.

\subsubsection{$|C_p|=1$}
If \ref{(beta1)} or \ref{(beta3)} occurs, then it can only be \ref{(beta1)}. And we have $C_p=\{b\},$ $C_{p-1}=\{a\}$, $d^+_{D[V_2]+C}(a)=1$, where $ab$ is the cut-arc in $D[V_2]+\bar{A}_{Q_1,Q_2}(B)+\bar{B}$.
We may assume $Q_1=bb_1b_1^+Q_1[b_1^+,Y], Q_2=bb_2b_2^+Q_2[b_2^+,Y]$, where $b_1,b_2\in V_1$, $a\notin V(Q_1), a\notin V(Q_2)$. As $D$ is 2-arc-strong, there is another out-neighbor of $a$, say $u$. And $u \in V_1$ since $a$ has exactly one out-neighbor in $V_2$. We give the following procedures to avoid \ref{(beta1)}. 

\begin{enumerate}
    \renewcommand{\theenumi}{\textbf{(B\arabic{enumi})}}
    \item\label{ope:b1} When $u \notin V_{Q_1, Q_2}$:
    
    Let $B_\text{new}:=B\leftarrow\{au,uu^+\}$. Observe that $V_{Q_1,Q_2}(B_\text{new})\cap V_1(B_\text{new})=\emptyset$, so there is an $(X,Y)_{B_\text{new}}$-critical path pair $\{Q_3,Q_4\}$ such that $A_{Q_3,Q_4}(B_\text{new})\subseteq A_{Q_1,Q_2}(B_\text{new})$. Let $G_\text{new}$ be $D[V_2]+B_\text{new}+A_{Q_3,Q_4}(B_\text{new})$. In this case, $d^{+}_{\bar{G}_\text{new}}(a) \geq 2$ which avoids \ref{(beta1)}. 
    
    Additionally, Remark~\ref{rmk:end procedure} applies here as do the following procedures from \ref{ope:b2} to \ref{ope:d5}. Therefore, either $|A_{Q_3,Q_4}(B_\text{new})|<|A_{Q_1,Q_2}(B)|$ or for each vertex $v \in V_2\backslash\{a\}$, we have $d^-_{\Bar{G}}(v) \leq d^-_{\Bar{G}_{\text{new}}}(v)$, and $V_1(B)\subseteq V_1(B_\text{new})$.
    
    \item\label{ope:b2} When there exists $i\in [2]$ such that $u\in Q_i$, and $b_i\neq u$:

    If $u\in V_1(B)$, we define $B_\text{new}:=B\leftarrow\{au,uu^+\}$, and if $u\notin V_1(B)$, we define $B_{\text{new}}:=B$. Let $Q_i'=bb_ib_i^+auQ_i[u,Y]$, observe that $Q_i'$ is arc-disjoint with $Q_{3-i}$, and $V_{Q_{3-i}',Q_i'}(B_\text{new})\cap V_1(B_\text{new})=\emptyset$, so there is an $(X,Y)_{B_{\text{new}}}$-critical path pair $\{Q_3,Q_4\}$ such that $A_{Q_3,Q_4}(B_\text{new})\subseteq A_{Q_{3-i},Q_i'}(B_\text{new})$. 
      
    Let $G_\text{new} := D[V_2]+B_\text{new}+A_{Q_3,Q_4}(B_\text{new})$. If $au\notin G_{\text{new}}$, then $B_\text{new}=B$, $A_{Q_3,Q_4}(B)\subsetneq A_{Q_1,Q_2}(B)$  as $u^-u\notin A_{Q_3,Q_4}(B)$, which contradicts to the fact that $\{Q_1,Q_2\}$ is $(X, Y)_{B}$-critical. So, we have $au\in G_{\text{new}}$, then we have $d^{+}_{\bar{G}_\text{new}}(a) \geq 2$, which avoids \ref{(beta1)}.
    
    \textbf{If \ref{ope:b1}-\ref{ope:b2} do not occur for all $u\in V_1 \cap N^{+}_D(a)$, then we have $N^+_D(a)\cap V_1\subseteq \{b_1,b_2\}$. If $b_i\in N^+_D(a)\cap V_1$ for some $i\in [2]$, we have $b_i\notin Q_{3-i}$, since if $b_i\notin Q_{3-i}$, this makes \ref{ope:b2} occur as $b_1\neq b_2$. Then we enter the following procedures.}\\
    
    \item\label{ope:b3} When there exists another out-neighbor $b_3\neq b_1,b_2$ of $b$:

    As $N^+_D(a)\cap V_1\neq \emptyset$, there exists $i\in [2]$, such that $b_i\in N^+_D(a)\cap V_1$. Like \ref{ope:a4} and \ref{ope:a5}, firstly, we get $B_\text{new}$ from $B$ by doing the following two operations.
    \begin{itemize}
        \item $B\leftarrow\{ab_i,b_ib_i^+\}$.
        

        \item If $b_3\notin V_{Q_1,Q_2}(B)$, then $B\leftarrow\{bb_3,b_3b_3^+\}$. And if $b_3\in V_{Q_1,Q_2}(B)$, then we do nothing.

    \end{itemize}

    If $b_3\notin Q_1, Q_2$, then let $Q_i'=bb_3b_3^+ab_iQ_i[b_i,Y]$ ($b_3^+$ can be $a$, or in $Y$), $Q_{3-i}'=Q_{3-i}$. 
    
    If $b_3\in Q_i$, then let $Q_i'=bb_3Q_i[b_3,Y]$, $Q_{3-i}'=Q_{3-i}$. 
    
    If $b_3\notin Q_i,b_3\in Q_{3-i}$, then let $Q_i'=bb_3Q_{3-i}[b_3,Y]$, $Q_{3-i}'=bb_{3-i}b_{3-i}^+ab_iQ_i[b_i,Y]$ ($b_{3-i}^+$ can be in $Y$).
    
    Observe that $Q_1'$ is arc-disjoint with $Q_2'$, and $V_{Q_1',Q_2'}(B_\text{new})\cap V_1(B_\text{new})=\emptyset$, so there an $(X,Y)_{B_{\text{new}}}$-critical path pair $\{Q_3,Q_4\}$ such that $A_{Q_3,Q_4}(B_\text{new})\subseteq A_{Q_1',Q_2'}(B_\text{new})$. Let $G_\text{new} := D[V_2]+B_\text{new}+A_{Q_3,Q_4}(B_\text{new})$. We have $d^{+}_{\bar{G}_\text{new}}(a) \geq 2$, which avoids \ref{(beta1)}, as $ab_i\in B_\text{new}$.
 
    \item\label{ope:b4} When $N_D^+(b)=\{b_1,b_2\}$, and $N_D^+(a)\cap V_1=\{b_i\}$ for some $i=[2]$:

    If $N_D^+(b_i)=\{b_i^+,a\},N_D^-(b_i)=\{a,b\}$, note that $D$ satisfies the structure in Lemma \ref{lem:counter example}. Therefore, $D$ has no strong arc decomposition.

    If there exists $w\in N^+_D(b_i)\backslash\{b_i^+,a\}$ or if $N_D^+(b_i)=\{b_i^+,a\}$ and there exists $w\in N^-_D(b_i)\backslash\{a,b\}$, then we additionally add two parallel arcs $ba$ and another arc $ab$ to $D[V_2]$. Let $B_\text{new}=B\leftarrow\{ab_i,b_ib_i^+\}$, and we can check $\{Q_3=Q_1, Q_4=Q_2\}$ is an $(X,Y)_{B_{\text{new}}}$-critical path pair, and let $G_\text{new} := D[V_2]+B_\text{new}+A_{Q_3,Q_4}(B_\text{new})$. We have $d^{+}_{\bar{G}_\text{new}}(a) \geq 2$, which avoids \ref{(beta1)}. Compared with the case $D[V_2]$ is strong but not 2-arc-strong, if we regard vertex $a$ as $x_1$, $b$ as $x_2$, $b_i$ as $t$, $2$-path $bb_{3-i}b_{3-i}^+$ as arc $x_2y_2$, then what we do here is indeed the same as those in \ref{ope:a6} and \ref{ope:a7}. Let $R:=\{b_i\}$, we call it optional vertex set.
    
    \item\label{ope:b5} When $N_D^+(b)=\{b_1,b_2\}$, $N_D^+(a)\cap V_1 =\{b_1,b_2\}$, and $N_D^+(b_i)=\{b_i^+,a\},N_D^-(b_i)=\{a,b\}$ for all $i\in [2]$:
    
    We additionally add two parallel arcs $ba$ and another arc $ab$ to $D[V_2]$. Let $B_\text{new}=(B\leftarrow\{ab_1,b_1b_1^+\})\leftarrow\{ab_2,b_2b_2^+\}, Q_3=bab_1Q_1[b_1,Y], Q_4=bab_2Q_2[b_2,Y]$. Observe that $\{Q_3,Q_4\}$ is $(X,Y)_{B_{\text{new}}}$-critical path pair, we define $G_\text{new} := D[V_2]+B_\text{new}+A_{Q_3,Q_4}(B_\text{new})$.

    \item\label{ope:b6} When $N_D^+(b)=\{b_1,b_2\}$, $N_D^+(a)\cap V_1 =\{b_1,b_2\}$, and not in \ref{ope:b5}:

    Let $R$ be th optional vertex set defined as $R:=\{b_i~|i\in [2], \text{such that $N_D^+(b_i)\neq\{b_i^+,a\}$ or $N_D^-(b_i)\neq\{a,b\}$}\}$. Note that $R\neq \emptyset$, otherwise we are in case~\ref{ope:b5}. Then, we proceed by applying the same operations as those in \ref{ope:b4} for some $b_i\in R$. Note that the choice of $b_i$ depends on additional considerations related to \ref{(beta2)} and \ref{(beta4)}.
\end{enumerate}

\subsubsection{$|C_p|=2$}

If \ref{(beta1)} or \ref{(beta3)} occurs, then it must be \ref{(beta3)}. Since $C_p$ is strong, then $D[C_p]$ forms a 2-cycle. As $G$ has a cut-arc in $D[C_p]$, we can deduce that paths $Q_1$ and $Q_2$ must share the same initial vertex, denoted $b$, and another vertex in $C_p$ is denoted $a$. Observe that only $ab$ can be the cut-arc and the out-degree of $a$ in $G$ must be $1$, so if we can prove $d^{+}_{\bar{G}_\text{new}}(a) \geq 2,d^{+}_{\bar{G}_\text{new}}(b) \geq 2$ for the $\bar{G}_\text{new}$ we get in the following, then we can avoid \ref{(beta3)}. As $D$ is 2-arc-strong, we have $u\in N^+_D(a)\cap V_1$, we proceed with the following operations:

\begin{enumerate}
    \renewcommand{\theenumi}{\textbf{(C\arabic{enumi})}}
    \item\label{ope:c1} When $u \notin V_{Q_1, Q_2}$:

    Let $B_\text{new}:=B\leftarrow\{au,uu^+\}$. Observe that $V_{Q_1,Q_2}(B_\text{new})\cap V_1(B_\text{new})=\emptyset$, so there is a $(X,Y)_{B_\text{new}}$-critical path pair $\{Q_3,Q_4\}$ such that $A_{Q_3,Q_4}(B_\text{new})\subseteq A_{Q_1,Q_2}(B)$. Let $G_\text{new}$ be $D[V_2]+B_\text{new}+A_{Q_3,Q_4}(B_\text{new})$.

    In this case, $d^{+}_{\bar{G}_\text{new}}(a)=2,d^{+}_{\bar{G}_\text{new}}(b) \geq 2$ which avoids \ref{(beta3)} as $au\in B_\text{new}$. 
	
    \item\label{ope:c2} When $u\in V_{Q_1, Q_2}$:

    There exists $i\in [2]$ such that $u\in Q_i$. If $u\in V_1(B)$, we define $B_\text{new}:=B\leftarrow\{au,uu^+\}$, and if $u\notin V_1(B)$, we define $B_{\text{new}}:=B$. Let $Q_i'=auQ_i[u,Y]$, observe that $Q_i'$ is arc-disjoint with $Q_{3-i}$ and $V_{Q_i',Q_{3-i}}(B_\text{new})\cap V_1(B_\text{new})=\emptyset$, so there is an $(X,Y)_{B_{\text{new}}}$-critical path pair $\{Q_3,Q_4\}$ such that $A_{Q_3,Q_4}(B_\text{new})\subseteq A_{Q_{3-i},Q_i'}(B_\text{new})$. 
      
    Let $G_\text{new} := D[V_2]+B_\text{new}+A_{Q_3,Q_4}(B_\text{new})$. 

    If $au\notin G_{\text{new}}$, then $B_\text{new}=B$, $A_{Q_3,Q_4}(B)\subsetneq A_{Q_1,Q_2}(B)$  as $u^-u\notin A_{Q_3,Q_4}(B)$, which contradicts to the fact that $\{Q_1,Q_2\}$ is $(X, Y)_{B}$-critical. So, we have $au\in G_{\text{new}}$, then we have $d^{+}_{\bar{G}_\text{new}}(a)=2,d^{+}_{\bar{G}_\text{new}}(b) \geq 2$, which avoids \ref{(beta3)}.
\end{enumerate}

\subsubsection{$|C_p|=3$}

If either \ref{(beta1)} or \ref{(beta3)} occurs, it must specifically be \ref{(beta3)}. There exists a 3-cycle $abca$ in $D[C_p]$. We can perform preliminary operations on ${Q_1,Q_2}$ to obtain a somewhat `minimal' $(X,Y)$-path pair.

\begin{enumerate}
    \renewcommand{\theenumi}{\textbf{(D\arabic{enumi})}}
    \setcounter{enumi}{-1}
    \item\label{ope:d0} We may assume that $Q_1=q_1q_1^+Q_1[q_1^+,Y], Q_2=q_2q_2^+Q_2[q_2^+,Y]$ (The vertices $q_1$ and $q_2$ are allowed to coincide). If, for $v \in \{a,b,c\}$, there exists $u\in (V_1(Q_i)\backslash\{q_i^+\})\cap N^+(v)$ for some $i\in[2]$ and $vu\neq q_{3-i}q_{3-i}^+$, then there exists $u^-$ such that $u^-u\in A(Q_i)$. If $u^-u\in B$, we replace $u^-u$ with $vu$ in $B$ to get $B_{\text{new}}$, and if $u^-u\notin B$, we define $B_{\text{new}}:=B$. For this $i$, let $Q_i'=vuQ_i[u,Y]$, observe that $Q_i'$ is arc-disjoint with $Q_{3-i}$ and $V_{Q_i',Q_{3-i}}(B_\text{new})\cap V_1(B_\text{new})=\emptyset$, there exists an $(X,Y)_{B_{\text{new}}}$-critical path pair $\{Q_3,Q_4\}$ such that $A_{Q_3,Q_4}(B_\text{new})\subseteq A_{Q_{3-i},Q_i'}(B_\text{new})$. And we choose $Q_3=Q_{3-i},Q_4=Q_i'$ if $\{Q_i',Q_{3-i}\}$ is already an $(X,Y)_{B_{\text{new}}}$-critical path pair. Let $G_\text{new} := D[V_2]+B_\text{new}+A_{Q_3,Q_4}(B_\text{new})$. In the end, we rename $G_\text{new}$ as $G$, $B_\text{new}$ as $B$, $\{Q_3,Q_4\}$ as $\{Q_1,Q_2\}$. We repeat the above process until we can not find such vertices $v$ and $u$.
\end{enumerate}

\begin{claim}
Procedure~\ref{ope:d0} will stop in some $G$.
\end{claim}

\begin{proof}
    Observe that $|A_{Q_3,Q_4}(B_\text{new})|\leq |A_{Q_i',Q_{3-i}}(B_\text{new})|\leq |A_{Q_1,Q_2}(B)|$. If $|A_{Q_3,Q_4}(B_\text{new})|=|A_{Q_1,Q_2}(B)|$, then $A_{Q_3,Q_4}(B_\text{new})=A_{Q_i',Q_{3-i}}(B_\text{new})$, so we have that $Q_3=Q_{3-i},Q_4=Q_i'$ and $|A_{Q_3, Q_4}|<|A_{Q_i, Q_{3-i}}|$. Since the original $|A_{Q_1,Q_2}|$ is bounded, the process of repeating \ref{ope:d0} will end in a finite number of repetitions if $|A_{Q_3,Q_4}(B_\text{new})|=|A_{Q_1,Q_2}(B)|$ always holds. Thus, if the claim does not hold, then the event that $|A_{Q_3,Q_4}(B_\text{new})|< |A_{Q_1,Q_2}(B)|$ will occur infinitely, which contradicts that the original $|A_{Q_1,Q_2}(B)|$ is bounded.
\end{proof}

\begin{observation}\label{obs: out-vertex q1q2}
    After \ref{ope:d0}, for any $v\in \{a,b,c\}$, if there exists $u\in N_D^+(v)\cap V_{Q_1,Q_2}$, then $u\in \{q_1^+,q_2^+\}$, where $Q_1=q_1q_1^+Q_1[q_1^+,Y], Q_2=q_2q_2^+Q_2[q_2^+,Y]$.
    
\end{observation}

\begin{observation}\label{obs: one of q1q2 appear 1}
    As $\{Q_1,Q_2\}$ is an $(X,Y)_{B}$-critical path pair, we have at least one of $q_1^+,q_2^+$ will not appear in another path if $q_1^+\neq q_2^+$, where $Q_1=q_1q_1^+Q_1[q_1^+,Y], Q_2=q_2q_2^+Q_2[q_2^+,Y]$.
\end{observation}

\begin{observation}\label{obs:change initial}
    For $G=D[V_2]+B+A_{Q_1,Q_2}(B)$ with $Q_1=q_1q_1^+Q_1[q_1^+,Y], Q_2=q_2q_2^+Q_2[q_2^+,Y]$, let $B_\text{new}:=B\leftarrow\{q_3q_i^+,q_i^+q_i^{++}\}$ if $q_i^+\in V_1(B)$ and $B_\text{new}:=B$ if $q_i^+\notin V_1(B)$ for some $q_3\in X, q_3\neq q_i$, then $\{Q_3,Q_4\}$ is an $(X,Y)_{B_{\text{new}}}$-critical path pair, where $Q_3=q_3q_i^+Q_i[q_i^+,Y], Q_4=q_{3-i}q_{3-i}^+Q_{3-i}[q_{3-i}^+,Y]$. Because we just replace an initial vertex of some $Q_i,i\in [2]$, with another vertex in $X$. In the following context, when we use this observation, we may omit to define $B_\text{new}$ and $Q_3,Q_4$ for simplicity.

\end{observation}

If \ref{(beta3)} still occurs after \ref{ope:d0}, we enter one of the following procedures for different conditions.


\begin{enumerate}
	\renewcommand{\theenumi}{\textbf{(D\arabic{enumi})}}
    \item\label{ope:d1} The initials of $Q_1,Q_2$ are same, says $a$, and $(N_D^+(b)\cup N_D^+(c))\cap V_{Q_1, Q_2}=\emptyset$:

    We may assume $Q_1=aa_1a_1^+Q_1[a_1^+,Y], Q_2=aa_2a_2^+Q_2[a_2^+,Y]$. By observation~\ref{obs: one of q1q2 appear 1}, there exists $i\in [2]$, such that $a_i\notin V_1(Q_{3-i})$, we move $aa_i,a_ia_i^+$ to $B$ if they belong to $A_{Q_1,Q_2}(B)$. 
    
    \begin{enumerate}
        \item When $ba\in D$, as \ref{(beta3)} occurs, we have $cb\notin D[C_p]$ and $d_G^+(c)=1$, $ca$ is the cut-arc. There is another out-neighbor of $c$ in $V_1$, which says $u$. $B_\text{new}:=B\leftarrow\{cu,uu^+\}$. We can find an $(X,Y)_{B_{\text{new}}}$-critical path pair $\{Q_3,Q_4\}$ such that $A_{Q_3,Q_4}(B_\text{new})\subseteq A_{Q_1,Q_2}(B_\text{new})$. Let $G_\text{new} := D[V_2]+B_\text{new}+A_{Q_3,Q_4}(B_\text{new})$. This avoids \ref{(beta3)} as $\{aa_i,aa_i^+,cu,uu^+\}\subseteq B_\text{new}$.

        \item When $ba\notin D$:
		
		If for any $u\in N^+_D(b)\cap V_1$, $N_D^+(u)=\{b,c\}$, then there exist two distinct vertices $u_1,u_2\in V_1$, such that $bu_1, u_1c, cu_2,u_2u_2^+\in D$, where $u_2^+\in V_2\setminus\{b,c\}$ as there is another $(\{b,c\},V_2\backslash\{b,c\})$-path, besides the arc $ca$. Let $B_\text{new}:=(B\leftarrow\{bu_1, u_1c\})\leftarrow\{cu_2,u_2u_2^+\}$. We can find an $(X,Y)_{B_{\text{new}}}$-critical path pair $\{Q_3,Q_4\}$ such that $A_{Q_3,Q_4}(B_\text{new})\subseteq A_{Q_1,Q_2}(B_\text{new})$. Let $G_\text{new} := D[V_2]+B_\text{new}+A_{Q_3,Q_4}(B_\text{new})$. As $\{aa_i,aa_i^+,bu_1,u_1c,cu_2,u_2u_2^+\}\subseteq B_\text{new}$, we avoid \ref{(beta3)}.
		
		If not, then there exists a vertex $u\in N^+_D(b)\cap V_1$, and it has an out-neighbor $u^+\in V_2\setminus\{b,c\}$. let $B_\text{new}:=B\leftarrow\{bu,uu^+\}$. We can find an $(X,Y)_{B_{\text{new}}}$-critical path pair $\{Q_3,Q_4\}$ such that $A_{Q_3,Q_4}(B_\text{new})\subseteq A_{Q_1,Q_2}(B_\text{new})$. Let $G_\text{new} := D[V_2]+B_\text{new}+A_{Q_3,Q_4}(B_\text{new})$. And if \ref{(beta3)} still occurs for $G_\text{new}$, then we have $cb\notin D$, and $N_{G_\text{new}}^+(c)=1$. We regard $G_\text{new}$ as $G$, $b$ as $x_1$, $c$ as $x_2$, $u$ as $t$ and enter the procedures of \ref{ope:a1}-\ref{ope:a8} to get the new $G_\text{new}$ or enter an counterexample. \\
        Initially, we move $\{aa_i,aa_i^+\}\subseteq B$. As $a_i\notin N_D^+(b)\cup N_D^+(c)$, we have never removed these two arcs from $B$ when we enter \ref{ope:a1}-\ref{ope:a7}. So, we can avoid \ref{(beta3)} now as there are two arc-disjoint $(\{b\},Y)$-paths and $d_{G_\text{new}}^+(c)\geq 2$. Note that if we enter \ref{ope:a6} or \ref{ope:a7}, then we define optional vertex set $R:=\{u\}$. 
    \end{enumerate}

    \item\label{ope:d2} The initials of $Q_1,Q_2$ are same, says $a$, and $(N^+(b)\cup N^+(c)\cap V_{Q_1, Q_2}\neq\emptyset$:

    We may assume that $Q_1=aa_1Q_1[a_1,Y], Q_2=aa_2Q_2[a_2,Y]$. As what we did in \ref{ope:d0}, we have $(N^+(b)\cup N^+(c)\cap V_{Q_1, Q_2}\in\{a_1,a_2\}$. If there exists $i\in [2]$, such that $a_i$ is an out-neighbor of $b$, then we replace $aa_i$ in $G$ with $ba_i$ to get $G_\text{new}$. If not, then there exists $i\in [2]$, such that $a_i$ is an out-neighbor of $c$, we replace $aa_i$ in $G$ with $ca_i$ to get $G_\text{new}$.

    Now the initials of $Q_3,Q_4$ are different and $\{Q_3,Q_4\}$ is still $(X,Y)_{B_\text{new}}$-critical path pair by Observation~\ref{obs:change initial}, and if \ref{(beta3)} still occurs for $G_\text{new}$, we rename $G_\text{new}$ as $G$, $B_\text{new}$ as $B$, $\{Q_3,Q_4\}$ as $\{Q_1,Q_2\}$ and do one of the following operations.
    

    \textbf{In the following cases, the initials of $Q_1,Q_2$ are different, we may assume $Q_1=aa_1a_1^+Q_1[a_1^+,Y]$, $Q_2=bb_1b_1^+Q_2[b_1^+,Y]$ by relabeling the name of $a,b,c$. Besides, as \ref{(beta3)} occurs for $G$, we have $d^+_{G_\text{new}}(c)=1,cb\notin D$. And as what we did in \ref{ope:d0}, if there is an out-neighbor $u$ of some vertex in $\{a,b,c\}$, such that $u\in V_{Q_1, Q_2}$, then it must be $a_1$ or $b_1$.}
	
    \item\label{ope:d3} When there exists $u\in N_D^+(c)\cap V_1$, such that $u\notin V_{Q_1, Q_2}$:
    
    Let $B_\text{new}:=B\leftarrow\{cu,uu^+\}$, we can find an $(X,Y)_{B_{\text{new}}}$-critical path pair $\{Q_3,Q_4\}$ such that $A_{Q_3,Q_4}(B_\text{new})\subseteq A_{Q_1,Q_2}(B_\text{new})$. Let $G_\text{new} := D[V_2]+B_\text{new}+A_{Q_3,Q_4}(B_\text{new})$.
    
    If \ref{(beta3)} still occurs for $G_\text{new}$, then we have $A_{Q_3,Q_4}(B_\text{new})\subsetneq A_{Q_1,Q_2}(B_\text{new})$, which means $|A_{Q_3,Q_4}(B_\text{new})|< |A_{Q_1,Q_2}(B)|$. Then we rename $G_\text{new}$ as $G$, $B_\text{new}$ as $B$, $\{Q_3,Q_4\}$ as $\{Q_1,Q_2\}$ and repeat the whole process from \ref{ope:d0}.


	
    \item\label{ope:d4} When $N_D^+(c)\cap V_1\subseteq V_{Q_1, Q_2}$, and $a_1=b_1$:
	
	As what we did in \ref{ope:d0}, we have $N_D^+(c)\cap V_1=\{a_1\}$.
	
	If there exists $x\in \{a,b\}$, such that there exists $u\in N_D^+(x)\cap V_1,u\neq a_1$, then as what we do in \ref{ope:d0}, we have $u\notin V_{Q_1, Q_2}$. We replace $xa_1$ with $ca_1$ in $G$ to obtain $G_\text{new}$, and if \ref{(beta3)} still occurs, then we relabel the name of $a,b,c$ and enter \ref{ope:d3}.
	
	If $N_D^+(a)\cap V_1=N_D^+(b)\cap V_1=N_D^+(c)\cap V_1=\{a_1\}$, then we add additional arcs $ab,bc,ca$ to $D[V_2]$ to get $G_\text{new}$.

    \item\label{ope:d5} When $N_D^+(c)\cap V_1\subseteq V_{Q_1, Q_2}$, and $a_1\neq b_1$:
    
    If $N_D^+(c)\cap V_1=\{a_1\}$, we replace $aa_1$ in $G$ with $ca_1$ to get $G_\text{new}$ by Observation~\ref{obs:change initial}, if \ref{(beta3)} still occurs for $G_\text{new}$, we rename $G_\text{new}$ as $G$, $B_\text{new}$ as $B$, $\{Q_3,Q_4\}$ as $\{Q_1,Q_2\}$, and we relabel $a,b,c$ to fit the conditions, then we enter \ref{ope:d3} or enter \ref{ope:d5} with $b_1\in N_D^+(c)$. Thus, We can always assume $b_1\in N_D^+(c)$.

    If $ba\in D$, then we replace $bb_1$ with $cb_1$ in $G$ to get $G_\text{new}$ by Observation~\ref{obs:change initial}, which avoid \ref{(beta3)}, so we may assume $ba\notin D$. 

    If there exists $u\in N_D^+(b)\cap V_1,u\notin \{a_1,b_1\}$, we replace $bb_1$ in $G$ with $cb_1$ to get $G_\text{new}$ by Observation~\ref{obs:change initial}, if \ref{(beta3)} still occurs for $G_\text{new}$, we rename $G_\text{new}$ as $G$, $B_\text{new}$ as $B$, $\{Q_3,Q_4\}$ as $\{Q_1,Q_2\}$, and we relabel $a,b,c$ to fit the conditions, so we enter \ref{ope:d3} by Observation \ref{obs: out-vertex q1q2}. 

    Now, we consider the following two cases with $ba\notin D$, $N_D^+(b)\cap V_1\subseteq\{a_1,b_1\}$.
    
    \begin{enumerate}
        \item\label{ope:d5a} When $N_D^+(c)\cap V_1=\{b_1\}$:
     
    	
    	
        If $N_D^+(b)\cap V_1=\{b_1\}$, then we regard $b$ as $x_1$, $c$ as $x_2$, $b_1$ as $t$, and same arguments as those in \ref{ope:a6}, \ref{ope:a7} and \ref{ope:a8} as $b_1\notin V(Q_1)$ by \ref{ope:d0}. And if it is the case like \ref{ope:a6}, \ref{ope:a7}, then we define optional vertex set $R:=\{b_1\}$
    	
        If $N_D^+(b)\cap V_1=\{a_1,b_1\}$, and if $ac\in D$, then we replace $bb_1$ with $cb_1$, $aa_1$ with $ba_1$ in $G$ to get $G_\text{new}$ by Observation~\ref{obs:change initial}, which avoid \ref{(beta3)}, so we may assume $ac\notin D$.
    	
        \begin{enumerate}
            \item When there exists $u\in N_D^+(a)\cap V_1,u\notin \{a_1,b_1\}$, we replace $bb_1$ with $cb_1$, $aa_1$ with $ba_1$ in $G$ to get $G_\text{new}$ by Observation~\ref{obs:change initial}, and we relabel $a,b,c$ to fit the conditions and enter \ref{ope:d3} by Observation \ref{obs: out-vertex q1q2}.
    		
            \item\label{ope:d5(a)ii} When $N_D^+(a)\cap V_1\subseteq\{a_1,b_1\}$. If it is in the case we illustrate in Lemma~\ref{lem:counter example2}, then $D$ has no strong arc decomposition.
            
            If not, as \ref{ope:b6}, we define $R\subseteq\{a_1,b_1\}$, such that if $a_1$ satisfies $N_D^+(a_1)\neq\{a_1^+,a\}$ or $N_D^-(a_1)\neq\{a,b\}$, then $a_1\in R$ and if $b_1$ satisfies $N_D^+(b_1)\neq\{b_1^+,b\}$ or $N_D^-(b_1)\neq\{b,c\}$, then $b_1\in R$, we call $R$ optional vertex set. Note that $R\neq \emptyset$, otherwise it is in the case we illustrate in Lemma~\ref{lem:counter example2}.
      

            We choose a vertex in $R$, and do the following operation to avoid \ref{(beta3)}. Note that the choice of vertex depends on another side which is related to \ref{(beta2)} and \ref{(beta4)}
            
            As what we do in \ref{ope:d0}, we have $a_1\notin V(Q_2), b_1\notin V(Q_1)$. 	
            
            If we choose $b_1$, let $B\leftarrow\{aa_1,a_1a_1^+\}$, then we regard $b$ as $x_1$, $c$ as $x_2$, $b_1$ as $t$, and same arguments as those in \ref{ope:a6} and \ref{ope:a7}. If we choose $a_1$, let $B\leftarrow\{cb_1,b_1b_1^+\}$, then we regard $a$ as $x_1$, $b$ as $x_2$, $a_1$ as $t$, and same operations as those in \ref{ope:a6} and \ref{ope:a7}. 
    		

        \end{enumerate}
    	
        \item When $N_D^+(c)\cap V_1=\{a_1,b_1\}$:
        
        If $ac\in D$, then we replace $aa_1$ with $ca_1$ in $G$ to get $G_\text{new}$ by Observation~\ref{obs:change initial}, which avoid \ref{(beta3)}, so we may assume $ac\notin D$.

        \begin{enumerate}
            \item If there exists $u\in N_D^+(a)\cap V_1,u\notin \{a_1,b_1\}$, we replace $aa_1$ with $ca_1$ in $G$ to get $G_\text{new}$ by Observation~\ref{obs:change initial}, if \ref{(beta3)} still occurs for $G_\text{new}$, we rename $G_\text{new}$ as $G$, $B_\text{new}$ as $B$, $\{Q_3,Q_4\}$ as $\{Q_1,Q_2\}$, and we relabel $a,b,c$ to fit the conditions, so we enter \ref{ope:d3} by Observation \ref{obs: out-vertex q1q2}. 

            \item If $N_D^+(a)\cap V_1=\{a_1\}$, then we replace $aa_1$ in $D$ with $ca_1$ to get $G_\text{new}$ by Observation~\ref{obs:change initial}, if \ref{(beta3)} still occurs for $G_\text{new}$, we rename $G_\text{new}$ as $G$, $B_\text{new}$ as $B$, $\{Q_3,Q_4\}$ as $\{Q_1,Q_2\}$, and we relabel $a,b,c$ to fit the conditions, then enter \ref{ope:d5a}.

            \item If $N_D^+(a)\cap V_1=\{a_1,b_1\}$. As this is not the case we illustrate in Lemma~\ref{lem:counter example2}, then we do similar operations to those in \ref{ope:d5(a)ii}, which is:
        
            As \ref{ope:b6}, we define $R\subseteq\{a_1,b_1\}$, such that $b_1\in R$ and if $a_1$ satisfies $N_D^+(a_1)\neq\{a_1^+,c\}$ or $N_D^-(a_1)\neq\{a,c\}$, then $a_1\in R$. We call $R$ optional vertex set. 

            We choose a vertex in $R$, and do the following operation to avoid \ref{(beta3)}. Note that the choice of vertex depends on another side which is related to \ref{(beta2)} and \ref{(beta4)}
            
            As what we do in \ref{ope:d0}, we have $a_1\notin V(Q_2), b_1\notin V(Q_1)$. 	
            
            If we choose $b_1$, let $B\leftarrow\{aa_1,a_1a_1^+\}$, then we regard $b$ as $x_1$, $c$ as $x_2$, $b_1$ as $t$, and same arguments as those in \ref{ope:a6} and \ref{ope:a7}. If we choose $a_1$, replace $aa_1$ with $ca_1$ in $G$, let $B\leftarrow\{bb_1,b_1b_1^+\}$, then we regard $c$ as $x_1$, $a$ as $x_2$, $a_1$ as $t$, and do the same operations as those in \ref{ope:a6} and \ref{ope:a7}. 
        \end{enumerate}
    \end{enumerate}
\end{enumerate}

\begin{remark}
If a loop occurs, it will involve passing through \ref{ope:d3}. Since we can only enter  \ref{ope:d3} a finite number of times, there will be no infinite loop. This is ensured because each entry into \ref{ope:d3} decreases $|A_{Q_3,Q_4}(B_\text{new})|$ compared to $|A_{Q_1,Q_2}(B)|$ if it is not the final entry. Specifically, $|A_{Q_1,Q_2}(B)|$ does not increase at any step in this paper.

\end{remark}

\subsubsection{$|C_p|\geq 4$}

If \ref{(beta1)} or \ref{(beta3)} occurs, then it can only be \ref{(beta3)}. In this case, $X=U_\ell$. Like Lemma~\ref{lem:cut-arc}, if there is cut-arc $xy\in A(D[C_p])$, then there is no path in form $x\rightarrow X\stackrel{Q_j}{\longrightarrow} Y\rightarrow y$ in $G\setminus xy$, where $xy\notin Q_j$ for some $j\in[2]$. Observe that there is always a path $Y\rightarrow y$ as $Y\subseteq C_1,y\in C_p$, so by~Lemma~\ref{lem:cut-arc}, we have $X=U_{\ell}=\{x_1\},$ $U_{\ell-1}=\{y_1\}=\{x_2\}=\{x\}$, $d^+_{\bar{G}}(x_2)=1$, then enter \ref{ope:a1}-\ref{ope:a8}.

\subsubsection{Completion of the proof}

We may use \ref{ope:b1}*-\ref{ope:d5}* 
to denote the symmetric procedures of \ref{ope:b1}-\ref{ope:d5} if \ref{(beta2)} or \ref{(beta4)} occurs, and denote by $u^*$ the corresponding vertex of $u$ ($u$ is an arbitrary vertex). If $D$ has the structure as we illustrated in Theorem~\ref{lem:counter example} or Theorem~\ref{lem:counter example2}, then $D$ has no strong arc decomposition. If $D$ has no such structure, then we do the following:

Like what we do when $D[V_2]$ is strong but not 2-arc-strong, we can finally get a 2-arc-strong $\bar{G}_\text{new}$ by repeating \ref{ope:a1}-\ref{ope:d5} and \ref{ope:a1}*-\ref{ope:d5}*. The process can come to an end by the following claim.

\begin{claim}\label{clm:end2}
    The process will eventually terminate.
\end{claim}

\begin{proof}
    If the process does not terminate, the following event will occur infinitely often: \ref{(beta1)} or \ref{(beta3)} occurs for $\bar{G}$ and \ref{(beta2)} or \ref{(beta4)} does not occur for $\bar{G}$. However, after \ref{ope:a1}-\ref{ope:d5}, \ref{(beta2)} or \ref{(beta4)} occurs for $\bar{G}_\text{new}$, that means in an infinite number of the repetitions of this event, we have $|A_{Q_3,Q_4}(B_\text{new})|=|A_{Q_1,Q_2}(B)|$, because the original $|A_{Q_1,Q_2}(B)|$ is bounded.
    However, this is impossible under the following two cases:
    \begin{enumerate}
        \item When $|C_1|=1,2$ or $|C_1|\geq 4$, if \ref{(beta2)} or \ref{(beta4)} occurs, then $d_{\bar{G}_\text{new}}^-(x)=1$ for some $x\in C_1$ or $\{x\}=C_2,|C_1|=1$. And when $|A_{Q_3,Q_4}(B_\text{new})|=|A_{Q_1,Q_2}(B)|$, we have that the in-degree of $x$ will not decrease. 

        \item When $|C_1|=\{a^*,b^*,c^*\}$, if \ref{(beta4)} does not occur for $\bar{G}$ and $|A_{Q_3,Q_4}(B_\text{new})|=|A_{Q_1,Q_2}(B)|$, we can deduce that $d_{\bar{G}_\text{new}}^-(x)\geq 2$ for all $x\in C_1$, in addition, if \ref{(beta4)} occurs for $\bar{G}_\text{new}$ then we can deduce that the terminals of $Q_3$ and $Q_4$ are same, say $a^*$, and there is no arc from $V_2\backslash C_1$ to $b^*,c^*$. When $|A_{Q_3,Q_4}(B_\text{new})|=|A_{Q_1,Q_2}(B)|$, entering \ref{ope:a1}-\ref{ope:d5} will neither decrease the in-degree for all $x\in C_1$, nor produce a new arc with both side in $C_1$, we can deduce that \ref{(beta4)} occurs for $\bar{G}$, which is a contradiction.
    \end{enumerate}
      
    Thus, we can finally obtain a 2-arc-strong $\bar{G}_\text{new}$ in a finite number of steps since the original $|A_{Q_1,Q_2}|$ is bounded.
\end{proof}

When we encounter the optional vertex set $R$ or $R^*$, we obey the following rules: 

Without loss of generality, if we encounter $R$ first, we randomly choose a vertex $x$ in $R$. If we also encounter $R^*$ later, we choose the same vertex $x$ for $R^*$ when $x\in R^*$; otherwise, we randomly choose a vertex in $R^*$. Note that once we handle $R$ ($R^*$), \ref{(beta1)} (\ref{(beta2)}) and \ref{(beta3)} (\ref{(beta4)}) will not occur again in the remaining process. This follows from the proof of Claim~\ref{clm:end}. 

If we have never added additional arcs to $D[V_2]$, then $G_\text{new}$ is a subdigraph of $D$, and for $|V_2|\geq 5$, $\bar{G}_\text{new}$ has a strong arc decomposition by Theorem~\ref{thm:multi semi SAD}, and so $D$ has a strong arc decomposition by Lemma~\ref{lem:critical AC to SAD}.

If we add additional arcs to $D[V_2]$, then by the following lemma and Lemma~\ref{lem:pending decom}, $D$ has a strong arc decomposition.

\begin{lemma}\label{lem: gnew has pending2}
	For $D$ with $|V_2| \geq 4$ and $D[V_2]$ is not strong. If additional arcs have been added to $G_\text{new}$, then we can find a pending decomposition of $D$.
\end{lemma}

\begin{proof}
    Note that sometimes we regard some vertices in the form of \ref{ope:a6} or \ref{ope:a7} to enter \ref{ope:a6} or \ref{ope:a7}, for convenience, we use the term `enter \ref{ope:a6} or \ref{ope:a7}' in the following, though we are now dealing with the non-strong part. 
    
    We can see that adding additional arcs only occurs in cases like \ref{ope:a6}, \ref{ope:a7}, \ref{ope:b5}, \ref{ope:d4}, \ref{ope:a6}*, \ref{ope:a7}*, \ref{ope:b5}* and \ref{ope:d4}*.
    
    We may first assume we always added additional arcs in \ref{ope:a1}-\ref{ope:d5} by symmetry.

    As we have added additional arcs in $D[V_2]$, such that there are two pairs of parallel arcs in the form of $x_1x_2,x_2x_1$ or three pair of parallel arcs $ab,bc,ca$ in $\bar{G}_\text{new}$, so $\bar{G}_\text{new}$ is not isomorphic to the counterexamples in Theorem~\ref{thm:multi semi SAD}. Thus, we can find a strong arc decomposition $\bar{G}_1$ and $\bar{G}_2$ of $\bar{G}_\text{new}$. By lifting all splitting arcs in $\bar{G}_1$ and $\bar{G}_2$ to obtain $G_1$ and $G_2$, which is a pending decomposition of $D$ with the additional arcs by Remark~\ref{rmk:critical is pending}. What to do next is to reallocate the arcs of $G_1$ and $G_2$ so that they are always strong. And we have the following: 
    (For convenience, here we continue the case count from Lemma~\ref{lem: gnew has pending})
    
    \noindent\textbf{Case 3:} We have added additional arcs for exact one time.
 
        \textbf{Case 3.1:} If enter~\ref{ope:a6} or \ref{ope:a7}, then we do \ref{choose (1)}-\ref{choose (4)} and the same proof in Case 1.1 and Case 1.2 in the proof of Lemma~\ref{lem: gnew has pending}.
	
        \textbf{Case 3.2:} If we enter~\ref{ope:b5}, like \ref{choose (1)}-\ref{choose (3)} of Lemma~\ref{lem: gnew has pending}, we may assume $ab_1,b_1b_1^+\in G_1, ab_2,b_2b_2^+\in G_2$. Let $G_i$ contain 2-cycle $aba$ for $i=1,2$. For any $v\in V_2\backslash\{a,b\}$, $va\in A(G_i)$ and $vb\in A(G_{3-i})$ for some $i\in [2]$ since $G_i$ contains a 2-cycle $aba$ for $i=[2]$.

        Remove the additional arcs $ab,ba$ in $G_1$, and the additional arc $ba$ in $G_2$. Add $bb_2,b_2a$ to $G_1$, $bb_1,b_1a$ to $G_2$.

        We want to show that by some reallocation of arcs in $G_1$ and $G_2$, the resulting two graphs (we still call them $G_1$ and $G_2$ for convenience) are still strong and can make up a pending decomposition.
        In the following, we only move the arcs which are incident to $a$ and $b$, in addition, for any vertex $v\in V_2\backslash\{a,b\}$, once we move $vx_i$ from $G_j$ to $G_{3-j}$, then we move $vx_{3-i}$ from $G_{3-j}$ to $G_j$ for $i,j=[2]$. After applying reallocation, we only need to check if there is an $(a,b)$-path and a $(b,a)$-path in $G_i$ for $i=[2]$ to make sure that both $G_1$ and $G_2$ remian strong.
        
        Choose parallel arc $b_1^+b$ in $G_1$, $b_1^+a$ in $G_2$, $b_2^+b$ in $G_1$, $b_2^+a$ in $G_2$, and there is no conflict if $b_1^+=b_2^+$. It's obvious that there is an $(a,b)$-path and a $(b,a)$-path in $G_i$ for $i=[2]$.
	
        \textbf{Case 3.3:} If we enter~\ref{ope:d4}, like \ref{choose (1)}-\ref{choose (3)} of Lemma~\ref{lem: gnew has pending}, we may assume $aa_1,a_1a_1^+\in G_1, ba_1,a_1b_1^+\in G_2$. Let $G_i$ contain 3-cycle $abca$ for $i=1,2$. For any $v\in V_2\backslash\{a,b,c\}$, $vv_1\in A(G_1)$ and $vv_2\in A(G_2)$ for some $v_1,v_2\in \{a,b,c\}$ since $G_i$ contains a 2-cycle $aba$ for $i=1,2$.

        Remove the additional arc $ab$ from $G_1$, and the additional arcs $bc,ca$ from $G_2$. Add $ca_1$ to $G_2$. 

        What we do next is like Case 3.2.
        
        Choose parallel arc $a_1^+b$ in $G_1$, $a_1^+a,a_1^+c$ in $G_2$. $b_1^+b$ in $G_1$, $b_1^+a,b_1^+c$ in $G_2$. It's obvious that there is an $(a,b)$-path, a $(b,c)$-path and and a $(c,a)$-path in $G_i$ for $i=[2]$.
	

	See Figure~\ref{fig:case of b5d4} as an illustration. 
	
	\begin{figure}[H]
		\centering
		\begin{minipage}[t]{0.45\linewidth}
			\centering\begin{tikzpicture}[scale=0.3]

				\filldraw[black](15,0) circle (5pt)node[label=right:$b$](b){};
				\filldraw[black](10,0) circle (5pt)node[label=left:{$a$}](a){};
				\filldraw[black](12.5,4) circle (5pt)node[label=right:{$b_1$}](b1){};
				\filldraw[black](12.5,-4) circle (5pt)node[label=right:{$b_2$}](b2){};
				\filldraw[black](3,1.5) circle (5pt)node[label=above:{$b_1^+$}](b1+){};
				\filldraw[black](3,-1.5) circle (5pt)node[label=below:{$b_2^+$}](b2+){};

				\path[draw, bend left=30, line width=0.8pt, red] (a) to (b1);
				\path[draw, bend left=-30, line width=0.8pt, red] (b2) to (a);
				\path[draw, line width=0.8pt, red] (b1) to (b1+);
				\path[draw, line width=0.8pt, red] (b) to (b2);
				
				\path[draw, bend left=30, line width=0.8pt, green] (b1) to (a);
				\path[draw, bend left=-30, line width=0.8pt, green] (a) to (b2);
				\path[draw, line width=0.8pt, green] (b) to (b1);
				\path[draw, line width=0.8pt, green] (b2) to (b2+);
				
				\path[draw, line width=0.8pt, red] (b1+) to (b);
				\path[draw, line width=0.8pt, green] (b2+) to (a);
				\path[draw, line width=0.8pt, green] (a) to (b);
			\end{tikzpicture}\caption*{\ref{ope:b5}}
		\end{minipage}
		\begin{minipage}[t]{0.45\linewidth}
			\centering\begin{tikzpicture}[scale=0.3]

				\filldraw[black](16,0) circle (5pt)node[label=right:$c$](c){};
				\filldraw[black](9,0) circle (5pt)node[label=above:{$a_1$}](a1){};
				\filldraw[black](12.5,3) circle (5pt)node[label=right:{$a$}](a){};
				\filldraw[black](12.5,-3) circle (5pt)node[label=right:{$b$}](b){};
				\filldraw[black](3,1.5) circle (5pt)node[label=above:{$a_1^+$}](a1+){};
				\filldraw[black](3,-1.5) circle (5pt)node[label=below:{$b_1^+$}](b1+){};

				\path[draw, line width=0.8pt, red] (b) to (c);
				\path[draw, line width=0.8pt, red] (c) to (a);
				\path[draw, line width=0.8pt, red] (a) to (a1);
				\path[draw, line width=0.8pt, red] (a1) to (a1+);
				
				\path[draw, line width=0.8pt, green] (a) to (b);
				\path[draw, line width=0.8pt, green] (b) to (a1);
				\path[draw, line width=0.8pt, green] (a1) to (b1+);
				\path[draw, line width=0.8pt, green] (c) to (a1);
				
				\path[draw, line width=0.8pt, red] (a1+) to (b);
				\path[draw, line width=0.8pt, green] (b1+) to (c);
				\path[draw, line width=0.8pt, green] (b1+) to (a);
			\end{tikzpicture}\caption*{\ref{ope:d4}}
		\end{minipage}
		\hfill
		\caption{An illustration of reduction in \ref{ope:b5},\ref{ope:d4}, where red arcs are in $G_1$, green arcs are in $G_2$.}
		\label{fig:case of b5d4}
	\end{figure}

    \noindent\textbf{Case 4:} We add additional arcs two times. We define $S:=\{t\}$ if we enter \ref{ope:a6} or \ref{ope:a7}, $S:=\{b_1,b_2\}$ if we enter \ref{ope:b5}, $S:=\{a_1\}$ if we enter \ref{ope:d4}, and so we can define the corresponding $S^*$.
 
    \textbf{Case 4.1:} When $S\cap S^*\neq \emptyset$:
    
    
    If we enter \ref{ope:d4}, then we must enter \ref{ope:d4}*, as there are 4 arcs incident to $a_1$ in $G_\text{new}$, and for \ref{ope:a6}* and \ref{ope:a7}*, there are only 2 arcs incident to $t^*$, for \ref{ope:b5}, there are only 2 arcs incident to both $b_1^*$ and $b_2^*$. Firstly, we do the same operations as those in Case 3.3, and then we do the same operations as those in Case 3.3*, where Case 3.3* is the symmetric operation of Case 3.3. The reason why we can do this is that after Case 3.3, $G_1$ and $G_2$ are strong, and $a_1^*c^*$ has not been used in $G_1$ and $G_2$.
    
    If we enter \ref{ope:a7}, then $S\cap S^*\neq \emptyset$ is impossible, as for any vertex $x$ in $S^*$, $x$ has at least two out-neighbors besides $x_1,x_2$, and $t$ has only one out-neighbor besides $x_1,x_2$. For a similar reason, if we enter \ref{ope:b5}, then $S\cap S^*\neq \emptyset$ is impossible.
	
    The only remained case is that we enter \ref{ope:a6} and \ref{ope:a6}*. We do the same operations as those in \ref{choose (1)}-\ref{choose (4)} and Case 2.1 in Lemma~\ref{lem: gnew has pending}.
	
    \textbf{Case 4.2:} When $S\cap S^*= \emptyset$ and we may assume we encounter $R$ first by symmetry:

    \textbf{Step 1.} We do the same arguments as those in Case 3.

    \textbf{Step 2.}  We do the same arguments as those in Case 3*, where Case 3* is the symmetric case of Case 3.

    In every case from Case 1 to Case 4, we can always get a pending decomposition $G_1$ and $G_2$ of $D$, which meets our needs.
    
    The reason why we can do the above two steps is that:

    \begin{itemize}
        \item If $G_1$ and $G_2$ are strong, and $a_1^*c^*$ has not been used in $G_1$ and $G_2$, then we can enter Case 3.3*. This is guaranteed by $S\cap S^*= \emptyset$.

        \item If $G_1$ and $G_2$ are strong, and $b_2^*b^*, a^*b_2^*, b_1^*b^*, a^*b_1^*$ have not been used in $G_1$ and $G_2$, then we can enter Case 3.2*. This is guaranteed by $S\cap S^*= \emptyset$.

        \item  If $G_1$ and $G_2$ satisfy \ref{i)})*, \ref{ii)})*, \ref{iii)})* in Remark~\ref{rmk:legal}, we can enter Case 3.1*. \ref{i)})* is guaranteed by $G_1$ and $G_2$ are strong. \ref{iii)})* is guaranteed by $S\cap S^*= \emptyset$. We may assume we choose $x$ in $R$, as $S\cap S^*= \emptyset$, so $x\notin R^*$ and this guarantees \ref{ii)})*.
    \end{itemize}
\end{proof}

\begin{lemma}\label{lem:counter example2}
	Let $D=(V_1,V_2;A)$ be a 2-arc-strong split digraph with $|V_2(D)|\geq 4$ and $D[V_2]$ is not strong, which means it has the acyclic ordering of its strong component $C_1,\ldots,C_p$ ($p\geq 2$). If $D$ has a copy of at least one of the following structures, then $D$ has no strong arc decomposition.
	\begin{itemize}
		\item $D[C_p]$ is a 3-cycle, say $abca$, and there exists $u,v\in V_1$, such that $N_D^+(b)=\{u,v,c\}, N_D^+(c)=\{v,a\}, N_D^+(a)=\{u,b\}, N_D^+(u)=\{a,u^+\}, N_D^+(v)=\{b,v^+\}$, where $u^+,v^+\in V_2\setminus C_p$, and they can be the same one, besides, $N_D^-(u)=\{a,b\}, N_D^-(v)=\{b,c\}$.
		
		\item Reversing arcs in the first case.
	\end{itemize}
\end{lemma}

\begin{proof}
	If the first case occurs, suppose to the contrary that $D$ has a strong arc decomposition $D_1, D_2$ and we may assume $au\in D_1$ and $ab\in D_2$ as $d_D^+(a)=2$. Since $d_D^-(u)=2,d_D^+(u)=2$, we have $uu^+\in D_1, bu,ua\in D_2$, otherwise there is no out-arc of $\{a,u\}$ in $D_1$. Besides, we have $vv^+\in D_2$ as there is a $(\{a,b,c\},V_2\setminus\{a,b,c\})$-path in $D_2$. Consider the $(\{a,b,u\},v)$-path in $D_2$, if it is $bcv$, then there is no $(b,c)$-path in $D_1$, so it must be $bv\in D_2$ and $cv,vb\in D_1$ as $d_D^-(v)=2,d_D^+(v)=2$. Besides, $ca\in D_2$ as $d_D^+(c)=2$. Now, there is no $(c,a)$-path in $D_1$ no matter how we distribute other arcs, which contradicts the fact that $D_1$ is strong. So, $D$ has no strong arc decomposition. By the same arguments, we can prove the second case.
\end{proof}

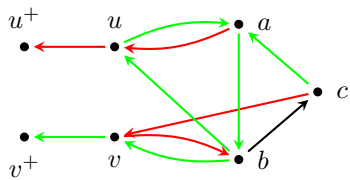
\begin{figure}[H]
    \centering
    \begin{tikzpicture}[scale=0.3]

				\filldraw[black](16,0) circle (5pt)node[label=right:$c$](c){};
				\filldraw[black](12.5,3) circle (5pt)node[label=right:{$a$}](a){};
				\filldraw[black](12.5,-3) circle (5pt)node[label=right:{$b$}](b){};
				\filldraw[black](7,2) circle (5pt)node[label=above:{$u$}](u){};
				\filldraw[black](7,-2) circle (5pt)node[label=below:{$v$}](v){};
                \filldraw[black](3,2) circle (5pt)node[label=above:{$u^+$}](u+){};
				\filldraw[black](3,-2) circle (5pt)node[label=below:{$v^+$}](v+){};
				
				\path[draw, line width=0.8pt, red, bend left=15] (a) to (u);
				\path[draw, line width=0.8pt, red] (u) to (u+);
				\path[draw, line width=0.8pt, red] (c) to (v);
				\path[draw, line width=0.8pt, red, bend left=15] (v) to (b);
				
				\path[draw, line width=0.8pt, green] (a) to (b);
				\path[draw, line width=0.8pt, green] (b) to (u);
				\path[draw, line width=0.8pt, green, bend left=15] (u) to (a);
				\path[draw, line width=0.8pt, green] (v) to (v+);
                \path[draw, line width=0.8pt, green, bend left=15] (b) to (v);
                \path[draw, line width=0.8pt, green] (c) to (a);
				
				\path[draw, line width=0.8pt] (b) to (c);
			\end{tikzpicture}
    \caption{An illustration of Proof of Lemma~\ref{lem:counter example2}}
    \label{fig:counter2}
\end{figure}

\section{Concluding remarks}
The proofs of Theorem~\ref{thm:2as} when $|V_2|\leq 4$ are included in the appendix. These cases are somewhat tedious but not particularly difficult.

By Theorem~\ref{thm:2as}, we can conclude the following interesting result.

\begin{corollary}\label{cor:add arc}
	For any 2-arc-strong split graph $D=(V_1,V_2;A)$, it has a strong arc decomposition by adding at most 1 specific arc in $D[V_2]$.
\end{corollary}

\newpage
\appendix
	
{\bf \huge Appendix} 

\section{Proof of Theorem~\ref{thm:2as} when $|V_2|\leq3$}

Let $D = (V_1,V_2;A)$ be a 2-arc-strong split digraph. We say $\{V_1,V_2\}$ is a {\em maximal partition} if there is no vertex $x$ in $V_1$ such that $N(x)=V_2$. If $V_1 = \emptyset$, then $D$ is a semicomplete digraph which is reduced to Theorem~\ref{thm:semi SAD}. Now, suppose $V_1 \neq \emptyset$. Since each vertex in $V_1$ has out-degree and in-degree at least 2, we have $|V_2| \geq 2$. When $|V_2| = 2$, say $V_2 = \{u, v\}$, for any vertex $x\in V_1$, the digraph must contain $2$-cycles $xux$ and $xvx$. This implies that we can find a new vertex partition $V(D) = V'_1 \cup V'_2$ where $V'_2 = V_2 \cup \{x\}$ for some vertex $x \in V_1$ and $V'_1 = V(D) \setminus V'_2$, such that $V'_1$ is an independent set and the subdigraph induced by $V'_2$ is semicomplete. Thus, the case when $|V_2| = 2$ is reduced to the case when $|V_2| = 3$ with maximal partition $V(D) = V_1 \cup V_2$.

\begin{proposition}
    Let $D = (V_1,V_2;A)$ be a 2-arc-strong split multi-digraph with maximal partition $V(D) = V_1 \cup V_2$ such that $V_1$ is an independent set, $V_2$ induces a semicomplete multi-digraph and there is no multi-arc between $V_1$ and $V_2$. If $|V_2| = 3$, then $D$ has a strong arc decomposition.
\end{proposition}

\begin{proof}
    Let $V_2 = \{u,v,w\}$. We prove this result by induction on  $|V_1|$. When $|V_1| = 0$, the result follows directly from Theorem~\ref{thm:multi semi SAD}. Now, assume that $|V_1| = n \geq 1$ and that the result holds for all 2-arc-strong split multi-digraphs with maximal partition $V(D) = V_1 \cup V_2$ and $|V_1| \leq n-1$. Choose a vertex $x \in V_1$, there is exactly one vertex in $V_2$, say $w$, which is not adjacent to $x$, since partition $V(D) = V_1 \cup V_2$ is maximal. And so the vertex $x$ dominates $u,v$ and is also dominated by $u,v$. We split off the pair $(ux,xv)$ at $x$ to obtain a  splitting arc $\Tilde{uv}$, and similarly, we split off the pair $(vx,xu)$ at $x$ to obtain a splitting arc $\Tilde{vu}$. Let $D_x$ denote the new multi-digraph obtained after these two splitting-off. Note that $d^{+}_{D_x}(x) = d^{-}_{D_x}(x) = 0$. Let $D'= D_x - x$. Let $V'_1 = V_1 \setminus \{x\}$, then $V(D') = V'_1 \cup V_2$ is also maximal partition of $D'$. 

    \begin{claim}
        The multi-digraph $D'$ is 2-arc-strong. 
    \end{claim}

    \begin{proof}
        Let $a$ and $b$ be two distinct vertices in $V(D')$, and thus also in $V(D)$. Then there are two arc-disjoint paths from $a$ to $b$ in $D$, denoted as $P_1$ and $P_2$.
        If path $P_i$ (where $i = 1,2$) contains $x$, then $P_i$ must include both $u$ and $v$ since $N^{-}(x) = N^{+}(x) = \{u,v\}$. We can then replace the segment $uxv$ on path $P_i$ with the corresponding splitting arc to get a new path $P'_i$ from $a$ to $b$ in $D'$. Thus, we find two arc-disjoint paths from $a$ to $b$ in $D'$.
    \end{proof}
    
    Hence, $D'$ has a strong arc decomposition $A(D') = B_1 \cup B_2$ by induction hypothesis. Now we consider the distribution of the splitting arcs $\Tilde{uv}$ and $\Tilde{vu}$. If both $B_1$ and $B_2$ have exactly one splitting arc, without loss of generality, let $\Tilde{uv}$ lie in $B_1$ and $\Tilde{vu}$ lie in $B_2$. Then $B'_1 := \left(B_1 \setminus \{\Tilde{uv}\}\right) \cup \{ux,xv\}$ and $B'_2 := \left(B_2 \setminus \{\Tilde{vu}\}\right) \cup \{vx,xu\}$ will form a strong arc decomposition of $D$. 
    
    If these two splitting arcs $\Tilde{uv}$ and $\Tilde{vu}$ lie in the same set, say $B_1$. Since $D[V_2]$ is semicomplete, either $uv$ or $vu$ lies in $A(D[V_2])$. Without loss of generality, suppose $uv \in A(D[V_2])$. If $uv \in B_1$, then $D_x[B_1 \setminus \{\Tilde{uv}\}]$ is still strong. This yields that $A(D_x) = B'_1 \cup B'_2$, in which $B'_1 = B_1 \setminus \{\Tilde{uv}\}$ and $B'_2 = B_2 \cup \{\Tilde{uv}\}$, is a strong arc decomposition such that each set has exactly one splitting arc and so we are done. If $uv \in B_2$, then $A(D_x) = B'_1 \cup B'_2$, in which $B'_1 = (B_1 \setminus \{\Tilde{uv}\}) \cup \{uv\}$ and $B'_2 = (B_2 \setminus \{uv\}) \cup \{\Tilde{uv}\}$, is a strong arc decomposition such that each set has exactly one splitting arc and so we are done.
\end{proof}

\begin{corollary}\label{cor:3}
    Let $D = (V_1,V_2;A)$ be a 2-arc-strong split digraph with maximal partition $V(D) = V_1 \cup V_2$ such that $V_1$ is an independent set and $V_2$ induces a semicomplete digraph. If $|V_2| = 3$, then $D$ has a strong arc decomposition.
\end{corollary}

\section{Proof of Theorem~\ref{thm:2as} when $|V_2|=4$}

We may only consider the case when each vertex in $V_1$ is adjacent to at most $3$ vertices in $V_2$. 
Since when there is a vertex adjacent to all vertices in $V_2$, then it can be viewed as a split digraph with $|V_2|=5$ or a semicomplete digraph on $5$ vertices, which has been previously discussed.

In the previous proof when $|V_2|\geq 5$, the condition $|V_2|\geq 5$ instead of $|V_2|\geq 4$ is specifically used to avoid certain configurations. The key role of this condition is to ensure that the new digraph $G_\text{new}$ (or original $G$, or $D[V_2]$) remains 2-arc-strong without adding additional arcs. We use $|V_2|\geq 5$ to avoid the case that $\bar{G}_\text{new}$ (or original $\bar{G}$) is isomorphic to one of the seven graphs shown in Theorem~\ref{thm:multi semi SAD}. So, here we only need to focus on this case.

Note that after removing parallel arcs in the seven graphs, each of them is isomorphic to $S_4$, so the semicomplete digraph $D[V_2]$ must be a subdigraph of $S_4$, which implies that there is a 4-circle $v_1v_2v_3v_4v_1$ in $D[V_2]$. Considering isomorphism, $D[V_2]$ can only be one of the following three digraphs. We only focus on the cases where $D$ has no strong arc decomposition.

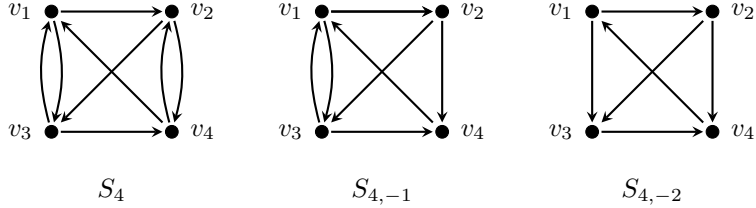
\begin{figure}[H]
	\centering
	\subfigure{\begin{minipage}[t]{0.23\linewidth}
			\centering\begin{tikzpicture}[scale=0.8]
				\filldraw[black](0,0) circle (3pt)node[label=left:$v_1$](v1){};
				\filldraw[black](2,0) circle (3pt)node[label=right:$v_2$](v2){};
				\filldraw[black](0,-2) circle (3pt)node[label=left:$v_3$](v3){};
				\filldraw[black](2,-2) circle (3pt)node[label=right:$v_4$](v4){};
				\foreach \i/\j/\t in {
					v1/v2/0,
					v2/v3/0,
					v3/v4/0,
					v4/v1/0,
					v1/v3/15,
					v2/v4/15,
					v3/v1/15,
					v4/v2/15
				}{\path[draw, line width=0.8] (\i) edge[bend left=\t] (\j);}			
			\end{tikzpicture}\caption*{$S_4$}\end{minipage}}
	\subfigure{\begin{minipage}[t]{0.23\linewidth}
			\centering\begin{tikzpicture}[scale=0.8]
				\filldraw[black](0,0) circle (3pt)node[label=left:$v_1$](v1){};
				\filldraw[black](2,0) circle (3pt)node[label=right:$v_2$](v2){};
				\filldraw[black](0,-2) circle (3pt)node[label=left:$v_3$](v3){};
				\filldraw[black](2,-2) circle (3pt)node[label=right:$v_4$](v4){};
				\foreach \i/\j/\t in {
					v1/v2/0,
					v1/v2/0,
					v2/v3/0,
					v3/v4/0,
					v4/v1/0,
					v1/v3/15,
					v2/v4/0,
					v3/v1/15
				}{\path[draw, line width=0.8] (\i) edge[bend left=\t] (\j);}	
			\end{tikzpicture}\caption*{$S_{4,-1}$}\end{minipage}}
	\subfigure{\begin{minipage}[t]{0.23\linewidth}
			\centering\begin{tikzpicture}[scale=0.8]
				\filldraw[black](0,0) circle (3pt)node[label=left:$v_1$](v1){};
				\filldraw[black](2,0) circle (3pt)node[label=right:$v_2$](v2){};
				\filldraw[black](0,-2) circle (3pt)node[label=left:$v_3$](v3){};
				\filldraw[black](2,-2) circle (3pt)node[label=right:$v_4$](v4){};
				\foreach \i/\j/\t in {
					v1/v2/0,
					v2/v3/0,
					v3/v4/0,
					v4/v1/0,
					v1/v3/0,
					v2/v4/0
				}{\path[draw, line width=0.8] (\i) edge[bend left=\t] (\j);}	
			\end{tikzpicture}\caption*{$S_{4,-2}$}\end{minipage}}
\end{figure}

\begin{lemma}\label{lem:parallel}
    If $D[V_2]+C$ has a strong arc decomposition, then $D$ has a strong arc decomposition, where $C$ is a splitting arc set, and for any vertex $x$ in $V_1$, $x$ satisfies one of the following
    \begin{itemize}
        \item There are at most one splitting arc in $C$ obtained by splitting-off an arc pair at $x$.

        \item At least one splitting arc in $C$ obtained by splitting-off an arc pair at $x$ has a parallel arc in $D[V_2]$.
    \end{itemize}
\end{lemma}

\begin{proof}
    We may assume $D[V_2]+C$ has a strong arc decomposition $D_1'$ and $D_2'$. Fix an vertex $x\in V_1$, if there are more than two splitting arcs obtained by splitting-off arc pairs at $x$ in $C$, say $\Tilde{ab}$ and $\Tilde{cd}$, where $\Tilde{ab}$ is the splitting arc with parallel arc $ab\in A(D[V_2])$, then we can distribute $\Tilde{ab}$ and $\Tilde{cd}$ to different $D_i',i\in [2]$. This is because if $\Tilde{ab},\Tilde{cd}\in A(D_i')$, then we can remove $\Tilde{ab}$ to $A(D_{3-i}')$ and remove $ab$ to $A(D_i')$ if $ab\in A(D_{3-i}')$, and we can check $D_1'$ and $D_2'$ are still strong. We do this distribution step by step for all $x\in V_1$. Now we lift all splitting arcs in $D_1'$ and $D_2'$ to get $D_1$ and $D_2$, we can check $D_1$ and $D_2$ form a pending decomposition of $D$, therefore, $D$ has a strong arc decomposition by Lemma~\ref{lem:pending decom}.
\end{proof}

\subsection{$D[V_2]$ is $S_{4}$.}

If $V_1=\emptyset$, then $D=S_4$ which has no strong arc decomposition. Now we will prove that if $|V_1|\geq 1$, then $D$ has a strong arc decomposition. Suppose for contradiction that $D$ does not have a strong arc decomposition. There exists $a\in V_1$, then $D[V_2]+\{a^-a^+\}$ has no strong arc decomposition by Lemma~\ref{lem:parallel}, where $a^-a^+$ is obtained by splitting-off the arc pair $(a^-a,aa^+)$ at $a$. So $a^-a^+=v_3v_1$ or $v_1v_2$ in the sense of symmetry by Theorem~\ref{thm:multi semi SAD}. Now we consider the other arcs adjacent to $a$ in $D$.

\begin{itemize}
	\item We assume $a^-a^+=v_3v_1$. If $av_2\in D$, then $D[V_2]+\{v_3v_2\}$ has a strong arc decomposition. If $v_2a\in D$, then $D[V_2]+\{v_2v_1\}$ has a strong arc decomposition. By Lemma~\ref{lem:parallel}, $D$ has a strong arc decomposition. Thus, $av_2,v_2a\notin D$. Then, $av_4\in D$ or $av_3\in D$ as $d_D^+(a)\geq 2$; $v_1a\in D$ or $v_3a\in D$ as $d_D^-(a)\geq 2$. 
	
	If $av_4\in D$ and $v_1a\in D$, then $D[V_2]+\{v_1v_4\}$ has a strong arc decomposition. If $av_4\in D$ and $v_4a\in D$, then $D[V_2]+\{v_4v_1,v_3v_4\}$ has a strong arc decomposition.
	
	If $av_3\in D$ and $v_1a\in D$, then $D[V_2]+\{v_1v_3,v_3v_1\}$ has a strong arc decomposition. If $av_3\in D$ and $v_4a\in D$, then $D[V_2]+\{v_4v_3\}$ has a strong arc decomposition.

    In conclusion, by Lemma~\ref{lem:parallel}, $D$ has a strong arc decomposition when $a^-a^+=v_3v_1$, a contradiction.
	
	\item We assume that $a^-a^+=v_1v_2$. It is a contradiction by a similar argument. 
\end{itemize}
So, we have that $a^-a^+\ne v_3v_1$ or $v_1v_2$, which is a contradiction. Thus, if $D[V_2]$ is isomorphic to $S_4$, $D$ has a strong arc decomposition if $|V_1|\geq 1$.

\begin{center}
	\rule{\linewidth}{2pt}
\end{center}

\subsection{$D[V_2]$ is $S_{4,-1}$.}

Since $S_4$ is a subdigraph of $\bar{G}_\text{new}$ (or $\bar{G}$), then there exists a vertex $a\in v_1$ such that $v_4a,av_2\in D$. If there is another vertex $b\in V_1$, then $D$ has a strong arc decomposition by splitting off $(v_4a,av_2)$ at $a$, and applying the proof for the case
where $D[V_2]$ is $S_4$. Therefore, we only need to consider the case where $V_1=\{a\}$.

\begin{center}
	\rule{\linewidth}{1.2pt}
\end{center}

When $a$ is only adjacent to $v_2$ and $v_4$:

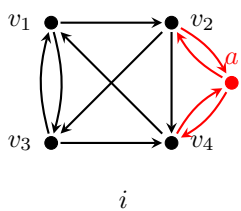
\begin{figure}[H]
	\begin{minipage}[t]{0.23\linewidth}
		\vspace{0pt}
		\centering
		\begin{tikzpicture}[scale=0.8]
			\filldraw[black](0,0) circle (3pt)node[label=left:$v_1$](v1){};
			\filldraw[black](2,0) circle (3pt)node[label=right:$v_2$](v2){};
			\filldraw[black](0,-2) circle (3pt)node[label=left:$v_3$](v3){};
			\filldraw[black](2,-2) circle (3pt)node[label=right:$v_4$](v4){};
			
			\filldraw[red](3,-1) circle (3pt)node[label=above:$a$](a){};
			
			\foreach \i/\j/\t in {
				v1/v2/0,
				v2/v3/0,
				v3/v4/0,
				v4/v1/0,
				v1/v3/15,
				v3/v1/15,
				v2/v4/0
			}{\path[draw, line width=0.8] (\i) edge[bend left=\t] (\j);}
			
			\foreach \i/\j/\t in {
				a/v2/15,
				a/v4/15,
				v4/a/15,
				v2/a/15
			}{\path[draw, red, line width=0.8] (\i) edge[bend left=\t] (\j);}
		\end{tikzpicture}
		\caption*{$i$}
	\end{minipage}
	\hfill
	\begin{minipage}[t]{0.73\linewidth}
		\vspace{0pt}
		It has no strong arc decomposition as $(iv)^*\times (iv)$ has no strong arc decomposition.
	\end{minipage}
\end{figure}

\begin{center}
	\rule{\linewidth}{1.2pt}
\end{center}

When $a$ is adjacent to $v_2$,$v_4$ and $v_1$:

\noindent\rule{0.5\linewidth}{0.4pt}

If $av_1, v_1a\in D$, $D$ has a strong arc decomposition as $D[V_2]+\{v_4v_1,v_1v_2\}$ has a strong arc decomposition by Lemma~\ref{lem:parallel}.

\noindent\rule{0.5\linewidth}{0.4pt}

If $av_1\in D, v_1a\notin D$:

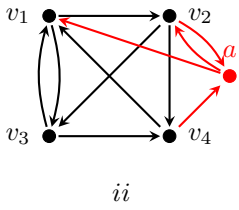
\begin{figure}[H]
	\begin{minipage}[t]{0.23\linewidth}
		\vspace{0pt}
		\centering
		\begin{tikzpicture}[scale=0.8]
			\filldraw[black](0,0) circle (3pt)node[label=left:$v_1$](v1){};
			\filldraw[black](2,0) circle (3pt)node[label=right:$v_2$](v2){};
			\filldraw[black](0,-2) circle (3pt)node[label=left:$v_3$](v3){};
			\filldraw[black](2,-2) circle (3pt)node[label=right:$v_4$](v4){};
			
			\filldraw[red](3,-1) circle (3pt)node[label=above:$a$](a){};
			
			\foreach \i/\j/\t in {
				v1/v2/0,
				v2/v3/0,
				v3/v4/0,
				v4/v1/0,
				v1/v3/15,
				v3/v1/15,
				v2/v4/0
			}{\path[draw, line width=0.8] (\i) edge[bend left=\t] (\j);}
			
			\foreach \i/\j/\t in {
				a/v2/15,
				v2/a/15,
				v4/a/0,
				a/v1/0
			}{\path[draw, red, line width=0.8] (\i) edge[bend left=\t] (\j);}
		\end{tikzpicture}
		\caption*{$ii$}
	\end{minipage}
	\hfill
	\begin{minipage}[t]{0.73\linewidth}
		\vspace{0pt}
		It has no strong arc decomposition as $(ii)^*\times (ii)$ has no strong arc decomposition.
	\end{minipage}
\end{figure}

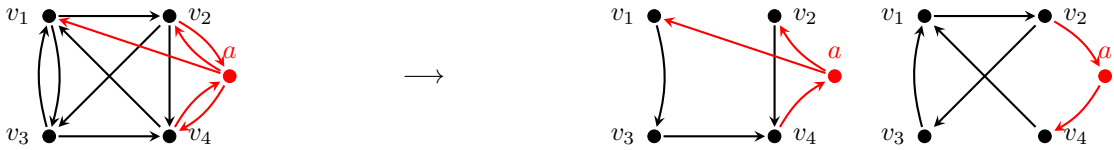
\begin{figure}[H]
	\begin{minipage}{0.23\linewidth}
		\vspace{0pt}
		\centering
		\begin{tikzpicture}[scale=0.8]
			\filldraw[black](0,0) circle (3pt)node[label=left:$v_1$](v1){};
			\filldraw[black](2,0) circle (3pt)node[label=right:$v_2$](v2){};
			\filldraw[black](0,-2) circle (3pt)node[label=left:$v_3$](v3){};
			\filldraw[black](2,-2) circle (3pt)node[label=right:$v_4$](v4){};
			
			\filldraw[red](3,-1) circle (3pt)node[label=above:$a$](a){};
			
			\foreach \i/\j/\t in {
				v1/v2/0,
				v2/v3/0,
				v3/v4/0,
				v4/v1/0,
				v1/v3/15,
				v3/v1/15,
				v2/v4/0
			}{\path[draw, line width=0.8] (\i) edge[bend left=\t] (\j);}
			
			\foreach \i/\j/\t in {
				a/v2/15,
				v2/a/15,
				v4/a/15,
				a/v4/15,
				a/v1/0
			}{\path[draw, red, line width=0.8] (\i) edge[bend left=\t] (\j);}
		\end{tikzpicture}
	\end{minipage}
	\hfill
	$\longrightarrow$
	\hfill
	\begin{minipage}{0.23\linewidth}
		\vspace{0pt}
		\centering
		\begin{tikzpicture}[scale=0.8]
			\filldraw[black](0,0) circle (3pt)node[label=left:$v_1$](v1){};
			\filldraw[black](2,0) circle (3pt)node[label=right:$v_2$](v2){};
			\filldraw[black](0,-2) circle (3pt)node[label=left:$v_3$](v3){};
			\filldraw[black](2,-2) circle (3pt)node[label=right:$v_4$](v4){};
			
			\filldraw[red](3,-1) circle (3pt)node[label=above:$a$](a){};
			
			\foreach \i/\j/\t in {
				v3/v4/0,
				v1/v3/15,
				v2/v4/0
			}{\path[draw, line width=0.8] (\i) edge[bend left=\t] (\j);}
			
			\foreach \i/\j/\t in {
				a/v2/15,
				v4/a/15,
				a/v1/0
			}{\path[draw, red, line width=0.8] (\i) edge[bend left=\t] (\j);}
		\end{tikzpicture}
	\end{minipage}
	\begin{minipage}{0.23\linewidth}
		\vspace{0pt}
		\centering
		\begin{tikzpicture}[scale=0.8]
			\filldraw[black](0,0) circle (3pt)node[label=left:$v_1$](v1){};
			\filldraw[black](2,0) circle (3pt)node[label=right:$v_2$](v2){};
			\filldraw[black](0,-2) circle (3pt)node[label=left:$v_3$](v3){};
			\filldraw[black](2,-2) circle (3pt)node[label=right:$v_4$](v4){};
			
			\filldraw[red](3,-1) circle (3pt)node[label=above:$a$](a){};
			
			\foreach \i/\j/\t in {
				v1/v2/0,
				v2/v3/0,
				v4/v1/0,
				v3/v1/15
			}{\path[draw, line width=0.8] (\i) edge[bend left=\t] (\j);}
			
			\foreach \i/\j/\t in {
				v2/a/15,
				a/v4/15
			}{\path[draw, red, line width=0.8] (\i) edge[bend left=\t] (\j);}
		\end{tikzpicture}
	\end{minipage}
\end{figure}

\noindent\rule{0.5\linewidth}{0.4pt}

If $av_1\notin D, v_1a\in D$:

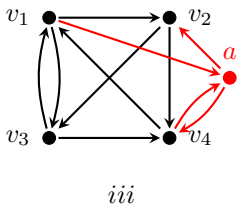
\begin{figure}[H]
	\begin{minipage}[t]{0.23\linewidth}
		\vspace{0pt}
		\centering
		\begin{tikzpicture}[scale=0.8]
			\filldraw[black](0,0) circle (3pt)node[label=left:$v_1$](v1){};
			\filldraw[black](2,0) circle (3pt)node[label=right:$v_2$](v2){};
			\filldraw[black](0,-2) circle (3pt)node[label=left:$v_3$](v3){};
			\filldraw[black](2,-2) circle (3pt)node[label=right:$v_4$](v4){};
			
			\filldraw[red](3,-1) circle (3pt)node[label=above:$a$](a){};
			
			\foreach \i/\j/\t in {
				v1/v2/0,
				v2/v3/0,
				v3/v4/0,
				v4/v1/0,
				v1/v3/15,
				v3/v1/15,
				v2/v4/0
			}{\path[draw, line width=0.8] (\i) edge[bend left=\t] (\j);}
			
			\foreach \i/\j/\t in {
				a/v2/0,
				a/v4/15,
				v4/a/15,
				v1/a/0
			}{\path[draw, red, line width=0.8] (\i) edge[bend left=\t] (\j);}
		\end{tikzpicture}
		\caption*{$iii$}
	\end{minipage}
	\hfill
	\begin{minipage}[t]{0.73\linewidth}
		\vspace{0pt}
		It has no strong arc decomposition as $(iii)^*\times (iii)$ has no strong arc decomposition. ($iii$ is isomorphic to $ii$)
	\end{minipage}
\end{figure}

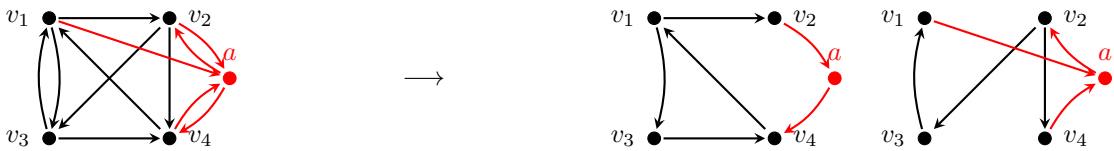
\begin{figure}[H]
	\begin{minipage}{0.23\linewidth}
		\vspace{0pt}
		\centering
		\begin{tikzpicture}[scale=0.8]
			\filldraw[black](0,0) circle (3pt)node[label=left:$v_1$](v1){};
			\filldraw[black](2,0) circle (3pt)node[label=right:$v_2$](v2){};
			\filldraw[black](0,-2) circle (3pt)node[label=left:$v_3$](v3){};
			\filldraw[black](2,-2) circle (3pt)node[label=right:$v_4$](v4){};
			
			\filldraw[red](3,-1) circle (3pt)node[label=above:$a$](a){};
			
			\foreach \i/\j/\t in {
				v1/v2/0,
				v2/v3/0,
				v3/v4/0,
				v4/v1/0,
				v1/v3/15,
				v3/v1/15,
				v2/v4/0
			}{\path[draw, line width=0.8] (\i) edge[bend left=\t] (\j);}
			
			\foreach \i/\j/\t in {
				a/v2/15,
				v2/a/15,
				v4/a/15,
				a/v4/15,
				v1/a/0
			}{\path[draw, red, line width=0.8] (\i) edge[bend left=\t] (\j);}
		\end{tikzpicture}
	\end{minipage}
	\hfill
	$\longrightarrow$
	\hfill
	\begin{minipage}{0.23\linewidth}
		\vspace{0pt}
		\centering
		\begin{tikzpicture}[scale=0.8]
			\filldraw[black](0,0) circle (3pt)node[label=left:$v_1$](v1){};
			\filldraw[black](2,0) circle (3pt)node[label=right:$v_2$](v2){};
			\filldraw[black](0,-2) circle (3pt)node[label=left:$v_3$](v3){};
			\filldraw[black](2,-2) circle (3pt)node[label=right:$v_4$](v4){};
			
			\filldraw[red](3,-1) circle (3pt)node[label=above:$a$](a){};
			
			\foreach \i/\j/\t in {
				v3/v4/0,
				v1/v3/15,
				v4/v1/0,
				v1/v2/0
			}{\path[draw, line width=0.8] (\i) edge[bend left=\t] (\j);}
			
			\foreach \i/\j/\t in {
				v2/a/15,
				a/v4/15
			}{\path[draw, red, line width=0.8] (\i) edge[bend left=\t] (\j);}
		\end{tikzpicture}
	\end{minipage}
	\begin{minipage}{0.23\linewidth}
		\vspace{0pt}
		\centering
		\begin{tikzpicture}[scale=0.8]
			\filldraw[black](0,0) circle (3pt)node[label=left:$v_1$](v1){};
			\filldraw[black](2,0) circle (3pt)node[label=right:$v_2$](v2){};
			\filldraw[black](0,-2) circle (3pt)node[label=left:$v_3$](v3){};
			\filldraw[black](2,-2) circle (3pt)node[label=right:$v_4$](v4){};
			
			\filldraw[red](3,-1) circle (3pt)node[label=above:$a$](a){};
			
			\foreach \i/\j/\t in {
				v2/v4/0,
				v2/v3/0,
				v3/v1/15
			}{\path[draw, line width=0.8] (\i) edge[bend left=\t] (\j);}
			
			\foreach \i/\j/\t in {
				v4/a/15,
				a/v2/15,
				v1/a/0
			}{\path[draw, red, line width=0.8] (\i) edge[bend left=\t] (\j);}
		\end{tikzpicture}
	\end{minipage}
\end{figure}

\begin{center}
	\rule{\linewidth}{1.2pt}
\end{center}

When $a$ is adjacent to $v_2$,$v_4$ and $v_3$:

\noindent\rule{0.5\linewidth}{0.4pt}

If $v_3a,av_3\in D$:

\begin{figure}[H]
	\begin{minipage}{0.23\linewidth}
		\vspace{0pt}
		\centering
		\begin{tikzpicture}[scale=0.8]
			\filldraw[black](0,0) circle (3pt)node[label=left:$v_1$](v1){};
			\filldraw[black](2,0) circle (3pt)node[label=right:$v_2$](v2){};
			\filldraw[black](0,-2) circle (3pt)node[label=left:$v_3$](v3){};
			\filldraw[black](2,-2) circle (3pt)node[label=right:$v_4$](v4){};
			
			\filldraw[red](3,-1) circle (3pt)node[label=above:$a$](a){};
			
			\foreach \i/\j/\t in {
				v1/v2/0,
				v2/v3/0,
				v3/v4/0,
				v4/v1/0,
				v1/v3/15,
				v3/v1/15,
				v2/v4/0
			}{\path[draw, line width=0.8] (\i) edge[bend left=\t] (\j);}
			
			\foreach \i/\j/\t in {
				a/v2/15,
				v3/a/10,
				v4/a/15,
				a/v3/10
			}{\path[draw, red, line width=0.8] (\i) edge[bend left=\t] (\j);}
			
			\path[draw, dashed, red, line width=0.8] (a) edge[bend left=15] (v4);
			
			\path[draw, dashed, red, line width=0.8] (v2) edge[bend left=15] (a);
		\end{tikzpicture}
	\end{minipage}
	\hfill
	$\longrightarrow$
	\hfill
	\begin{minipage}{0.23\linewidth}
		\vspace{0pt}
		\centering
		\begin{tikzpicture}[scale=0.8]
			\filldraw[black](0,0) circle (3pt)node[label=left:$v_1$](v1){};
			\filldraw[black](2,0) circle (3pt)node[label=right:$v_2$](v2){};
			\filldraw[black](0,-2) circle (3pt)node[label=left:$v_3$](v3){};
			\filldraw[black](2,-2) circle (3pt)node[label=right:$v_4$](v4){};
			
			\filldraw[red](3,-1) circle (3pt)node[label=above:$a$](a){};
			
			\foreach \i/\j/\t in {
				v1/v2/0,
				v3/v1/15,
				v2/v4/0
			}{\path[draw, line width=0.8] (\i) edge[bend left=\t] (\j);}
			
			\foreach \i/\j/\t in {
				v4/a/15,
				a/v3/10
			}{\path[draw, red, line width=0.8] (\i) edge[bend left=\t] (\j);}
		\end{tikzpicture}
	\end{minipage}
	\begin{minipage}{0.23\linewidth}
		\vspace{0pt}
		\centering
		\begin{tikzpicture}[scale=0.8]
			\filldraw[black](0,0) circle (3pt)node[label=left:$v_1$](v1){};
			\filldraw[black](2,0) circle (3pt)node[label=right:$v_2$](v2){};
			\filldraw[black](0,-2) circle (3pt)node[label=left:$v_3$](v3){};
			\filldraw[black](2,-2) circle (3pt)node[label=right:$v_4$](v4){};
			
			\filldraw[red](3,-1) circle (3pt)node[label=above:$a$](a){};
			
			\foreach \i/\j/\t in {
				v2/v3/0,
				v3/v4/0,
				v4/v1/0,
				v1/v3/15
			}{\path[draw, line width=0.8] (\i) edge[bend left=\t] (\j);}
			
			\foreach \i/\j/\t in {
				a/v2/15,
				v3/a/10
			}{\path[draw, red, line width=0.8] (\i) edge[bend left=\t] (\j);}
		\end{tikzpicture}
	\end{minipage}
\end{figure}

\noindent\rule{0.5\linewidth}{0.4pt}

If $v_3a\notin D,av_3\in D$:

\begin{figure}[H]
	\begin{minipage}[t]{0.23\linewidth}
		\vspace{0pt}
		\centering
		\begin{tikzpicture}[scale=0.8]
			\filldraw[black](0,0) circle (3pt)node[label=left:$v_1$](v1){};
			\filldraw[black](2,0) circle (3pt)node[label=right:$v_2$](v2){};
			\filldraw[black](0,-2) circle (3pt)node[label=left:$v_3$](v3){};
			\filldraw[black](2,-2) circle (3pt)node[label=right:$v_4$](v4){};
			
			\filldraw[red](3,-1) circle (3pt)node[label=above:$a$](a){};
			
			\foreach \i/\j/\t in {
				v1/v2/0,
				v2/v3/0,
				v3/v4/0,
				v4/v1/0,
				v1/v3/15,
				v3/v1/15,
				v2/v4/0
			}{\path[draw, line width=0.8] (\i) edge[bend left=\t] (\j);}
			
			\foreach \i/\j/\t in {
				a/v2/15,
				v2/a/15,
				v4/a/15,
				a/v3/0
			}{\path[draw, red, line width=0.8] (\i) edge[bend left=\t] (\j);}
			
			\path[draw, dashed, red, line width=0.8] (a) edge[bend left=15] (v4);
		\end{tikzpicture}
		\caption*{$iv$}
	\end{minipage}
	\hfill
	\begin{minipage}[t]{0.73\linewidth}
		\vspace{0pt}
		It has no strong arc decomposition no matter the existence of dashed arcs as $(iv)^*\times (iv)$ has no strong arc decomposition.
	\end{minipage}
\end{figure}

\noindent\rule{0.5\linewidth}{0.4pt}

If $v_3a\in D,av_3\notin D$:

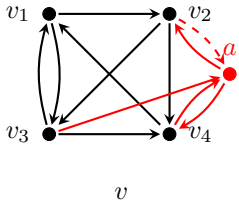
\begin{figure}[H]
	\begin{minipage}[t]{0.23\linewidth}
		\vspace{0pt}
		\centering
		\begin{tikzpicture}[scale=0.8]
			\filldraw[black](0,0) circle (3pt)node[label=left:$v_1$](v1){};
			\filldraw[black](2,0) circle (3pt)node[label=right:$v_2$](v2){};
			\filldraw[black](0,-2) circle (3pt)node[label=left:$v_3$](v3){};
			\filldraw[black](2,-2) circle (3pt)node[label=right:$v_4$](v4){};
			
			\filldraw[red](3,-1) circle (3pt)node[label=above:$a$](a){};
			
			\foreach \i/\j/\t in {
				v1/v2/0,
				v2/v3/0,
				v3/v4/0,
				v4/v1/0,
				v1/v3/15,
				v3/v1/15,
				v2/v4/0
			}{\path[draw, line width=0.8] (\i) edge[bend left=\t] (\j);}
			
			\foreach \i/\j/\t in {
				a/v2/15,
				a/v4/15,
				v4/a/15,
				v3/a/0
			}{\path[draw, red, line width=0.8] (\i) edge[bend left=\t] (\j);}
			
			\path[draw, dashed, red, line width=0.8] (v2) edge[bend left=15] (a);
		\end{tikzpicture}
		\caption*{$v$}
	\end{minipage}
	\hfill
	\begin{minipage}[t]{0.73\linewidth}
		\vspace{0pt}
		It has no strong arc decomposition no matter the existence of dashed arcs as $(iii)^*\times (v)$ has no strong arc decomposition.
	\end{minipage}
\end{figure}
We have now completed all cases where $D[V_2]$ is $S_{4,-1}$.

\begin{center}
	\rule{\linewidth}{2pt}
\end{center}

\subsection{$D[V_2]$ is $S_{4,-2}$.}

Similarly, as $S_4$ is a subdigraph of $\bar{G}_\text{new}$ (or $\bar{G}$), there are arcs $v_4a,av_2,v_3b,bv_1$ in $D$ where $a,b\in V_1$.

If $a=b$, then $a$ is adjacent to four vertices, which has been previously discussed. So we only consider the case $a\neq b$ in the following. Additionally, if there is another vertex $c\in V_1$ where $c\neq a,b$, then $D$ has a strong arc decomposition by splitting off $v_4a,av_2,v_3b,bv_1$ and applying the proof for the case where 
$D[V_2]$ is $S_4$. Therefore, we only need consider the case $V_1=\{a,b\}$.

Since we can split off either $v_4a,av_2$ or $v_3b,bv_1$ to obtain a graph similar to $S_{4,-1}$, if $D$ has no strong arc decomposition, then each of $a$ and $b$ must fall into one of the cases $(i), (ii), (iii), (iv)$ and $(v)$.

By reversing all arcs in cases $(i), (ii), (iii), (iv)$ and $(v)$, rotate 180 degrees clockwise, and relabeling, we obtain the corresponding reversed and rotated cases:  $(i)^*, (ii)^*, (iii)^*, (iv)^*$, and $(v)^*$ as described below.

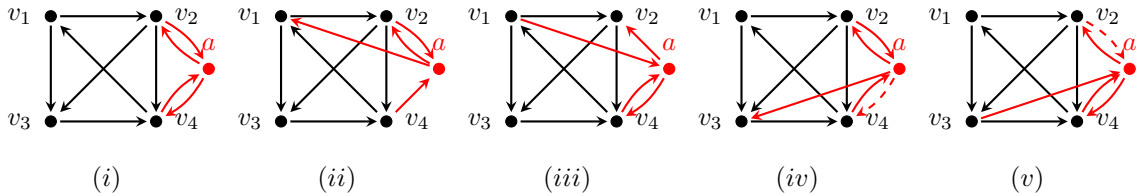
\begin{figure}[H]
	\begin{minipage}[t]{0.19\linewidth}
		\vspace{0pt}
		\centering
		\begin{tikzpicture}[scale=0.7]
			\filldraw[black](0,0) circle (3pt)node[label=left:$v_1$](v1){};
			\filldraw[black](2,0) circle (3pt)node[label=right:$v_2$](v2){};
			\filldraw[black](0,-2) circle (3pt)node[label=left:$v_3$](v3){};
			\filldraw[black](2,-2) circle (3pt)node[label=right:$v_4$](v4){};
			
			\filldraw[red](3,-1) circle (3pt)node[label=above:$a$](a){};
			
			\foreach \i/\j/\t in {
				v1/v2/0,
				v2/v3/0,
				v3/v4/0,
				v4/v1/0,
				v1/v3/0,
				v2/v4/0
			}{\path[draw, line width=0.8] (\i) edge[bend left=\t] (\j);}
			
			\foreach \i/\j/\t in {
				a/v2/15,
				a/v4/15,
				v4/a/15,
				v2/a/15
			}{\path[draw, red, line width=0.8] (\i) edge[bend left=\t] (\j);}
		\end{tikzpicture}
		\caption*{$(i)$}
	\end{minipage}
	\hfill
	\begin{minipage}[t]{0.19\linewidth}
		\vspace{0pt}
		\centering
		\begin{tikzpicture}[scale=0.7]
			\filldraw[black](0,0) circle (3pt)node[label=left:$v_1$](v1){};
			\filldraw[black](2,0) circle (3pt)node[label=right:$v_2$](v2){};
			\filldraw[black](0,-2) circle (3pt)node[label=left:$v_3$](v3){};
			\filldraw[black](2,-2) circle (3pt)node[label=right:$v_4$](v4){};
			
			\filldraw[red](3,-1) circle (3pt)node[label=above:$a$](a){};
			
			\foreach \i/\j/\t in {
				v1/v2/0,
				v2/v3/0,
				v3/v4/0,
				v4/v1/0,
				v1/v3/0,
				v2/v4/0
			}{\path[draw, line width=0.8] (\i) edge[bend left=\t] (\j);}
			
			\foreach \i/\j/\t in {
				a/v2/15,
				v2/a/15,
				v4/a/0,
				a/v1/0
			}{\path[draw, red, line width=0.8] (\i) edge[bend left=\t] (\j);}
		\end{tikzpicture}
		\caption*{$(ii)$}
	\end{minipage}
	\hfill
	\begin{minipage}[t]{0.19\linewidth}
		\vspace{0pt}
		\centering
		\begin{tikzpicture}[scale=0.7]
			\filldraw[black](0,0) circle (3pt)node[label=left:$v_1$](v1){};
			\filldraw[black](2,0) circle (3pt)node[label=right:$v_2$](v2){};
			\filldraw[black](0,-2) circle (3pt)node[label=left:$v_3$](v3){};
			\filldraw[black](2,-2) circle (3pt)node[label=right:$v_4$](v4){};
			
			\filldraw[red](3,-1) circle (3pt)node[label=above:$a$](a){};
			
			\foreach \i/\j/\t in {
				v1/v2/0,
				v2/v3/0,
				v3/v4/0,
				v4/v1/0,
				v1/v3/0,
				v2/v4/0
			}{\path[draw, line width=0.8] (\i) edge[bend left=\t] (\j);}
			
			\foreach \i/\j/\t in {
				a/v2/0,
				a/v4/15,
				v4/a/15,
				v1/a/0
			}{\path[draw, red, line width=0.8] (\i) edge[bend left=\t] (\j);}
		\end{tikzpicture}
		\caption*{$(iii)$}
	\end{minipage}
	\hfill
	\begin{minipage}[t]{0.19\linewidth}
		\vspace{0pt}
		\centering
		\begin{tikzpicture}[scale=0.7]
			\filldraw[black](0,0) circle (3pt)node[label=left:$v_1$](v1){};
			\filldraw[black](2,0) circle (3pt)node[label=right:$v_2$](v2){};
			\filldraw[black](0,-2) circle (3pt)node[label=left:$v_3$](v3){};
			\filldraw[black](2,-2) circle (3pt)node[label=right:$v_4$](v4){};
			
			\filldraw[red](3,-1) circle (3pt)node[label=above:$a$](a){};
			
			\foreach \i/\j/\t in {
				v1/v2/0,
				v2/v3/0,
				v3/v4/0,
				v4/v1/0,
				v1/v3/0,
				v2/v4/0
			}{\path[draw, line width=0.8] (\i) edge[bend left=\t] (\j);}
			
			\foreach \i/\j/\t in {
				a/v2/15,
				v2/a/15,
				v4/a/15,
				a/v3/0
			}{\path[draw, red, line width=0.8] (\i) edge[bend left=\t] (\j);}
			
			\path[draw, dashed, red, line width=0.8] (a) edge[bend left=15] (v4);
		\end{tikzpicture}
		\caption*{$(iv)$}
	\end{minipage}
	\hfill
	\begin{minipage}[t]{0.19\linewidth}
		\vspace{0pt}
		\centering
		\begin{tikzpicture}[scale=0.7]
			\filldraw[black](0,0) circle (3pt)node[label=left:$v_1$](v1){};
			\filldraw[black](2,0) circle (3pt)node[label=right:$v_2$](v2){};
			\filldraw[black](0,-2) circle (3pt)node[label=left:$v_3$](v3){};
			\filldraw[black](2,-2) circle (3pt)node[label=right:$v_4$](v4){};
			
			\filldraw[red](3,-1) circle (3pt)node[label=above:$a$](a){};
			
			\foreach \i/\j/\t in {
				v1/v2/0,
				v2/v3/0,
				v3/v4/0,
				v4/v1/0,
				v1/v3/0,
				v2/v4/0
			}{\path[draw, line width=0.8] (\i) edge[bend left=\t] (\j);}
			
			\foreach \i/\j/\t in {
				a/v2/15,
				a/v4/15,
				v4/a/15,
				v3/a/0
			}{\path[draw, red, line width=0.8] (\i) edge[bend left=\t] (\j);}
			
			\path[draw, dashed, red, line width=0.8] (v2) edge[bend left=15] (a);
		\end{tikzpicture}
		\caption*{$(v)$}
	\end{minipage}
\end{figure}

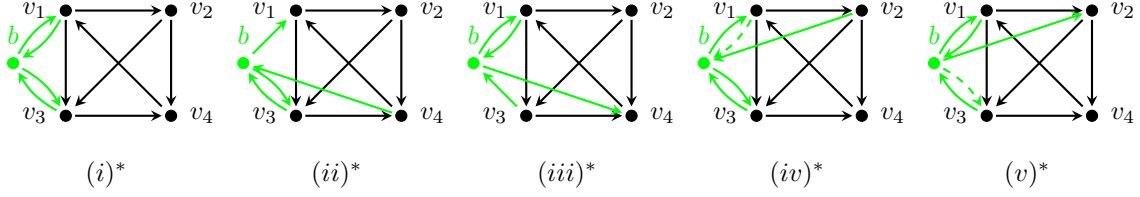
\begin{figure}[H]
	\begin{minipage}[t]{0.19\linewidth}
		\vspace{0pt}
		\centering
		\begin{tikzpicture}[scale=0.7]
			\filldraw[black](0,0) circle (3pt)node[label=left:$v_1$](v1){};
			\filldraw[black](2,0) circle (3pt)node[label=right:$v_2$](v2){};
			\filldraw[black](0,-2) circle (3pt)node[label=left:$v_3$](v3){};
			\filldraw[black](2,-2) circle (3pt)node[label=right:$v_4$](v4){};
			\filldraw[green](-1,-1) circle (3pt)node[label=above:$b$](b){};
			
			\foreach \i/\j/\t in {
				v1/v2/0,
				v2/v3/0,
				v3/v4/0,
				v4/v1/0,
				v1/v3/0,
				v2/v4/0
			}{\path[draw, line width=0.8] (\i) edge[bend left=\t] (\j);}
			
			\foreach \i/\j/\t in {
				b/v1/15,
				b/v3/15,
				v3/b/15,
				v1/b/15
			}{\path[draw, green, line width=0.8] (\i) edge[bend left=\t] (\j);}
		\end{tikzpicture}
		\caption*{$(i)^*$}
	\end{minipage}
	\hfill
	\begin{minipage}[t]{0.19\linewidth}
		\vspace{0pt}
		\centering
		\begin{tikzpicture}[scale=0.7]
			\filldraw[black](0,0) circle (3pt)node[label=left:$v_1$](v1){};
			\filldraw[black](2,0) circle (3pt)node[label=right:$v_2$](v2){};
			\filldraw[black](0,-2) circle (3pt)node[label=left:$v_3$](v3){};
			\filldraw[black](2,-2) circle (3pt)node[label=right:$v_4$](v4){};
			\filldraw[green](-1,-1) circle (3pt)node[label=above:$b$](b){};
			
			\foreach \i/\j/\t in {
				v1/v2/0,
				v2/v3/0,
				v3/v4/0,
				v4/v1/0,
				v1/v3/0,
				v2/v4/0
			}{\path[draw, line width=0.8] (\i) edge[bend left=\t] (\j);}
			
			\foreach \i/\j/\t in {
				b/v1/0,
				b/v3/15,
				v3/b/15,
				v4/b/0
			}{\path[draw, green, line width=0.8] (\i) edge[bend left=\t] (\j);}
		\end{tikzpicture}
		\caption*{$(ii)^*$}
	\end{minipage}
	\hfill
	\begin{minipage}[t]{0.19\linewidth}
		\vspace{0pt}
		\centering
		\begin{tikzpicture}[scale=0.7]
			\filldraw[black](0,0) circle (3pt)node[label=left:$v_1$](v1){};
			\filldraw[black](2,0) circle (3pt)node[label=right:$v_2$](v2){};
			\filldraw[black](0,-2) circle (3pt)node[label=left:$v_3$](v3){};
			\filldraw[black](2,-2) circle (3pt)node[label=right:$v_4$](v4){};
			\filldraw[green](-1,-1) circle (3pt)node[label=above:$b$](b){};
			
			\foreach \i/\j/\t in {
				v1/v2/0,
				v2/v3/0,
				v3/v4/0,
				v4/v1/0,
				v1/v3/0,
				v2/v4/0
			}{\path[draw, line width=0.8] (\i) edge[bend left=\t] (\j);}
			
			\foreach \i/\j/\t in {
				b/v1/15,
				b/v4/0,
				v1/b/15,
				v3/b/0
			}{\path[draw, green, line width=0.8] (\i) edge[bend left=\t] (\j);}
		\end{tikzpicture}
		\caption*{$(iii)^*$}
	\end{minipage}
	\hfill
	\begin{minipage}[t]{0.19\linewidth}
		\vspace{0pt}
		\centering
		\begin{tikzpicture}[scale=0.7]
			\filldraw[black](0,0) circle (3pt)node[label=left:$v_1$](v1){};
			\filldraw[black](2,0) circle (3pt)node[label=right:$v_2$](v2){};
			\filldraw[black](0,-2) circle (3pt)node[label=left:$v_3$](v3){};
			\filldraw[black](2,-2) circle (3pt)node[label=right:$v_4$](v4){};
			\filldraw[green](-1,-1) circle (3pt)node[label=above:$b$](b){};
			
			\foreach \i/\j/\t in {
				v1/v2/0,
				v2/v3/0,
				v3/v4/0,
				v4/v1/0,
				v1/v3/0,
				v2/v4/0
			}{\path[draw, line width=0.8] (\i) edge[bend left=\t] (\j);}
			
			\foreach \i/\j/\t in {
				b/v1/15,
				b/v3/15,
				v3/b/15,
				v2/b/0
			}{\path[draw, green, line width=0.8] (\i) edge[bend left=\t] (\j);}
			
			\path[draw, dashed, green, line width=0.8] (v1) edge[bend left=15] (b);
		\end{tikzpicture}
		\caption*{$(iv)^*$}
	\end{minipage}
	\hfill
	\begin{minipage}[t]{0.19\linewidth}
		\vspace{0pt}
		\centering
		\begin{tikzpicture}[scale=0.7]
			\filldraw[black](0,0) circle (3pt)node[label=left:$v_1$](v1){};
			\filldraw[black](2,0) circle (3pt)node[label=right:$v_2$](v2){};
			\filldraw[black](0,-2) circle (3pt)node[label=left:$v_3$](v3){};
			\filldraw[black](2,-2) circle (3pt)node[label=right:$v_4$](v4){};
			\filldraw[green](-1,-1) circle (3pt)node[label=above:$b$](b){};
			
			\foreach \i/\j/\t in {
				v1/v2/0,
				v2/v3/0,
				v3/v4/0,
				v4/v1/0,
				v1/v3/0,
				v2/v4/0
			}{\path[draw, line width=0.8] (\i) edge[bend left=\t] (\j);}
			
			\foreach \i/\j/\t in {
				b/v1/15,
				v1/b/15,
				v3/b/15,
				b/v2/0
			}{\path[draw, green, line width=0.8] (\i) edge[bend left=\t] (\j);}
			
			\path[draw, dashed, green, line width=0.8] (b) edge[bend left=15] (v3);
		\end{tikzpicture}
		\caption*{$(v)^*$}
	\end{minipage}
\end{figure}

In this way, we only need to discuss the different combinations of cases $(i), (ii), (iii), (iv)$ and $(v)$. Additionally, since $(e)^*\times (f)$ can be transformed into $(f)^*\times (e)$ by reversing arcs and relabeling, where $e$ and $f$ are elements of $\{i, ii, iii, iv, v\}$, we only need to examine 15 distinct graphs.

\begin{center}
	\rule{\linewidth}{2pt}
\end{center}

$(i)^*\times (i)$ has no strong arc decomposition as $(iv)^*\times (iv)$ has no strong arc decomposition.

\begin{center}
	\rule{\linewidth}{0.4pt}
\end{center}

$(i)^*\times (ii)$ has no strong arc decomposition as $(ii)^*\times (iv)$ has no strong arc decomposition.

\begin{center}
	\rule{\linewidth}{0.4pt}
\end{center}

$(i)^*\times (iii)$ has no strong arc decomposition as $(iii)^*\times (v)$ has no strong arc decomposition.

\begin{center}
	\rule{\linewidth}{0.4pt}
\end{center}

$(i)^*\times (iv)$ has no strong arc decomposition as $(iv)^*\times (iv)$ has no strong arc decomposition.

\begin{center}
	\rule{\linewidth}{0.4pt}
\end{center}

$(i)^*\times (v)$ has a strong arc decomposition regardless of the existence of the dashed arc. This is because the subdigraph $D[V_2]+\{v_3v_1,v_1v_3,v_3v_4,v_4v_2\}$ has a strong arc decomposition by Lemma~\ref{lem:parallel}.

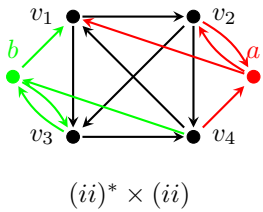
\begin{center}
	\rule{\linewidth}{0.4pt}
\end{center}

\begin{figure}[H]
	\begin{minipage}[t]{0.23\linewidth}
		\vspace{0pt}
		\centering
		\begin{tikzpicture}[scale=0.8]
			\filldraw[black](0,0) circle (3pt)node[label=left:$v_1$](v1){};
			\filldraw[black](2,0) circle (3pt)node[label=right:$v_2$](v2){};
			\filldraw[black](0,-2) circle (3pt)node[label=left:$v_3$](v3){};
			\filldraw[black](2,-2) circle (3pt)node[label=right:$v_4$](v4){};
			\filldraw[green](-1,-1) circle (3pt)node[label=above:$b$](b){};
			\filldraw[red](3,-1) circle (3pt)node[label=above:$a$](a){};
			
			\foreach \i/\j/\t in {
				v1/v2/0,
				v2/v3/0,
				v3/v4/0,
				v4/v1/0,
				v1/v3/0,
				v2/v4/0
			}{\path[draw, line width=0.8] (\i) edge[bend left=\t] (\j);}
			
			\foreach \i/\j/\t in {
				b/v1/0,
				b/v3/15,
				v3/b/15,
				v4/b/0
			}{\path[draw, green, line width=0.8] (\i) edge[bend left=\t] (\j);}
			
			\foreach \i/\j/\t in {
				a/v2/15,
				v2/a/15,
				v4/a/0,
				a/v1/0
			}{\path[draw, red, line width=0.8] (\i) edge[bend left=\t] (\j);}
		\end{tikzpicture}
		\caption*{$(ii)^*\times (ii)$}
	\end{minipage}
	\hfill
	\begin{minipage}[t]{0.73\linewidth}
		\vspace{0pt}
		It has no strong arc decomposition.
		
		Assume, for the sake of contradiction, that $D$ has a strong arc decomposition into $D_1$ and $D_2$. Given that $N^+(v_3)=2$, we may assume $v_3b\in D_1$ and $v_3v_4\in D_2$. Since $N^-(b)=2$, we have $v_4b\in D_2$. Since $N^-(v_4)=2$, we have $v_2v_4\in D_1$. Since $v_3v_4\in D_2$ and there are only two arcs from $\{v_1,v_3,b\}$ to $\{v_2,v_4,a\}$, we have $v_1v_2\in D_1$. Since $N^+(v_1)=2$ and $v_1v_2\in D_1$, we have $v_1v_3\in D_2$. Since $v_1v_2\in D_1$ and $N^-(v_2)=2$, we have $av_2\in D_2$. We have $v_4a\in D_2$ as there are only two arcs from $\{v_1,v_3,v_4,b\}$ to $\{v_2,a\}$, and so $v_2a,av_1\in D_1$. We can also conclude that $bv_3\in D_2, bv_1\in D_1$ as there are only two arcs from $\{v_3,b\}$ to $\{v_1,v_2,v_4,a\}$. $v_1v_1\in D_1$ as $v_4$ need out-arc. Now there is no in-arc of $v_1$ in $D_2$, a contradiction. So $D$ has no strong arc decomposition.
		
	\end{minipage}
\end{figure}

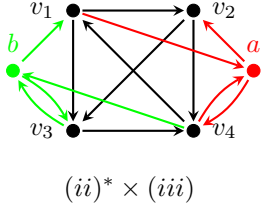
\begin{center}
	\rule{\linewidth}{0.4pt}
\end{center}

\begin{figure}[H]
	\begin{minipage}[t]{0.23\linewidth}
		\vspace{0pt}
		\centering
		\begin{tikzpicture}[scale=0.8]
			\filldraw[black](0,0) circle (3pt)node[label=left:$v_1$](v1){};
			\filldraw[black](2,0) circle (3pt)node[label=right:$v_2$](v2){};
			\filldraw[black](0,-2) circle (3pt)node[label=left:$v_3$](v3){};
			\filldraw[black](2,-2) circle (3pt)node[label=right:$v_4$](v4){};
			\filldraw[green](-1,-1) circle (3pt)node[label=above:$b$](b){};
			\filldraw[red](3,-1) circle (3pt)node[label=above:$a$](a){};
			
			\foreach \i/\j/\t in {
				v1/v2/0,
				v2/v3/0,
				v3/v4/0,
				v4/v1/0,
				v1/v3/0,
				v2/v4/0
			}{\path[draw, line width=0.8] (\i) edge[bend left=\t] (\j);}
			
			\foreach \i/\j/\t in {
				b/v1/0,
				b/v3/15,
				v3/b/15,
				v4/b/0
			}{\path[draw, green, line width=0.8] (\i) edge[bend left=\t] (\j);}
			
			\foreach \i/\j/\t in {
				a/v2/0,
				a/v4/15,
				v4/a/15,
				v1/a/0
			}{\path[draw, red, line width=0.8] (\i) edge[bend left=\t] (\j);}
		\end{tikzpicture}
		\caption*{$(ii)^*\times (iii)$}
	\end{minipage}
	\hfill
	\begin{minipage}[t]{0.73\linewidth}
		\vspace{0pt}
		It has no strong arc decomposition.
		
		Assume, for the sake of contradiction, that $D$ has a strong arc decomposition into $D_1$ and $D_2$. Given that  $N^+(v_3)=2$, we may assume $v_3b\in D_1, v_3v_4\in D_2$. we also have $bv_1\in D_1,v_4b,bv_3\in D_2$ as there are only two arcs from $\{v_3,b\}$ to $\{v_1,v_2,v_4,a\}$, and we have that $v_4v_1\in D_2$ since $N^-(v_1)=2$. Since $N^+(v_4)=3$, we have $v_4a\in D_1$, then $av_2\in D_1, v_1a,av_4\in D_2$, and $v_1v_2\in D_2$ as $N^-(v_2)=2$.
		
		Now, there is no arc from $\{v_1,v_3,b\}$ to $\{v_2,v_4,a\}$ in $D_1$, which leads to a contradiction.  So, $D$ has no strong arc decomposition.
	\end{minipage}
\end{figure}

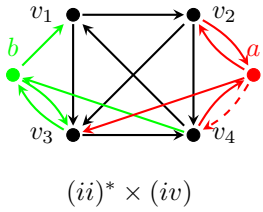
\begin{center}
	\rule{\linewidth}{0.4pt}
\end{center}

\begin{figure}[H]
	\begin{minipage}[t]{0.23\linewidth}
		\vspace{0pt}
		\centering
		\begin{tikzpicture}[scale=0.8]
			\filldraw[black](0,0) circle (3pt)node[label=left:$v_1$](v1){};
			\filldraw[black](2,0) circle (3pt)node[label=right:$v_2$](v2){};
			\filldraw[black](0,-2) circle (3pt)node[label=left:$v_3$](v3){};
			\filldraw[black](2,-2) circle (3pt)node[label=right:$v_4$](v4){};
			\filldraw[green](-1,-1) circle (3pt)node[label=above:$b$](b){};
			\filldraw[red](3,-1) circle (3pt)node[label=above:$a$](a){};
			
			\foreach \i/\j/\t in {
				v1/v2/0,
				v2/v3/0,
				v3/v4/0,
				v4/v1/0,
				v1/v3/0,
				v2/v4/0
			}{\path[draw, line width=0.8] (\i) edge[bend left=\t] (\j);}
			
			\foreach \i/\j/\t in {
				b/v1/0,
				b/v3/15,
				v3/b/15,
				v4/b/0
			}{\path[draw, green, line width=0.8] (\i) edge[bend left=\t] (\j);}
			
			\foreach \i/\j/\t in {
				a/v2/15,
				v2/a/15,
				v4/a/15,
				a/v3/0
			}{\path[draw, red, line width=0.8] (\i) edge[bend left=\t] (\j);}
			
			\path[draw, dashed, red, line width=0.8] (a) edge[bend left=15] (v4);
		\end{tikzpicture}
		\caption*{$(ii)^*\times (iv)$}
	\end{minipage}
	\hfill
	\begin{minipage}[t]{0.73\linewidth}
		\vspace{0pt}
		It has no strong arc decomposition regardless of the existence of dashed arcs.
		
	Assume, for the sake of contradiction, that $D$ has a strong arc decomposition into $D_1$ and $D_2$. Given that $N_D^+(v_3)=2$, we may assume $v_3b\in D_1, v_3v_4\in D_2$. Besides, $bv_1\in D_1,v_4b,bv_3\in D_2$, $v_4v_1\in D_2$ as $N^-(v_1)=2$, $v_1v_2\in D_1$ as $v_3v_4\in D_2$ and they are from $\{v_1,v_3,b\}$ to $\{v_2,v_4,a\}$, $v_4a\in D_1$ as $N^+(v_4)=3$. 
		
		Now there is no in-arc of $\{v_2,a\}$ in $D_2$, a contradiction. So, $D$ has no strong arc decomposition. 	
	\end{minipage}
\end{figure}

\begin{center}
	\rule{\linewidth}{0.4pt}
\end{center}

$(ii)^*\times (v)$ has a strong arc decomposition regardless of the existence of the dashed arc as $D[V_2]+\{v_3v_2,v_4v_1\}$ has a strong arc decomposition by Lemma~\ref{lem:parallel}.

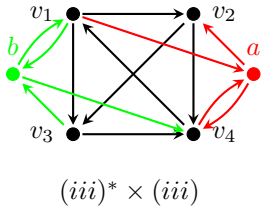
\begin{center}
	\rule{\linewidth}{0.4pt}
\end{center}

\begin{figure}[H]
	\begin{minipage}[t]{0.23\linewidth}
		\vspace{0pt}
		\centering
		\begin{tikzpicture}[scale=0.8]
			\filldraw[black](0,0) circle (3pt)node[label=left:$v_1$](v1){};
			\filldraw[black](2,0) circle (3pt)node[label=right:$v_2$](v2){};
			\filldraw[black](0,-2) circle (3pt)node[label=left:$v_3$](v3){};
			\filldraw[black](2,-2) circle (3pt)node[label=right:$v_4$](v4){};
			\filldraw[green](-1,-1) circle (3pt)node[label=above:$b$](b){};
			\filldraw[red](3,-1) circle (3pt)node[label=above:$a$](a){};
			
			\foreach \i/\j/\t in {
				v1/v2/0,
				v2/v3/0,
				v3/v4/0,
				v4/v1/0,
				v1/v3/0,
				v2/v4/0
			}{\path[draw, line width=0.8] (\i) edge[bend left=\t] (\j);}
			
			\foreach \i/\j/\t in {
				b/v1/15,
				v1/b/15,
				v3/b/0,
				b/v4/0
			}{\path[draw, green, line width=0.8] (\i) edge[bend left=\t] (\j);}
			
			\foreach \i/\j/\t in {
				a/v2/0,
				a/v4/15,
				v4/a/15,
				v1/a/0
			}{\path[draw, red, line width=0.8] (\i) edge[bend left=\t] (\j);}
		\end{tikzpicture}
		\caption*{$(iii)^*\times (iii)$}
	\end{minipage}
	\hfill
	\begin{minipage}[t]{0.73\linewidth}
		\vspace{0pt}
		It has no strong arc decomposition.
		
		Assume, for the sake of contradiction, that $D$ has a strong arc decomposition into $D_1$ and $D_2$. Given that $N_D^+(v_3)=2$, we may assume $v_3b\in D_1, v_3v_4\in D_2$.
	By the in-degree of $v_1$, out-degree of $v_4$, we have to divide $b$ into $v_3bv_1,v_1bv_4$, $a$ into $v_4av_2,v_1av_4$.
		
		Besides, $bv_1\in D_1,v_1b,bv_4\in D_2$, $v_4v_1\in D_2$ as $N^-(v_1)=2$, $v_4a,av_2\in D_1, v_1a,av_4\in D_2$ as $N^+(v_4)=2$, $v_1v_2 \in D_2$ as $N^-(v_2)=2$. Then there is no arc from $\{v_1,v_3,b\}$ to $\{v_2,v_4,a\}$ in $D_1$, a contradiction. So, $D$ has no strong arc decomposition.
	\end{minipage}
\end{figure}

\begin{center}
	\rule{\linewidth}{0.4pt}
\end{center}

\begin{figure}[H]
	\begin{minipage}[t]{0.23\linewidth}
		\vspace{0pt}
		\centering
		\begin{tikzpicture}[scale=0.8]
			\filldraw[black](0,0) circle (3pt)node[label=left:$v_1$](v1){};
			\filldraw[black](2,0) circle (3pt)node[label=right:$v_2$](v2){};
			\filldraw[black](0,-2) circle (3pt)node[label=left:$v_3$](v3){};
			\filldraw[black](2,-2) circle (3pt)node[label=right:$v_4$](v4){};
			\filldraw[green](-1,-1) circle (3pt)node[label=above:$b$](b){};
			\filldraw[red](3,-1) circle (3pt)node[label=above:$a$](a){};
			
			\foreach \i/\j/\t in {
				v1/v2/0,
				v2/v3/0,
				v3/v4/0,
				v4/v1/0,
				v1/v3/0,
				v2/v4/0
			}{\path[draw, line width=0.8] (\i) edge[bend left=\t] (\j);}
			
			\foreach \i/\j/\t in {
				b/v1/15,
				v1/b/15,
				v3/b/0,
				b/v4/0
			}{\path[draw, green, line width=0.8] (\i) edge[bend left=\t] (\j);}
			
			\foreach \i/\j/\t in {
				a/v2/15,
				v2/a/15,
				v4/a/15,
				a/v3/0
			}{\path[draw, red, line width=0.8] (\i) edge[bend left=\t] (\j);}
			
			\path[draw, dashed, red, line width=0.8] (a) edge[bend left=15] (v4);
		\end{tikzpicture}
		\caption*{$(iii)^*\times (iv)$}
	\end{minipage}
	\hfill
	\begin{minipage}[t]{0.73\linewidth}
		\vspace{0pt}
		It has no strong arc decomposition regardless of whether it has dashed arcs.
		
		Assume, for the sake of contradiction, that $D$ has a strong arc decomposition into $D_1$ and $D_2$. Given that $N^-(v_1)=2$, we may assume $v_3b,bv_1\in D_1,v_1b,bv_4,v_4v_1\in D_2$. Besides, $v_3v_4\in D_2$ as $N^+(v_3)=2$, $v_1v_2\in D_1$ as there are only two arcs from $\{v_1,v_3,b\}$ to $\{v_2,v_4,a\}$, $v_4a\in D_1$ as $N^+(v_4)=2$.
		
		Now, there is no in-arc of $\{a,v_2\}$ in $D_2$, which means $D$ has no strong arc decomposition.
	\end{minipage}
\end{figure}

\begin{center}
	\rule{\linewidth}{0.4pt}
\end{center}

\begin{figure}[H]
	\begin{minipage}[t]{0.23\linewidth}
		\vspace{0pt}
		\centering
		\begin{tikzpicture}[scale=0.8]
			\filldraw[black](0,0) circle (3pt)node[label=left:$v_1$](v1){};
			\filldraw[black](2,0) circle (3pt)node[label=right:$v_2$](v2){};
			\filldraw[black](0,-2) circle (3pt)node[label=left:$v_3$](v3){};
			\filldraw[black](2,-2) circle (3pt)node[label=right:$v_4$](v4){};
			\filldraw[green](-1,-1) circle (3pt)node[label=above:$b$](b){};
			\filldraw[red](3,-1) circle (3pt)node[label=above:$a$](a){};
			
			\foreach \i/\j/\t in {
				v1/v2/0,
				v2/v3/0,
				v3/v4/0,
				v4/v1/0,
				v1/v3/0,
				v2/v4/0
			}{\path[draw, line width=0.8] (\i) edge[bend left=\t] (\j);}
			
			\foreach \i/\j/\t in {
				b/v1/15,
				v1/b/15,
				v3/b/0,
				b/v4/0
			}{\path[draw, green, line width=0.8] (\i) edge[bend left=\t] (\j);}
			
			\foreach \i/\j/\t in {
				a/v2/15,
				a/v4/15,
				v4/a/15,
				v3/a/0
			}{\path[draw, red, line width=0.8] (\i) edge[bend left=\t] (\j);}
			
			\path[draw, dashed, red, line width=0.8] (v2) edge[bend left=15] (a);
		\end{tikzpicture}
		\caption*{$(iii)^*\times (v)$}
	\end{minipage}
	\hfill
	\begin{minipage}[t]{0.73\linewidth}
		\vspace{0pt}
		It has no strong arc decomposition no matter if it has dashed arcs.
		
		Assume, for the sake of contradiction, that $D$ has a strong arc decomposition into $D_1$ and $D_2$. Given that $N^-(v_1)=2$, we may assume $v_3b,bv_1\in D_1,v_1b,bv_4,v_4v_1\in D_2$. Besides, $v_2v_3\in D_1$ as there are only two arcs from $\{v_2,v_4,a\}$ to $\{v_1,v_3,b\}$, $v_1v_3\in D_2$ as $N^-(v_3)=2$, $v_1v_2\in D_1$ as $N^+(v_1)=3$, $av_2\in D_2$ as $N^-(v_2)=2$, $av_4\in D_1$ as $N^+(a)=2$.
		
		Now, there is no out-arc of $\{a,v_4\}$ in $D_1$, which means $D$ has no strong arc decomposition.
	\end{minipage}
\end{figure}

\begin{center}
	\rule{\linewidth}{0.4pt}
\end{center}

\begin{figure}[H]
	\begin{minipage}[t]{0.23\linewidth}
		\vspace{0pt}
		\centering
		\begin{tikzpicture}[scale=0.8]
			\filldraw[black](0,0) circle (3pt)node[label=left:$v_1$](v1){};
			\filldraw[black](2,0) circle (3pt)node[label=right:$v_2$](v2){};
			\filldraw[black](0,-2) circle (3pt)node[label=left:$v_3$](v3){};
			\filldraw[black](2,-2) circle (3pt)node[label=right:$v_4$](v4){};
			\filldraw[green](-1,-1) circle (3pt)node[label=above:$b$](b){};
			\filldraw[red](3,-1) circle (3pt)node[label=above:$a$](a){};
			
			\foreach \i/\j/\t in {
				v1/v2/0,
				v2/v3/0,
				v3/v4/0,
				v4/v1/0,
				v1/v3/0,
				v2/v4/0
			}{\path[draw, line width=0.8] (\i) edge[bend left=\t] (\j);}
			
			\foreach \i/\j/\t in {
				b/v1/15,
				b/v3/15,
				v3/b/15,
				v2/b/0
			}{\path[draw, green, line width=0.8] (\i) edge[bend left=\t] (\j);}
			
			\path[draw, dashed, green, line width=0.8] (v1) edge[bend left=15] (b);
			
			\foreach \i/\j/\t in {
				a/v2/15,
				v2/a/15,
				v4/a/15,
				a/v3/0
			}{\path[draw, red, line width=0.8] (\i) edge[bend left=\t] (\j);}
			
			\path[draw, dashed, red, line width=0.8] (a) edge[bend left=15] (v4);
		\end{tikzpicture}
		\caption*{$(iv)^*\times (iv)$}
	\end{minipage}
	\hfill
	\begin{minipage}[t]{0.73\linewidth}
		\vspace{0pt}
		It has no strong arc decomposition no matter the existence of dashed arcs.
		
		Assume, for the sake of contradiction, that $D$ has a strong arc decomposition into $D_1$ and $D_2$. Given that $N_D^+(v_3)=2$, we may assume $v_3b\in D_1, v_3v_4\in D_2$. Besides, $bv_1\in D_1, bv_3\in D_2$ as $N_D^+(b)=2$ and $v_3$ has no out-neighbor besides $b$ in $D_1$, $v_1v_2\in D_1$ as there are only two arcs from $\{v_1,v_3,b\}$ to $\{v_2,v_4,a\}$, $v_4v_1\in D_2$ as $N_D^-(v_1)=2$, $v_4a\in D_1$ as $N_D^+(v_4)=2$.

        Now, there is no in-arc of $\{a,v_2\}$ in $D_2$, which means $D$ has no strong arc decomposition.
	\end{minipage}
\end{figure}

\begin{center}
	\rule{\linewidth}{0.4pt}
\end{center}

$(iv)^*\times (v)$ has a strong arc decomposition regardless of the existence of the dashed arc as $D[V_2]+\{v_3v_1,v_2v_3,v_4v_2,v_3v_4\}$ has a strong arc decomposition by Lemma~\ref{lem:parallel}.

\begin{center}
	\rule{\linewidth}{0.4pt}
\end{center}

$(v)^*\times (v)$ has a strong arc decomposition regardless of the existence of the dashed arc as $D[V_2]+\{v_3v_1,v_1v_2,v_4v_2,v_3v_4\}$ has a strong arc decomposition by Lemma~\ref{lem:parallel}.


\begin{thebibliography}{1}
\bibitem{aiDM346}
J.~Ai, S.~Gerke, G.~Gutin, A.~Yeo, and Y.~Zhou.
\newblock Results on the small quasi-kernel conjecture.
\bibitem{BG09}
J. Bang-Jensen, G. Gutin, Digraphs-Theory, Algorithms and Applications, 2nd Ed., Springer Monographs in Mathematics, London: Springer-Verlag London, Ltd., 2009.

\bibitem{bangJGT95}
J. Bang-Jensen, G. Gutin, A. Yeo, Arc-disjoint strong spanning subdigraphs of semicomplete compositions. J. Graph Theory, 95(2):267--289, 2020.

\bibitem{bangJGT102}
J. Bang-Jensen, F. Havet, A. Yeo, Spanning eulerian subdigraphs in semicomplete digraphs, J. Graph Theory, 102(3):527--606, 2023.

\bibitem{bangJCTB102}
J. Bang-Jensen, J. Huang, Decomposing locally semicomplete digraphs into strong spanning subdigraphs, J. Combin. Theory Ser. B, 102:701--714, 2010.

\bibitem{bang2024}
J. Bang-Jensen, Y. Wang, Strong arc decomposition of split digraphs, 	arXiv:2309.06904.

\bibitem{bangCOM24}
J. Bang-Jensen, A. Yeo, Decomposition $k$-arc-strong tournaments into strong spanning subdigraphs, Combinatorica, 24(3):331--349, 2004.

\bibitem{hellDAM216}
P.~Hell and C.~Hern{\'{a}}ndez{-}Cruz.
\newblock Strict chordal and strict split digraphs.
\newblock {\em Discret. Appl. Math.}, 216:609--617, 2017.

\bibitem{lamarDM312}
M.~D. Lamar.
\newblock Split digraphs.
\newblock {\em Discret. Math.}, 312(7):1314--1325, 2012.

\bibitem{sunDM342}
Y.~Sun, G.~Gutin, and J.~Ai.
\newblock Arc-disjoint strong spanning subdigraphs in compositions and products
  of digraphs.
\newblock {\em Discret. Math.}, 342(8):2297--2305, 2019.

\bibitem{thomassen1984}
C.~Thomassen.
\newblock {Connectivity in tournaments}.
\newblock In {\em {Graph theory and combinatorics (Cambridge, 1983)}}, pages
  305--313. Academic Press, 1984.

\bibitem{thomassen1989}
C.~Thomassen.
\newblock {Configurations in graphs of large minimum degree, connectivity, or
  chromatic number}.
\newblock {\em Annals of the New York Academy of Sciences}, 555:402--412, 1989.
\end{thebibliography}
\end{document}